%% file: conjtest.tex
\documentclass[onefignum,onetabnum]{siamonline220329}

\usepackage{tikz-cd}
\usepackage{diagbox}
\usepackage{makecell}
\usepackage{hyperref}
\usepackage{faktor}
\usepackage{bbm}
\usepackage{enumitem}

\input{conjtest_shared}

\usepackage{algorithm}

\ifpdf
\hypersetup{
  pdftitle={Testing topological conjugacy of time series},
  pdfauthor={P. D\L{}otko, M. Lipi\'nski, and J. Signerska-Rynkowska}
}
\fi

\newcommand\abs[1]{\lvert#1\rvert}

\newcommand\modone[1]{{#1}_\mathbbm{1}}
\newcommand\bmodone[1]{(#1)_\mathbbm{1}}
\newcommand\diam[1]{\textup{diam}(#1)}

\newcommand\proj[2]{#1^{(#2)}}
\newcommand{\cA}{\mathcal{A}}
\newcommand{\cB}{\mathcal{B}}
\newcommand{\cL}{\mathcal{L}}
\newcommand{\cK}{\mathcal{K}}
\newcommand{\cP}{\mathcal{P}}
\newcommand{\cR}{\mathcal{R}}
\newcommand{\cS}{\mathcal{S}}
\newcommand{\cT}{\mathcal{T}}

\newcommand{\bbA}{\mathbb{A}}
\newcommand{\KK}{\mathbb{K}}
\newcommand{\NN}{\mathbb{N}}
\newcommand{\RR}{\mathbb{R}}
\newcommand{\sph}{\mathbb{S}}
\newcommand{\TT}{\mathbb{T}}
\newcommand{\ZZ}{\mathbb{Z}}

\newcommand{\ttA}{\mathtt{A}}
\newcommand{\ttH}{\mathtt{H}}
\newcommand{\tth}{\mathtt{h}}
\newcommand{\ttV}{\mathtt{V}}
\newcommand{\tts}{\mathtt{s}}
\newcommand{\ttp}{\mathtt{p}}

\newcommand\smetr{\mathbf{d}_\sph}

\DeclareMathOperator{\conjtest}{ConjTest}
\DeclareMathOperator{\haus}{d_{H}}

\DeclareMathOperator{\err}{err}
\DeclareMathOperator{\fnn}{FNN}
\DeclareMathOperator{\KNN}{KNN}
\DeclareMathOperator{\knn}{\kappa}
\DeclareMathOperator{\knno}{\overline{\kappa}}
\DeclareMathOperator{\emb}{\Pi}

\DeclareMathOperator{\id}{id}
\DeclareMathOperator{\inter}{int}
\DeclareMathOperator{\im}{im}

\DeclareMathOperator{\hdiff}{\tt diff}
\DeclareMathOperator{\hscore}{\tt score}

\begin{document}

\maketitle

% REQUIRED
\begin{abstract}
This paper considers a problem of testing, from a finite sample, a topological conjugacy of two trajectories coming from dynamical systems $(X,f)$ and $(Y,g)$. More precisely, given $x_1,\ldots, x_n \subset X$ and $y_1,\ldots,y_n \subset Y$ such that $x_{i+1} = f(x_i)$ and $y_{i+1} = g(y_i)$ as well as $h: X \rightarrow Y$, we deliver a number of tests to check if $f$ and $g$ are topologically conjugated via $h$. The values of the tests are close to zero for systems conjugate by $h$ and large for systems that are not. 
Convergence of the test values, in case when sample size goes to infinity, is established. 
We provide a number of numerical examples indicating scalability and robustness of the presented methods. 
In addition, we show how the presented method gives rise to a test of sufficient embedding dimension, mentioned in Takens' embedding theorem.
Our methods also apply to the situation when we are given two observables of deterministic processes, of a form of one or higher dimensional time-series. In this case, their similarity can be accessed by comparing the dynamics of their Takens' reconstructions. 
Finally, we include a proof-of-concept study using the presented methods to search for an approximation of the homeomorphism conjugating given systems.
\end{abstract}

% REQUIRED

% REQUIRED

\section{Introduction}
Understanding sampled dynamics is of primal importance in multiple branches of science where there is a lack of solid theoretical models of the underlying phenomena \cite{doi:10.1073/pnas.1906995116,DeAngelis:2015te, Lejarza:2022vl, PhysRevLett.82.1144,Yuan:2019vx}. 
It delivers a foundation for various equation--free models of observed dynamics and allows to draw conclusions about the unknown observed processes. In the considered case we start with two, potentially different, phase spaces $X$ and $Y$ and a map $h : X \rightarrow Y$. Given two sampled trajectories, referred to in this paper by \emph{time series}, $x_1,\ldots,x_n \subset X$ and $y_1,\ldots,y_n \subset Y$ we assume that they are both generated by a continuous maps $f : X \rightarrow X$ and  $g : Y \rightarrow Y$\footnote{For finite samples those maps always exist assuming that $x_i = x_j$ and $y_i = y_j$ if and only if $i = j$.} in a way that $x_{i+1} = f(x_i)$ and $y_{i+1} = g(y_i)$.
In what follows, we build a number of tests that allow to distinguish trajectories that are conjugated by the given map $h$ from those that are not. 
 It should be noted that the problem of finding an appropriate $h$ for two conjugated dynamical system is in general very difficult and goes beyond the scope of this paper.
 However, in Section \ref{sec:in_a_search_for_h} and in Appendix \ref{apx:approx_h} we propose and validate a method of approximating map $h$ for one dimensional systems $f$ and $g$ utilizing one of the proposed statistics.
 
 The presented problem is practically important for the following reasons.
 Firstly, the proposed machinery allows to test for conjugacy, in case when the formulas that generate the underlying dynamics, as $f$ and $g$ above, are not known explicitly, and the input data are based on observations of the considered system. 

Secondly, some of the presented methods apply in the case when the dynamics $f$ and $g$ on $X$ and $Y$ is explicitly known, but we want to test if a given map $h: X \rightarrow Y$ between the phase spaces has a potential to be a topological conjugacy. It is important as the theoretical results on conjugacy are given only for a handful of systems and our methods give a tool for numerical hypothesis testing.

Thirdly, those methods can be used to estimate the optimal parameters of the dynamics reconstruction. A basic way to achieve such a reconstruction is via time delay embedding, a technique that depends on parameters including the \emph{embedding dimension} and the \emph{time lag} (or \emph{delay}). 
When the parameters of the method are appropriately set up and the assumptions of Takens' Embedding Theorem hold (see \cite{TakensOriginal,Broer2011}), then (generically) a reconstruction is obtained, meaning that for \emph{generic} dynamical system and \emph{generic} observable, the delay-coordinate map produces a \emph{conjugacy} (dynamical equivalence) between reconstructed dynamics and the original (unknown) dynamics (cut to the limit set of a given trajectory)\footnote{One should, though, be aware that the \emph{generic set} in the classical Takens' Embedding Theorem might be a set of a small measure. However, recent advances in probabilistic versions of Takens' Theorem (\cite{Baranski_2020}) assert that, under even milder assumptions, the delay-coordinate map provides injective (not necessary conjugacy) correspondence between the points of the original system in the subset of full measure and the points in the reconstructed space.}. However, without the prior knowledge of the underlying dynamics (e.g. dimensions of the attractor), the values of those parameters have to be determined experimentally from the data.
It is typically achieved by implicitly testing for a conjugacy of the time delay embeddings to spaces of constitutive dimensions. 
Specifically, it is assumed that the optimal dimension of reconstruction $d$ is achieved when there is no conjugacy of the reconstruction in dimension $d$ to the reconstruction in the dimension $d'$, where $d' < d$, while there is a conjugacy between reconstruction in dimension $d$ and reconstruction in dimension $d''$, where $d < d''$. 
Those conditions can be tested with methods presented in this paper.

The main contributions of this paper include:
\begin{itemize}
    \item We propose a generalization of the FNN (\emph{False Nearest Neighbor})  method~\cite{HeggerKantzFNN1999} so that it can be applied to test for topological conjugacy of time series\footnote{Classical FNN method was used only to estimate the embedding dimension in a dynamics reconstruction using time delay embedding.}.  
    Moreover, we present its further modification called $\KNN$ method.
    \item We propose two entirely new methods: $\conjtest$ and $\conjtest^+$.
        Instead of providing an almost binary answer to a question if two sampled dynamical systems are conjugate (which happens for the generalized FNN and the $\KNN$ method), their result is a continuous variable that can serve as a scale of similarity of two dynamics. 
        This property makes the two new methods appropriate for noisy data. 
    \item We present a number of benchmark experiments to test the presented methods. In particular we  analyze how different methods are robust for the type of testing (e.g. noise, determinism, alignment of a time series).
    \item Additionally, in one dimensional setting, 
    we propose a heuristic method for approximating the possible conjugating homeomorphism between two dynamical systems given by time series.
\end{itemize}

To the best of our knowledge there are no ``explicit'' methods available to test conjugacy of dynamical systems given by their finite sample in a form of time series as proposed in this paper. A number of methods exist to estimate the parameters of a time delay embedding. They include, among others, mutual information \cite{mutual}, autocorrelation and higher order correlations \cite{correlation}, a curvature-based approach  \cite{Bradley} or wavering product \cite{buzug} for selecting the time-lag, selecting of embedding dimension based on GP algorithm \cite{gp} or the above mentioned FNN algorithm, as well as some methods allowing to  choose the embedding dimension and the time lag simultaneously as, for example, C-C method based on correlation integral \cite{kim}, methods based on symbolic analysis and entropy \cite{garcia} or some rigorous statistical tests \cite{pecora}. However, the problem of topological conjugacy between the maps generating two given time series and finding the connecting homeomorphism which conjugates the two dynamical systems, due to its complexity, has been mainly approached using machine learning tools (see e.g. the recent work \cite{Bramburger} which for the unknown map $f$ and given time series generated by $f,$ employed deep neural network for discovering the simple map $g$ which could model the unknown dynamics $f$ together with the map $h$ conjugating $f$ and $g$). 
Some theoretical ideas on finding conjugating homeomorphism (or, in general, a \emph{commuter} between two maps)  are discussed later in Section~\ref{sec:in_a_search_for_h} together with related works.

Numerous methods providing some similarity measures between time series exist (see reviews \cite{KIANIMAJD201711005}). However, we claim that those classical methods are not suitable for the problem we tackle in this paper.
While those methods often look for an actual similarity of signals or correlation, we are more interested in the dynamical generators hiding behind the data. For instance, two time series sampled from the same chaotic system can be highly uncorrelated, yet we would like to recognize them as similar, because the dynamical system constituting them is the same.     Moreover, methods introduced in this work are applicable for time series embedded in any metric space, while most of the methods are restricted to $\RR$, some of them are still useful in $\RR^d$.

%
%What is the structure of the paper
%

The paper consists of four parts: Section \ref{sec:preliminaries} introduces the basic concepts behind the proposed methods.
Section~\ref{sec:conjugacy_tests} presents four methods designed for data-driven evaluation of conjugacy of two dynamical systems.
Section~\ref{sec:experiments} explores the features of the proposed methods using a number of numerical experiments. 
Section \ref{sec:in_a_search_for_h} develops the method of estimating the possible conjugacy map $h: X\to Y$ for time series generated from dynamical systems $(X,f)$ and $(Y,g)$ in the case when  the phase spaces $X$ and $Y$ are intervals in $\mathbb{R}$.
Additional details of that procedure and proofs are contained in \ref{apx:approx_h}.
Lastly, in Section~\ref{sec:disscussion} we summarize most important observations and discuss their possible significance in real-world time series analysis.

Finally, it should be noted that in the continuous setting, topological conjugacy is very fragile; it may be destroyed by an infinitesimal change of parameters of the system once that causes bifurcation. However, two finite sample of the trajectories obtained from the system before and after bifurcation are very close and it would require much large change of parameters to detect problems with conjugacy. It is a consequence of the fact that the techniques proposed in this paper operates on finite data. Therefore, they can provide evidences that the proposed connecting homeomorphism is not a topological conjugacy of the two considered systems, but they will not allow to prove, in any rigorous sense, the conjugacy between them.

%%%%%%%%%%%%%%%%%%%%%%%%%%%%%%%%%%%%%%%%%%%%%%%%%%%%%%%%%%%%%%%%%%%%%%%%%%%%%%%%%%%%%%%%%%%%%%%%%%%
%%%%%%%%%%%%%%%%%%%%%%%%%%%%%%%%%%%%%%%%%%%%%%%%%%%%%%%%%%%%%%%%%%%%%%%%%%%%%%%%%%%%%%%%%%%%%%%%%%%
%%%%%%%%%%%%%%%%%%%%%%%%%%%%%%%%%%%%%%%%%%%%%%%%%%%%%%%%%%%%%%%%%%%%%%%%%%%%%%%%%%%%%%%%%%%%%%%%%%%
\section{Preliminaries}\label{sec:preliminaries}
\subsection{Topological conjugacy}
We start with a pair of metric spaces $X$ and $Y$ and a pair of dynamical systems: $\varphi:X\times\TT\rightarrow X$ and $\psi:Y\times\TT\rightarrow Y$, where $\TT\in\{\ZZ,\RR\}$.
Fixing $t_X, t_Y \in \TT$ define $f: X \ni x \rightarrow \varphi(x,t_X)$ and $g: Y\ni y \rightarrow \psi(y,t_Y)$. We say that $f$ and $g$ are \emph{topologically conjugate} if there exists a homeomorphism $h:X\rightarrow Y$ such that the diagram 
\begin{equation}\label{eq:conj_diag}
   \begin{tikzcd}
        X \arrow{r}{f} \arrow[swap]{d}{h} & X \arrow{d}{h} \\
        Y \arrow{r}{g} & Y
    \end{tikzcd} 
\end{equation}
commutes, i.e., $h\circ f = g\circ h$.
If the map $h:X \to Y$ is not a homeomorphism but a continuous surjection then we say that $g$ is \emph{topologically semiconjugate} to $f$. 

Let us consider as an example $X$ being a unit circle, and $f_{\alpha}$ a rotation of $X$ by an angle $\alpha$. In this case, two maps, $f_{\alpha}, f_{\beta}: X \rightarrow X$ are conjugate if and only if $\alpha = \beta$ or $\alpha = -\beta$. This known fact is verified in the benchmark test in Section~\ref{ssec:circle_rotation}.

In our work we will consider finite time series $\cA=\{x_i\}_{i=1}^{n}$ and $\cB=\{y_i\}_{i=1}^{n}$ so that $x_{i+1} = f^{i}(x_1)$ and $y_{i+1} = g^{i}(y_1)$ for $i \in \{1,2,\ldots,n-1\}$, $x_1 \in X$ and $y_1 \in Y$ and derive criteria to test (semi)topological conjugacy of $f$ and $g$ via $h$ based on those samples and the given possible (semi)conjugacy $h$.

In what follows, a Hausdorff distance between $A, B \subset X$ will be used. It is defined as
\[ \haus(A,B) = \max\{ \sup_{a \in A} d(a,B) , \sup_{b \in B} d(b,A) \} \]
where $d$ is metric in $X$ and $d(x,A):=\inf_{a\in A}d(x,a)$.

\subsection{Takens' Embedding Theorem}
Our work is related to the problem of reconstruction of dynamics from one dimensional time series. 
For a fixed map $f : X \rightarrow X$ and $x_1 \in X$ take a time series $\cA = \{x_i = f^{i-1}(x_1)\}_{i \geq 1}$ being a subset of an attractor $\Omega\subset X$ of the (box-counting) dimension $m$. 
Take $s:X\rightarrow\RR$, a generic measurement function of observable states of the system, and one dimensional time series $\cS=\{s(x_i)\}_{x_i \in \cA}$, associated to $\cA$ . 
The celebrated Takens' Embedding Theorem \cite{TakensOriginal} states that given $\cS$ it is possible to reconstruct the original system with delay vectors, for instance $(s(x_i), s(x_{i+1}), \ldots, s(x_{i+d-1}))$, for sufficiently large \emph{embedding dimension} $d\geq 2m+1$ (the bound is often not optimal). 
The Takens' theorem implies that, under certain generic assumptions,  an embedding of the attractor $\Omega$ into $\RR^d$ given by
\begin{equation}\label{eq:embedding}
    F_{s,f}: \Omega\ni x \mapsto \left(s(x), s(f(x)), \ldots, s(f^{d-1}(x))\right)\in \RR^d
\end{equation}
establishes a \emph{topological conjugacy} between the original system $(\Omega,f)$ and $(F_{s,f}(\Omega),\sigma)$ with the dynamics on $F_{s,f}(\Omega)\subset \mathbb{R}^d$ given by the shift $\sigma$ on the sequence space.
Hence, Takens' Embedding Theorem allows to reconstruct both the topology of the original attractor and the dynamics. 

The formula presented above is a special case of a reconstruction with a \emph{lag} $l$ given by
\begin{equation*}%\label{eq:eq2}
    \emb(\cA, d, l) := \left\{
        (s(x_i), s(x_{i+l}), \ldots, s(x_{i+(d-1)l})) 
        \mid i\in\{1,2,\ldots,n-d\,l\}
    \right\}.
\end{equation*}
From the theoretical point of view, the Takens' theorem holds for an arbitrary lag. However in practice a proper choice of $l$ may strongly affect numerical reconstructions (see \cite[Chapter 3]{kantz_schreiber_2003}).

The precise statements, interpretations and conclusions of the mentioned theorems can be found in \cite{Broer2011, TakensOriginal, Embedology},  and references therein.

\subsection{Search for an optimal dimension for reconstruction}
In practice, the bound in Takens' theorem is often not sharp and an embedding dimension less than $2m+1$ is already sufficient to reconstruct the original dynamics (see \cite{Baranski_2020,Baranski_2022}). Moreover, for time series encountered in practice, the attractor's dimension $m$ is almost always unknown. 
To discover the sufficient dimension of reconstruction, the False Nearest Neighbor (FNN) method \cite{HeggerKantzFNN1999,Kennel1992}, a heuristic technique for estimating the optimal dimension using a finite time series, is typically used.
It is based on an idea to compare the embeddings of a time series into a couple  of consecutive dimensions and to check if the introduction of an additional $d+1$ dimension separates some points that were close in $d$-dimensional embedding.
Hence, it tests whether $d$-dimensional neighbors are (false) neighbors just because of the tightness of the 
$d$-dimensional space.
The dimension where the value of the test stabilizes and no more false neighbors can be detected is proclaimed to be the optimal embedding dimension.

\subsection{False Nearest Neighbor and beyond}
The False Nearest Neighbor method implicitly tests semiconjugacy of $d$ and $d+1$ dimensional Takens' embedding by checking if the neighborhood of $d$-embedded points are preserved in $d+1$ dimension.
This technique was an inspiration for stating a more general question: given two time series, can we test if they were generated from conjugate dynamical systems?
The positive answer could suggest that the two observed signals were actually generated by the same dynamics, but obtained by a different measurement function. In what follows, a number of tests inspired by these observations concerning False Nearest Neighbor method and Takens' Embedding Theorem, are presented.

%%%%%%%%%%%%%%%%%%%%%%%%%%%%%%%%%%%%%%%%%%%%%%
%%%%%%%%%%%%%%%%%%%%%%%%%%%%%%%%%%%%%%%%%%%%%%
%%%%%%%%%%%%%%%%%%%%%%%%%%%%%%%%%%%%%%%%%%%%%%
\section{Conjugacy testing methods}
\label{sec:conjugacy_tests}

In this section we introduce a number of new methods for quantifying the dynamical similarity of two time series.
Before digging into them let us introduce some basic pieces of notation used throughout the section. From now on we assume that $X$ is a metric space. 
Let $\cA=\{x_i\}_{i=1}^n$ be a finite time series in space $X$.
For $k\in\mathbb{N}$, by $\knn(x, k, \cA)$ we denote the set of \emph{$k$-nearest neighbors} of a point $x\in X$ among points in $\cA$. 
Thus, the nearest neighbor of point $x$ can be denoted by $\knn(x,\cA):=\knn(x,1,\cA)$. 
If $x\in\cA$ then clearly $\knn(x,\cA) = \{x\}$.
Hence, it is handful to consider also $\knno(x,k,\cA) :=\knn(x, k, \cA\setminus\{x\})$ and $\knno(x,\cA) :=\knn(x,1, \cA\setminus\{x\})$\footnote{In case of non uniqueness, and arbitrary choice of a neighbor is made.}. 

%%%%%%%%%%%%%%%%%%%%%%%%%%%%%%%%%%%%%%%%%%%%%%
%%%%%%%%%%%%%%%%%%%%%%%%%%%%%%%%%%%%%%%%%%%%%%
\subsection{False Nearest Neighbor method}

The first proposed method is an extension of the already mentioned $\fnn$ technique for estimating the optimal embedding dimension of time series.
    The idea of the classical $\fnn$ method relies on counting  the number of so-called false nearest neighbors depending on the threshold parameter $r$. 
    This is based on the observation that if the two reconstructed points 
    \[\mathbf{s}^1_d:=(s(x_{k_1}), s(x_{k_1+l}), \ldots, s(x_{k_1+(d-1)l}))\] 
    and 
    \[\mathbf{s}^2_d:=(s(x_{k_2}), s(x_{k_2+l}), \ldots, s(x_{k_2+(d-1)l}))\] 
    are nearest neighbors in the $d$-dimensional embedding but the distance between their (d+1)-dimensional counterparts 
    \[\mathbf{s}^1_{d+1}:=(s(x_{k_1}), \ldots, s(x_{k_1+(d-1)l}), s(x_{k_1+dl}))\] 
    and 
    \[\mathbf{s}^2_{d+1}:=(s(x_{k_2}), \ldots, s(x_{k_2+(d-1)l}),s(x_{k_2+dl}))\] 
    in $(d+1)$-dimensional embedding differs too much, 
    then $\mathbf{s}^1_d$ and $\mathbf{s}^2_d$ were $d$-dimensional neighbors only due to folding of the space. In this case, we will refer to them as ``false nearest neighbors''. Precisely, the ordered pair $(\mathbf{s}^1_d, \mathbf{s}^2_d)$ of $d$-dimensional points is counted as false nearest neighbor, if the following conditions are satisfied: (I.) the point $\mathbf{s}^2_d$ is the closest point to $\mathbf{s}^1_d$ among all points in the $d$-dimensional embedding, (II.) the distance $\abs{\mathbf{s}^1_d -\mathbf{s}^2_d}$ between the points $\mathbf{s}^1_d$ and $\mathbf{s}^2_d$ is less than $\sigma/r$, where $\sigma$ is the standard deviation of $d$-dimensional points formed from delay-embedding of the time series and (III.) the ratio between the distance $\abs{\mathbf{s}^1_{d+1} -\mathbf{s}^2_{d+1}}$ of $d+1$-dimensional counterparts of these points, $\mathbf{s}^1_{d+1}$ and $\mathbf{s}^2_{d+1}$, and the distance $\abs{\mathbf{s}^1_d -\mathbf{s}^2_d}$ is greater than the threshold $r$.  
    The condition (III.) is motivated by the fact that  under continuous evolution, even if the original dynamics is chaotic, the position of two close points should not deviate too much in the nearest future (we assume that the system is deterministic, even if subjected to some noise, which is the main assumption of all the nonlinear analysis time series methods). On the other hand, the condition (II.) means that we consider only pairs of points which are originally not too far away since applying the condition (III.) to points which are already outliers in $d$ dimensions does not make sense. 
    Next, the statistic $\fnn(r)$ counts the relative number of such false nearest neighbors 
    i.e. after normalizing with respect to the number of all the ordered pairs of points which satisfy (I.) and (II.). For discussion and some examples  see e.g. \cite{kantz_schreiber_2003}. 

    We generalize the $\fnn$ method to operate in the case of two time series (not necessarily created in a time-delay reconstruction) as follows.
    Let $\cA=\{a_i\}_{i=1}^n\subset X$ and $\cB=\{b_i\}_{i=1}^n\subset Y$ be two time series of the same length.
    Let $\xi:\cA\rightarrow\cB$ be a bijection relating points with the same index, i.e., $\xi(a_i):=b_i$. 
    Then we define the directed $\fnn$ ratio between $\cA$ and $\cB$ as
    \begin{equation}\label{eq:fnn}
        \fnn(\cA, \cB; r) := 
        \frac{
            \sum_{i=1}^{n}
            \Theta\left(
                \frac{\mathbf{d}_Y(b_i,\xi(\knno(a_i,\cA)))}{\mathbf{d}_X(a_i,\knno(a_i,\cA))} - r
            \right)
            \Theta\left(
                \frac{\sigma}{r} - \mathbf{d}_X(a_i,\knno(a_i,\cA))
            \right)
        }{
            \sum_{i=1}^{n}
            \Theta\left(
                \frac{\sigma}{r} - \mathbf{d}_X(a_i,\knno(a_i,\cA))
            \right)
        }
    \end{equation}
    where $\mathbf{d}_X$ and $\mathbf{d}_Y$ denote the distance function respectively in $X$ and $Y$, $\sigma$ is the standard deviation of the data (i.e. the standard deviation of the elements of the sequence $\cA$), $r$ is the parameter of the method and $\Theta$ is the usual Heaviside step function, i.e. $\Theta(x)=1$ if $x>0$ and $0$ otherwise. Note that the distance $\mathbf{d}$ (i.e. $\mathbf{d}_X$ or  $\mathbf{d}_Y$) might be defined in various ways however, as elements of time-series are usually elements of $\RR^k$ (for some $k$), then $\mathbf{d}(x,y)$ is often simply the Euclidean norm $|x - y|$.

    In the original $\fnn$ procedure we compare embeddings of a $1$-dimensional time series $\cA$ into $d$- versus $(d+1)$-dimensional space for a sequence of values of $d$ and $r$.
    In particular, the following application of \eqref{eq:fnn}:
    \begin{equation}
        \fnn(\cA; r, d) := \fnn(\emb_{d}(\cA),\emb_{d+1}(\cA); r),
    \end{equation}
    coincides with the formula used in the standard form of $\fnn$ technique (compare with \cite{kantz_schreiber_2003}).
    For a fixed value of $d$, if the values of $\fnn$ decline rapidly with the increase of $r$, then we interpret that dimension $d$ is large enough not to introduce any artificial neighbors.
    The heuristic says that the lowest $d$ with that property is the optimal embedding dimension for time series $\mathcal{A}$.

%%%%%%%%%%%%%%%%%%%%%%%%%%%%%%%%%%%%%%%%%%%%%%
%%%%%%%%%%%%%%%%%%%%%%%%%%%%%%%%%%%%%%%%%%%%%%
\subsection{K-Nearest Neighbors}

    The key to the method presented in this section is an attempt to weaken and simplify the condition posed by FNN by considering a larger neighborhood of a point.
    As in the previous case, let $\cA=\{a_i\}_{i=1}^n$ and $\cB=\{b_i\}_{i=1}^n$ be two time series of the same length.
    Let $\xi:\cA\rightarrow\cB$ be a bijection defined $\xi(a_i):=b_i$. The proposed statistics, taking into account $k$ nearest neighbors of each point, is given by the following formula:
    \begin{equation}\label{eq:KNNeq}
        \KNN(\cA, \cB; k)
         := \frac{
            \sum_{i=1}^{n} 
            \min\left\{ 
                e\in \NN\ \mid\ \xi\left(\knno(a_i,k,\cA))\subseteq\knno(\xi(a_i), e+k, \cB\right)
            \right\}
            }{
                n^2
            },
    \end{equation}
    where $n$ is the length of time series $\cA$ and $\cB$.
    We refer to the above method as $\KNN$ distance. 
    The idea of the $\KNN$ method can be seen in the Figure~\ref{fig:idea_of_knn}.

    \begin{figure}
        \centering
        \includegraphics[width=0.75\textwidth]{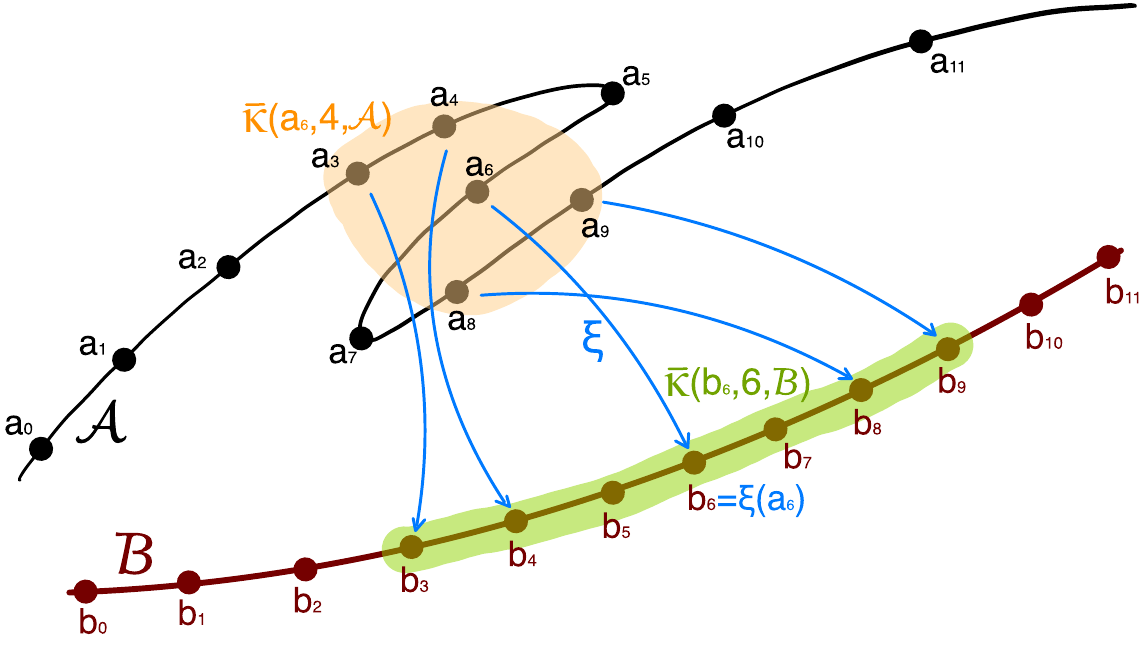}
        \caption{Top (continuous) black line represents trajectory from which $\cA$ is sampled (black dots). 
            Bottom (continuous) trajectory is sampled to obtain $\cB$ (burgundy dots). 
            Set $U:=\knno(a_6,4,\cA)$ highlighted with orange color, represents  4-nearest neighbors of $a_6 \in \cA$.
            The smallest $k$-neighborhood of $b_6$ that contains $\xi(U)$ is the one with $k=6$.
            The corresponding $\knno(b_6,4,\cB)$ is highlighted with green color.
            Hence, the contribution of point $a_6$ to the numerator of $\KNN(\cA,\cB,4)$ is $6-4=2$.}
        \label{fig:idea_of_knn}
    \end{figure}
    
    \begin{remark}
    In the above formula \eqref{eq:KNNeq}, for simplicity there is no counterpart of the parameters $r$ that was present in $\fnn$ which controlled the dispersion of data and outliers. This means that one should assume that the data (perhaps after some preprocessing) does not contain unexpected outliers. Alternatively, the formula might be easily modified to include such a parameter. 
    \end{remark}

    Set $\knno(a_i,k,\cA)$ can be interpreted as a discrete approximation of the neighborhood of $a_i$.
    Thus, for a point $a_i$ the formula measures how much larger neighborhood of the corresponding point $b_i=\xi(a_i)$ we need to take to contain the image of the chosen neighborhood of $a_i$.
    This discrepancy is expressed relatively to the size of the point cloud. 
    Next we compute the average of this relative discrepancy among all points. 
    Moreover, looking at the formula \eqref{eq:KNNeq} immediately reveals that in the numerator we sum up $n$ terms each of which takes values between $0$ and $n$ and it is not hard to give an example when all of these terms are $n$ actually (like a standard $n$-simplex). Therefore, as we want $\KNN$ to be upper-bounded by $1$, we put $n^2$ in the denominator of \eqref{eq:KNNeq} as the normalization factor. 
    
    Note that neither $f$ nor $g$ appear in the definitions of $\fnn$ and $\KNN$.
    Nevertheless, the dynamics is hidden in the indices. 
    That is,  $a_j\in\knno(a_i,k,\cA)$ means that $a_i$ returns to its own vicinity in $\abs{j-i}$ time steps.

%%%%%%%%%%%%%%%%%%%%%%%%%%%%%%%%%%%%%%%%%%%%%%
%%%%%%%%%%%%%%%%%%%%%%%%%%%%%%%%%%%%%%%%%%%%%%
\subsection{Conjugacy test}
   The third method tests the conjugacy of two time series by directly checking the commutativity of the diagram (\ref{eq:conj_diag}) which is tested in a more direct way compared to the methods presented so far.
   We no longer assume that both time series are of the same size, however, the method requires a \emph{connecting map} $h:X\rightarrow Y$, a candidate for a (semi)conjugating map.
   Unlike the map $\xi$ in $\fnn$ and $\KNN$ method map $h$ may transform a point $a_i\in\cA$ into a point in $Y$ that doesn't belong to $\cB$. 
   Nevertheless, the points in $\cB$ are crucial because they carry the information about the dynamics $g:Y\rightarrow Y$.
    Thus, in order to follow trajectories of points in $Y$ we introduce $\tilde{h}:\cA\rightarrow\cB$, a discrete approximation of $h$:
    \begin{equation}\label{eq:h_tilde}
        \tilde{h}(a_i) := \knn\left(h(a_i), \cB\right).
    \end{equation}
    
\begin{figure}
    \centering
    \includegraphics[width=.75\textwidth]{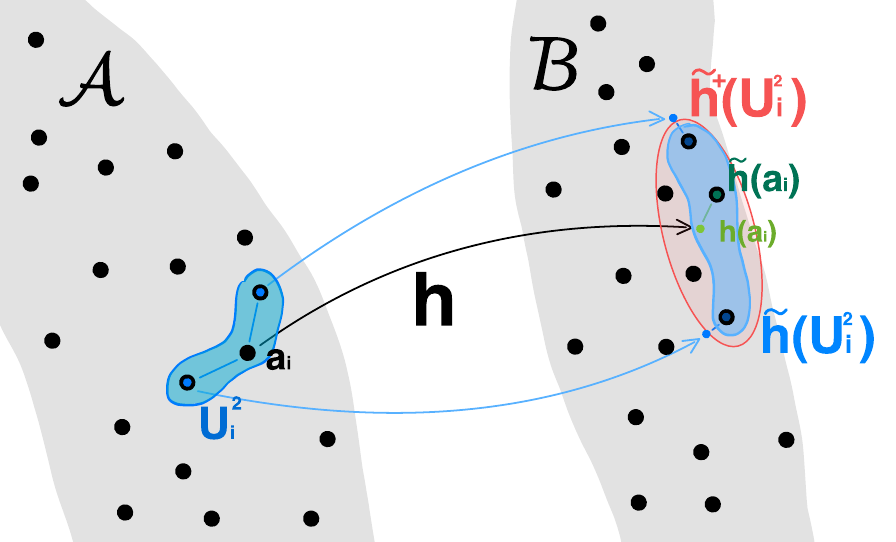}
    
    \caption{
        A pictorial visualization of a difference between $h$, $\tilde{h}$ and $\tilde{h}^+$.
        Map $h$ transforms a point $a\in\cA$ into a point $h(a)\in Y$.
        Map $\tilde{h}$ approximates the value of the map $h$ by finding the closest point in $\cB$ for $h(a)$.
        The discrete neighborhood $U^2_i\subset\cA$ of $a_i$ consists of three points and its image under $\tilde{h}$ has three points as well.
        However, $\tilde{h}^+(U^2_i)$ counts five elements, as there are points in $\cB$ closer to $\tilde{h}(a_i)$ then points in $\tilde{h}(U^2_i)$.
    }
    \label{fig:h_intuition}
\end{figure}

    The map $\tilde{h}$ simply assigns to $a_i$ the closest element(s) of $h(a_i)$ from the time series $\cB$.
    For a set $A\subset\cA$ we compute the value pointwise, i.e. 
        $\tilde{h}(A)=\{\tilde{h}(a)\mid a\in A\}$ (see Figure \ref{fig:h_intuition}).
    Note that it may happen that $\tilde{h}(A)$ has less elements than $A$.

    Denote the discrete $k$-approximation of the neighborhood of $a_i$ in $\cA$, namely the $k$ nearest neighbors of $a_i$, by $U_i^k:=\knn(a_i,k,\cA)\subset\cA$. 
    Then we define
    \begin{align}\label{eq:conjtest}
        \conjtest(\cA,\cB; k,t,h) :=
            \frac{
            \sum_{i=1}^{n} \haus\left(
                    (h\circ f^t)(U_i^k),\ 
                    (g^t\circ\tilde{h})(U_i^k)
                    \right)
            }{
                n\, \diam{\cB}
            },
    \end{align}
    where $\haus$ is the Hausdorff distance between two discrete sets and $\diam{\cB}$ is the diameter of the set $\cB$.
    The idea of the formula \eqref{eq:conjtest} is to test at every point $a_i\in\cA$ how two time series together with map $h$ are close to satisfy diagram \eqref{eq:conj_diag} defining topological conjugacy. 
    First, we approximate the neighborhood of $a_i\in\cA$ with $U_i^k$ and then we try to traverse the diagram in two possible ways.
    Thus, we end up with two sets in $Y$, that is  
                    $(h\circ f^t)(U_i^k)$ and $(g^t\circ\tilde{h})(U_i^k)$.
    We measure how those two sets diverge using the Hausdorff distance.

    The extended version of the test presented above considers a larger approximation of $\tilde{h}(U_i^k)$.
    To this end, find the smallest $k_i$ such that $\tilde{h}(U_i^k)\subset\knn(h(a_i), k_i, \cB)$.
    The corresponding superset defines the enriched approximation (see Figure \ref{fig:h_intuition}): 
    \begin{equation}\label{eq:h_tilde_plus}
        \tilde{h}^+(U_i^k):=\knn(h(a_i), k_i, \cB).
    \end{equation}
    We use it to define a modified version of \eqref{eq:conjtest}.
    
    \begin{equation}\label{eq:conjtest_plus}
        \conjtest^+(\cA,\cB; k,t,h) :=
            \frac{
            \sum_{i=1}^{n} \haus\left(
                    (h\circ f^t)(U_i^k),\ 
                    g^t \left( \tilde{h}^+(U_i^k) \right)
                    \right)
            }{
                n\, \diam{\cB}
            }.
    \end{equation}

    The extension of $\conjtest$ to $\conjtest^+$ was motivated by results of Experiment $\hyperref[sssec:exp_4a]{\text{4A}}$ described in Subsection \ref{sec:lorenz_experiment}.
    The experiment should clarify the purpose of making the method more complex.

    We refer collectively to $\conjtest$ and $\conjtest^+$ as $\conjtest$ methods.
    
The forthcoming results provide mathematical justification of our method, i.e. ``large'' and non-decreasing values of the above tests suggest that there is no conjugacy between two time-series. 
\begin{theorem}\label{thm:ConjTestSemi}
    Let $f:X\rightarrow X$ and $g:Y\rightarrow Y$, where $X\subset\RR^{d_X}$ and $Y\subset\RR^{d_Y}$, be continuous maps ($d_X$ and $d_Y$ denote dimensions of the spaces). For $y_1\in Y$ 
    define  
        $\cB_m:=\left\{b_i:=g^{i-1}(y_1)\mid i\in\{1,\ldots,m\}\right\}$.
    
  Suppose that $Y$ is compact and that the trajectory of $y_1$ is dense in $Y$, i.e. the set $\cB_m$ becomes dense in $Y$ as $m\to \infty$. If $g$ is semiconjugate to $f$ with $h$ as a semiconjugacy map, then for every fixed $n$, $t$ and $k$
\begin{equation}
        \lim_{m\to\infty} \conjtest(\cA_n, \cB_m; k, t, h) = 0, \label{eq:ConjTestConv1}
    \end{equation}
    where $\cA_n:=\left\{a_i:=f^{i-1}(x_1)\mid i\in\{1,\ldots,n\}\right\}$, $x_1\in X$, is any time-series in $X$ of a length $n$. 

    Moreover, the convergence is uniform with respect to $n$ and with respect to the choice of the starting point $x_1$ (i.e. the ``rate'' of convergence does not depend on the time-series $\cA_n$).  
\end{theorem}

\begin{proof}
Since $g$ is semiconjugate to $f$ via $h$, $h:X\rightarrow Y$ is a continuous surjection such that for every $t\in \mathbb{N}$ we have $h\circ f^t=g^t \circ h$.  Fix $t\in\mathbb{N}$ and $k\in\mathbb{N}$ and let $\varepsilon>0$. 
We will show that there exists $M$ such that for all $m>M$, all $n\in\mathbb{N}$ and every finite time-series $\cA_n:=\left\{a_i:=f^{i-1}(x_1)\mid i\in\{1,\ldots,n\}\right\}\subset X$ of length $n$ (where $x_1\in X$ is some point in $X$) it holds that
\begin{equation}
    \conjtest(\cA_n, \cB_m; k, t, h)<\varepsilon.\label{eq:conjpom}
\end{equation}

Note that $|b_2-b_1|\leq |\cB_m|$ for any $m\geq2$ (with $|\cB_m|$ denoting cardinality of the set $\cB_m$), which we will use at the end of the proof. As $g$ is continuous and $Y$ is compact, there exists $\delta$ such that $\vert g^t(y)-g^t(\tilde{y})\vert <\varepsilon\, \vert b_2-b_1\vert$ for every $y, \ \tilde{y}\in Y$ with $\vert y-\tilde{y}\vert<\delta$. 
As $\cB=\{y_1, g(y_1), \ldots, g^m(y_1), \ldots \}=\{b_1, \ldots, b_m, \ldots \}$ is dense in $Y$, there exists $M$ such that if $m>M$ then for every $n\in \mathbb{N}$, every $x_1\in X$ and every $i\in\{1,2, \ldots, \}$ there exists $j_m(i)\in \{1,2, \ldots, m\}$ such that 
\[
\vert b_{j_m(i)}- h(a_i)\vert <\delta,
\]
where $a_i=f^{i-1}(x_1)\in \cA_n$.

Thus for $m>M$, we always (independently of the point $a_i\in X$) have 
\[
\vert h(f^t(a_i))-g^t(\tilde{h}(a_i))\vert = \vert g^t(h(a_i)) - g^t(\tilde{h}(a_i))\vert<\varepsilon \,\vert b_2-b_1\vert
\]
as $g^t(h(a_i))=h(f^t(a_i))$ and $\vert \tilde{h}(a_i) - h(a_i)\vert <\delta$. Consequently,
\[
\haus\left(
                    (h\circ f^t)(U_i^k),\ 
                    (g^t\circ\tilde{h})(U_i^k)
                    \right)<\varepsilon\, \vert b_2-b_1\vert,
\]
where $U_i^k=\knn(a_i,k,\cA_n)$ and $\tilde{h}(U_i^k)=\{\knn(h(a_j),\cB_m) \mid a_j\in U^i_k\}$. Therefore
\[
\frac{
            \sum_{i=1}^{n} \haus\left(
                    (h\circ f^t)(U_i^k),\ 
                    (g^t\circ\tilde{h})(U_i^k)
                    \right)
            }{
                n\, \diam{\cB_m}
            }<\frac{n \, \varepsilon\,  \vert b_2-b_1\vert}{n\, \diam{\cB_m}}\leq \varepsilon
\]
since $\abs{b_2-b_1}\leq \diam{\cB_m}$ for every $m\geq 1$. This proves \eqref{eq:conjpom}.
\end{proof}
The compactness of $Y$ and the density of the set $\cB=\{y_1, g(y_1), \ldots, g^m(y_1), \ldots \}$ in $Y$ is needed to obtain the uniform convergence in \eqref{eq:ConjTestConv1} but,  as follows from the proof above, these assumptions can be relaxed at the cost of possible losing the uniformity of the convergence:  
\begin{corollary}
Let $f:X\rightarrow X$ and $g:Y\rightarrow Y$, where $X\subset\RR^{d_X}$ and $Y\subset\RR^{d_Y}$, be continuous maps.  
    Let $x_1\in X$ and $y_1\in Y$. Define $\cA_n:=\left\{a_i:=f^{i-1}(x_1)\mid i\in\{1,\ldots,n\}\right\}$ and 
$\cB_m:=\left\{b_i:=g^{i-1}(y_1)\mid i\in\{1,\ldots,m\}\right\}$.
    Suppose that $\{h(a_1), \ldots h(a_n)\}\subset \hat{Y}$ for some compact set $\hat{Y}\subset Y$ such that the set $\hat{Y}\cap \cB$ is dense in $\hat{Y}$, where $\cB=\left\{b_1, \ldots, b_m, \ldots \right\}$.  
    
    If $g$ is \textbf{semiconjugate} to $f$ with $h$ as a semiconjugacy, then for every $t$ and $k$
\[
        \lim_{m\to\infty} \conjtest(\cA_n, \cB_m; k, t, h) = 0. 
    \]
\end{corollary}

\begin{remark}
In the above corollary the assumption on the existence of the set $\hat{Y}$ means just that the trajectory of the point $y_1$ contains points $g^{j_i}(y_1)$ which, respectively, ``well-approximate'' points $h(a_i)$, $i=1, 2, \ldots, n$.  

Note also that we do not need the compactness of the space $X$ nor the density of $\cA=\{a_1, a_2, \ldots, a_n,\ldots \}$ in $X$. 
\end{remark}

The following statement is an easy consequence of the statements above
\begin{theorem}\label{thm:ConjTest2}
Let $X\subset\RR^{d_X}$ and $Y\subset\RR^{d_Y}$ be compact sets and $f:X\rightarrow X$ and $g:Y\rightarrow Y$ be continuous maps which are 
\textbf{conjugate} by a homeomorphism $h:X\rightarrow Y$. 
    Let $x_1\in X$, $y_1\in Y$ and $\cA_n$ and  
$\cB_m$ be defined as before. Suppose that $\cA_n$ and $\cB_m$ are dense, respectively, in $X$  and $Y$ as $n\to \infty$ and $m\to \infty$.  
    Then for every $t$ and $k$
\[
        \lim_{m\to\infty} \conjtest(\cA_n, \cB_m; k, t, h) = \lim_{n\to\infty} \conjtest(\cB_m, \cA_m; k, t, h)= 0. 
    \]
\end{theorem}
The assumptions on the compactness of the spaces and density of the trajectories can be slightly relaxed in the similar vein as before.

The above results concern $\conjtest$. Note that in case of $\conjtest^+$ the neighborhoods $\tilde{h}^+(U_i^k)$, thus also $(g^t\circ\tilde{h}^+)(U_i^k)$, can be significantly enlarged by adding additional points to $\tilde{h}(U_i^k)$ and thus increasing the Hausdorff distance between corresponding sets. In order to still control this distance and formally prove desired convergence additional assumptions concerning space $X$ and the sequence $\cA$ are needed:
\begin{theorem}\label{thm:ConjTestPlus}
Let $f:X\to X$ and $g:Y\to Y$, where $X\subset \RR^{d_X}$ and $Y\subset \RR^{d_Y}$ be continuous functions. 
For $x_1\in X$ and $n\in\mathbb{N}$ define $\cA_n:=\left\{a_i:=f^{i-1}(x_1)\mid i\in\{1, 2, \ldots,n\}\right\}$. Similarly, for $y_1\in Y$ and $m\in \mathbb{N}$ define $\cB_m:=\{b_i:=g^{i-1}(y_1)\mid i\in\{1, 2, \ldots,m\}\}$.
Assume that $X$ and $Y$ are compact and that the set $\cA_n$ becomes dense in $X$ as $n\to\infty$, and $\cB_m$ becomes dense in $Y$ as $m\to\infty$.
Under those assumptions, if $g$ is semiconjugate to $f$ with $h:X\to Y$ as a semiconjugacy we have that
\begin{equation}\label{eq:convforplus}
\lim_{n\to\infty}\lim_{m\to\infty} \conjtest^{+}(\cA_n,\cB_m;k,t,h)=0
\end{equation}
for any $k\in\mathbb{N}$ and $t\in \mathbb{N}$.
\end{theorem}

\begin{proof}
Since $g$ is semiconjugate to $f$ via $h$, for every $t\in\mathbb{N}$ we have $h\circ f^t=g^t\circ h$, where $h:X\to Y$ is a continuous surjection.
Expanding \eqref{eq:convforplus} yields 
\begin{equation} \label{eq:twosums}
\begin{split}
& \lim_{n\to\infty}\lim_{m\to\infty} \frac{\sum_{i=1}^{n} \haus\left(
                    (h\circ f^t)(U_i^k),\ 
                    (g^t\circ\tilde{h}^{+})(U_i^k)
                    \right)}{n\ \diam{\cB_m}}\leq\\
 & \lim_{n\to\infty}\lim_{m\to\infty}\frac{\sum_{i=1}^{n} \haus\left(
                    (h\circ f^t)(U_i^k),\ 
                    (g^t\circ\tilde{h})(U_i^k)
                    \right)}{n\ \diam{\cB_m}}+\\ 
                    &  \lim_{n\to\infty}\lim_{m\to\infty}\frac{\sum_{i=1}^{n} \haus\left(
                    (h\circ f^t)(U_i^k),\ 
                    (g^t(\tilde{h}^+ (U_i^k)\setminus \tilde{h}(U_i^k)))
                    \right)}{n\ \diam{\cB_m}}.
\end{split}
\end{equation}

Recall that $U_{i}^k:=\knn(a_i,k,\cA_n)$, $\tilde{h}(a_i):=\knn(h(a_i),\cB_m)$, $\tilde{h}(U_i^k):=\{\tilde{h}(a_j): \ a_j\in U_i^k\}$ and $\tilde{h}^+(U_i^k):=\knn(h(a_i),k_i,\cB_m)$, where $k_i$ is the smallest integer $k_i$ such that $\tilde{h}(U_i^k)\subset \knn(h(a_i),k_i,\cB_m)$. Thus in particular, $\tilde{h}(U_i^k)\subset \tilde{h}^+(U_i^k)$. Obviously all these neighborhoods   $U_i^k$, $\tilde{h}(U_i^k)$ and   $\tilde{h}^+(U_i^k)$  depend on $n$ and $m$ (since they are taken with respect to $\cA_n$ and $\cB_m$). 

Note that from Theorem \ref{thm:ConjTestSemi} already follows that the first of the two terms in the sum in \eqref{eq:twosums} vanishes. Thus we will only show that the second double limit vanishes as well. 

Let $\varepsilon>0$, $k\in \NN$ and $t\in \mathbb{N}$. Since $g^t: Y\to Y$ is a continuous function on a compact metric space $Y$, there exists $\delta$ such that $\vert g^t(x)-g^t(y)\vert<\frac{\varepsilon}{2}$ whenever $x,y\in Y$ are such that $\vert x-y\vert<\delta$. Similarly, since $X$ is compact and $h:X\to Y$ is continuous, there exists $\delta_1$ such that $\vert h(x)-h(y)\vert<\frac{\delta}{2}$ whenever $x,y \in X$ such that $\vert x-y\vert<\delta_1$. 

Since $\cB$ is dense in $Y$, there exists $M\in\mathbb{N}$ such that for $m>M$ and every $y\in Y$, there exists $\tilde{b}\in \cB_m$ such that $\vert \tilde{b}-y\vert <\frac{\delta}{4}$. Moreover, from the density of $\cA$, there exists $N\in\mathbb{N}$ such that for every $n>N$ and every $i\in \{1, 2, \ldots, n\}$ we have $\diam{U_i^k}<\delta_1$, i.e. if $a_j\in U^k_i=\knn(a_i,k,\cA_n)$ then $\vert a_j-a_i\vert <\delta_1$ and consequently 
\begin{equation}\label{eq:gwiazdka1}
\vert g^t(h(a_j))-g^t(h(a_i))\vert <\frac{\varepsilon}{2}.
\end{equation}
Assume thus $n>N$. Then for $m>M$ and every $i\in\{1, 2, \ldots n\}$ we have $\diam{U^k_i}<\delta_1$ which also implies $\diam{h(U^k_i)}<\frac{\delta}{2}$. As $m>M$, every point of $h(U_i^k)$ can be approximated by some point of $\cB_m$  with the accuracy better than $\frac{\delta}{4}$. Consequently, $\diam{\tilde{h}(U_i^k)}<\delta$ for every $i\in \{1, 2, \ldots, n\}$.

Suppose that $\tilde{b}\in \tilde{h}^+(U_i^k)\setminus \tilde{h}(U_i^k)$ for some $\tilde{b}\in \cB_m$. Then, by definition of $\tilde{h}^+$, 
\begin{equation}\label{eq:gwiazdka2}
\vert \tilde{b} -h(a_i)\vert \leq \diam{\tilde{h}(U_i^k)}<\delta .
\end{equation}
Thus for any $a_j\in U^k_i=\knn(a_i,k,\cA_n)$ and any $\tilde{b}\in \tilde{h}^+(U_i^k)\setminus \tilde{h}(U_i^k)$ we obtain
\begin{equation*} 
\begin{split}
    &\vert h(f^t(a_j)) - g^t(\tilde{b})\vert \\ 
    &\hspace{1.cm}     \leq  \vert h(f^t(a_j)-g^t(h(a_j))\vert + \vert g^t(h(a_j)) - g^t(h(a_i))\vert + \vert g^t(h(a_i)) -g^t(\tilde{b}) \vert  
\end{split}
\end{equation*}
where
\begin{itemize}
\item $\vert h(f^t(a_j)-g^t(h(a_j))\vert =0$ by semiconjugacy assumption
\item $\vert g^t(h(a_j)) - g^t(h(a_i))\vert <\frac{\varepsilon}{2}$ as follows from \eqref{eq:gwiazdka1}
\item $\vert g^t(h(a_i)) -g^t(\tilde{b}) \vert <\frac{\varepsilon}{2}$ as follows from \eqref{eq:gwiazdka2}.
\end{itemize}

Finally for every $i\in\{1, 2, \ldots n\}$, every $a_j\in U^k_i$ and every $\tilde{b}\in (\tilde{h}^+(U_i^k)\setminus \tilde{h}(U_i^k))$ we have
$\vert h(f^t(a_j))-g^t(\tilde{b})\vert<\varepsilon$ meaning that
\[
 \frac{\sum_{i=1}^{n} \haus\left(
                    (h\circ f^t)(U_i^k),\ 
                    g^t(\tilde{h}^+ (U_i^k)\setminus \tilde{h}(U_i^k))
                    \right)}{n\ \diam{\cB_m}}<\frac{\varepsilon}{\diam{\cB_m}}
\]
if only $n>N$ and $m>M$. 

This shows that the value of $\conjtest^+(\cA_n,\cB_m;k,t,h)$ can be arbitrarily small if $n$ and $m$ are sufficiently large and ends the proof.
\end{proof}

Finally, let us mention that conjugacy tests described in this Section are not tests in the statistical sense. They should be rather considered as methods of assessing dynamical similarity of the two unknown systems when only small finite samples of their trajectories are available. Trajectories related by topological conjugacy will give values of the tests close to $0$, and those coming from not conjugate systems are expected to result with higher values of the tests. 

The discussed task is already, to a certain extent, considered in the literature. The paper~\cite{pecora1995} develops sets of statistics which are intended to characterize, in terms of probabilities and confidence levels, whether time delay embeddings of the two time series are connected by a continuous, injected or differentiable map. 
The work~\cite{pecora1995} presents method to assess (generalized) synchronization of time series, coupling in complex population dynamics (see~\cite{LMoniz_nichols2005}) or detecting damage in some material structures (see~\cite{moniz2004}). 
The statistics proposed in those papers are inspired by notions of continuity, differentiability etc., typically involving quantities like $\epsilon$'s and $\delta$'s. These values need to be fixed and enforce the user to understand how $\delta$'s scale with $\epsilon$ which is typically hard. It seems to be possible to adopt $\conjtest$'s methods to the framework of statistical tests and will be 
considered in the future.

%%%%%%%%%%%%%%%%%%%%%%%%%%%%%%%%%%%%%%%%%%%%%%%%%%%%%%%%%%%%%%%%%%%%%%%%%%%%%%%%%%%%%%%%%%%
%%%%%%%%%%%%%%%%%%%%%%%%%%%%%%%%%%%%%%%%%%%%%%%%%%%%%%%%%%%%%%%%%%%%%%%%%%%%%%%%%%%%%%%%%%%
%%%%%%%%%%%%%%%%%%%%%%%%%%%%%%%%%%%%%%%%%%%%%%%%%%%%%%%%%%%%%%%%%%%%%%%%%%%%%%%%%%%%%%%%%%%
%%%%%%%%%%%%%%%%%%%%%%%%%%%%%%%%%%%%%%%%%%%%%%%%%%%%%%%%%%%%%%%%%%%%%%%%%%%%%%%%%%%%%%%%%%%
%%%%%%%%%%%%%%%%%%%%%%%%%%%%%%%%%%%%%%%%%%%%%%%%%%%%%%%%%%%%%%%%%%%%%%%%%%%%%%%%%%%%%%%%%%%
%%%%%%%%%%%%%%%%%%%%%%%%%%%%%%%%%%%%%%%%%%%%%%%%%%%%%%%%%%%%%%%%%%%%%%%%%%%%%%%%%%%%%%%%%%%
\section{Conjugacy experiments}\label{sec:experiments}

In this section the behavior of the described methods is experimentally studied. For that purpose a benchmark set of a number of time series originating from (non-)conjugate dynamical systems is generated. 
A time series of length $N$ generated by a map $f:X\rightarrow X$ with a starting point $x_1\in X$ is denoted by
\begin{equation*}
    \varrho(f,x_1,N):=\left\{f^{j-1}(x_1)\in X\mid j\in\{1, 2, \ldots,N\}\right\}.
\end{equation*}

All the experiments were computed in Python using floating number precision. The implementations of the methods presented in this paper as well as the notebooks recreating the presented experiments are available at \href{https://github.com/dioscuri-tda/conjtest}{https://github.com/dioscuri-tda/conjtest}.

%%%%%%%%%%%%%%%%%%%%%%%%%%%%%%%%%%%%%%%%%%%%%%
%%%%%%%%%%%%%%%%%%%%%%%%%%%%%%%%%%%%%%%%%%%%%%
%%%%%%%%%%%%%%%%%%%%%%%%%%%%%%%%%%%%%%%%%%%%%%
\subsection{Irrational rotation on a circle}\label{ssec:circle_rotation}
    The first example involves a dynamics generated by rotation on a circle by an irrational angle.
    Let us define a 1-dimensional circle as a quotient space $\sph:=\RR/\ZZ$.
    Denote the operation of taking a decimal part of a number (modulo 1) by $\modone{x} := x-\lfloor x\rfloor$.
    Then, for a parameter $\phi\in[0,1)$ we define a rigid rotation on a circle, $f_{[\phi]}:\sph\rightarrow\sph$, as 
    \begin{equation*}
        f_{[\phi]}(x) := 
            \bmodone{x+\phi}.
    \end{equation*}

    We consider the following metric on $\sph$ 
    \begin{equation}\label{eq:sphere_metric}
        \mathbf{d}_\sph:\sph\times\sph\ni (x,y) \mapsto \min\left(\bmodone{x-y}, \bmodone{y-x}\right)\in [0,1).
    \end{equation}
    
In this case $\mathbf{d}_\sph(x,y)$ can be interpreted as the length of the shorter arc joining points $x$ and $y$ on $\sph$.
    It is known that two rigid rotations, $f_{[\phi]}$ and $f_{[\psi]}$, are topologically conjugate if and only if $\phi=\psi$ or  $\phi+\psi=1$ (see e.g. Theorem 2.4.3 and Corollary 2.4.1 in \cite{szlenk}). 
    In the first case the conjugating circle homeomorphism $h$ preserves the orientation i.e. the lift $H:\mathbb{R} \to \mathbb{R}$ of $h: \sph \to \sph $ satisfies $H(x+1)=H(x)+1$ for every $x\in \mathbb{R}$  
    and in the second case $h$ reverses the orientation $H(x+1)=H(x)-1$ and the two rotations $f_{[\phi]}$ and $f_{[\psi]}$ are just mutually inverse. 

    Moreover, for a map $f_{[\phi]}$ we introduce a family of topologically conjugate maps given by
    \begin{equation*}
        f_{[\phi], s}(x):=\Bigl(\bmodone{\modone{x}^s + \phi}\Bigr)^{1/s}, \ x\in\RR
    \end{equation*}
    with $s>0$.
    In particular, $f_{[\phi]}=f_{[\phi],1}$. It is easy to check that by putting $h_{s}(x):=\modone{x}^s$ we get 
    $f_{[\phi],s} = h_{s}^{-1}\circ f_{[\phi]}\circ h_s$.

\subsubsection{Experiment 1A}\label{sssec:exp_1a}
    \paragraph{Setup}
    Let $\alpha=\frac{\sqrt{2}}{10}$. 
    In the first experiment we compare the following time series:
    \begin{align*}
        \cR_1&=     \varrho(f_{[\alpha]}, 0.0, 2000),\qquad
        &\cR_2&=    \varrho(f_{[\alpha]},    0.25, 2000),\\
        \cR_3&=     \varrho(f_{[\alpha+0.02]},   0.0, 2000),\qquad
        &\cR_4&=    \varrho(f_{[2\alpha]},    0.0, 2000),\\
        \cR_5&=     \varrho(f_{[\alpha], 2}, 0.0, 2000), \qquad  
        &\cR_6&=     \cR_5 + \err(0.05),  
    \end{align*}
where $\err(\epsilon)$ denotes a uniform noise sampled from the interval $[-\epsilon, \epsilon]$.

As follows from Poincar\'{e} Classification Theorem, $f_{[\alpha]}$ and $f_{[2\alpha]}$ are not conjugate nor semiconjugate whereas $f_{[\alpha]}$ and $f_{[\alpha],2}$ are conjugate via $h_2$. 
Thus the expectation is to confirm conjugacy of $\cR_1$ and $\cR_2$ and of $\cR_1$ and $\cR_5$ and indicate deviations from conjugacy in all the remaining cases.   

In case of $\conjtest$ the comparison $\cR_1$ versus $\cR_2$, $\cR_1$ versus $\cR_3$ and $\cR_1$ versus $\cR_4$ was done with $h\equiv\id_{\sph}$. As we already mentioned, there is no conjugacy between $\cR_1$ and $\cR_3$, nor between $\cR_1$ and $\cR_4$, as the angles of these rotations are different. Thus there is no true connecting homeomorphism between $\cR_1$ and $\cR_3$ and between $\cR_1$ and $\cR_4$. However, in order to apply $\conjtest$s  we need to pick some candidate for a matching map between two point clouds and as the first choice one can always start with the identity map, especially for comparing point clouds generated by trajectories starting at the same or close initial points. 
Therefore in this experiment we use $h\equiv\id_{\sph}$ for comparing  $\cR_1$ both with $\cR_3$ and $\cR_4$. 
When comparing $\cR_1$ versus $\cR_4$ and $\cR_1$ versus $\cR_5$ we use homeomorphism $h_2(x):=\modone{x}^2$.
Let us recall that for $\fnn$ and $\KNN$ methods we always use $h(x_i)=y_i$, a connecting homeomorphism based on the indices correspondence.

\begin{table}
\centering
\begin{footnotesize}
\begin{tabular}{|c|c|c|c|c|c|} \hline
\backslashbox{method}{test}
            & \thead{$\cR_1$ vs. $\cR_2$\\ (starting point\\ perturbation)} 
            & \thead{$\cR_1$ vs. $\cR_3$\\ (angle\\ perturbation)} 
            & \thead{$\cR_1$ vs. $\cR_4$\\ (angle\\ doubled)} 
            & \thead{$\cR_1$ vs. $\cR_5$\\ (nonlinear\\ rotation)}
            & \thead{$\cR_1$ vs. $\cR_6$\\ (noise)} \\ \hline
$\fnn$ (r=2) &
    \slashbox{0.0}{0.0} &
    \slashbox{1.0}{1.0} &
    \slashbox{.393}{.393} & 
    \slashbox{.063}{0.0} & 
    \slashbox{.987}{.986}  \\\hline 
$\KNN$ (k=5) & 
    \slashbox{0.0}{0.0} &
    \slashbox{.257}{.756} &
    \slashbox{.003}{.997} &
    \slashbox{0.0}{0.0} &
    \slashbox{.150}{.152}  \\\hline
\thead{$\conjtest$\\(k=5, t=5)}  & 
    \slashbox{.001}{.001} &
    \slashbox{.201}{.201} &
    \slashbox{.586}{.586} &
    \slashbox{0.0}{0.0} &
    \slashbox{.142}{.181}  \\\hline
\thead{$\conjtest^+$\\(k=5, t=5)}  & 
    \slashbox{.001}{.001} &
    \slashbox{.201}{.201} &
    \slashbox{.586}{.586} &
    \slashbox{0.0}{0.0} &
    \slashbox{.162}{.181}  \\\hline
\end{tabular}
\end{footnotesize}
\caption{
    Comparison of conjugacy measures for time series generated by the rotation on a circle. 
    The number in the upper left part of the cell corresponds to a comparison of the first time series vs. the second one, while the lower right corresponds to the inverse comparison. 
    As follows from formulas at the beginning of Section \ref{sssec:exp_1a} the considered trajectories have length $N=2000$, other corresponding  parameters are stated in the table. 
}
\label{tab:experiment_rotation}
\end{table}

\paragraph{Results}
The results are presented in Table \ref{tab:experiment_rotation}. 
Since the presented methods are not symmetric, order of input time series matters. To accommodate this information, every cell contains two values, above and below the diagonal.
For the column with header "$\cR_i$ vs. $\cR_j$", the cells upper value corresponds to the outcome of $\fnn(\cR_i, \cR_j;r)$,  
    $\KNN(\cR_i, \cR_j;k)$, $\conjtest(\cR_i, \cR_j;k,t,h)$ and $\conjtest^+(\cR_i, \cR_j;k,t,h)$, respectively to the row.
The lower values corresponds to $\fnn(\cR_j, \cR_i;r),\ \ \KNN(\cR_j, \cR_i;k),\ \ \conjtest(\cR_j, \cR_i;k,t,h)$ and \newline$\conjtest^+(\cR_j, \cR_i;k,t,h)$, respectively.

As we can see from Table \ref{tab:experiment_rotation} the starting point does not affect results of methods ($\cR_1$ vs. $\cR_2$) since all the values in the first column are close to $0$. It is expected due to the symmetry of the considered system.
A nonlinearity introduced in time series $\cR_5$ also does not affect the results. Despite the fact that $f_{[\alpha],2}$ is nonlinear, it is conjugate to the rotation $f_{[\alpha]}$ which is reflected by tests' values.   
However, when we change the rotation parameter we can see an increase of measured values ($\cR_1$ vs. $\cR_3$ and $\cR_1$ vs. $\cR_4$).
It is particularly visible in the case of $\fnn$ and $\KNN$.
Interestingly, a small perturbation of the angle ($\cR_3$) can cause a bigger change in a value then a large one ($\cR_4$).
We investigate how the perturbation of the rotation parameter affects values of examined methods in Experiment $\hyperref[sssec:exp_1b]{\text{1B}}$.
Moreover, the last column ($\cR_1$ vs. $\cR_6$) shows that $\fnn$ is very sensitive to  noise, while $\KNN$ and $\conjtest$ methods present some robustness. 
The influence of noise on the value of the test statistics is further studied in Experiment $\hyperref[sssec:exp_1c]{\text{1C}}$. 

Note also that additional summary comments concerning Table \ref{tab:experiment_rotation}, as well as results of other forthcoming experiments, will be also presented at the end of the article. 

\subsubsection{Experiment 1B}\label{sssec:exp_1b}
In this experiment we test how the difference of the system parameter affects tested methods.
\paragraph{Setup}
Let $\alpha:=\frac{\sqrt{2}}{10}\approx 0.141$.
We consider a family of time series parameterized by $\beta$.
\begin{equation}\label{eq:exp_rot_ts_family}
    \left\{\cR_\beta:=\varrho(f_{[\beta]}, 0.0, 2000)\mid\beta=\alpha + \frac{i \alpha}{100},\ i\in[-50,-49,\ldots, 125]\right\}.
\end{equation}
Thus, the tested interval of values of $\beta$ is approximately $[0.07, 0.32]$.
As a reference value we chose $\alpha=\frac{\sqrt{2}}{10}\approx 0.141$.
We denote the corresponding time series as $\cR_\alpha$.
We compare all time series from the family \eqref{eq:exp_rot_ts_family} with $\cR_\alpha$.
In the case of $\conjtest$ methods we use $h=\id$. 

\begin{figure}
    \centering
    \includegraphics[width=0.49\textwidth]{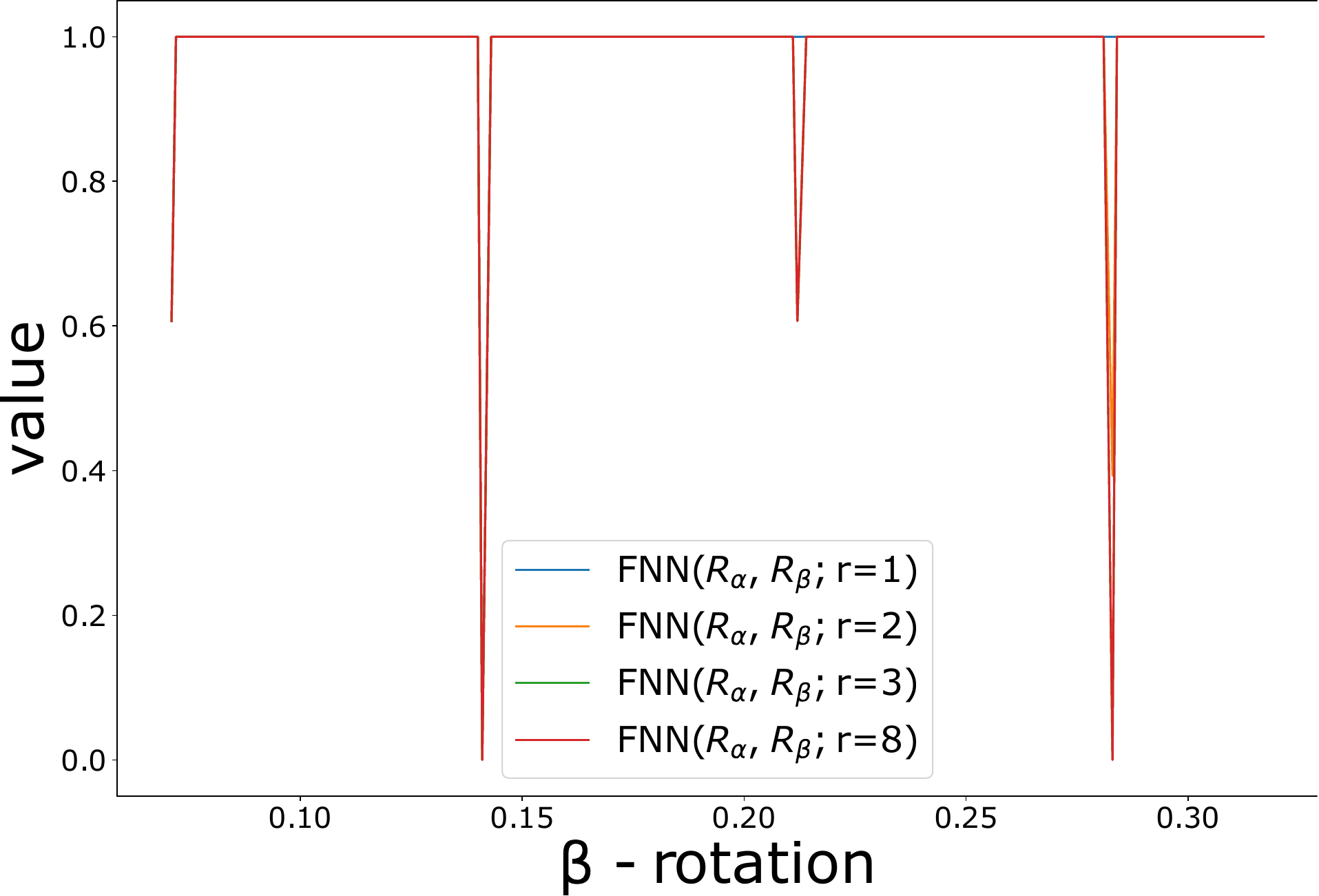}
    \includegraphics[width=0.49\textwidth]{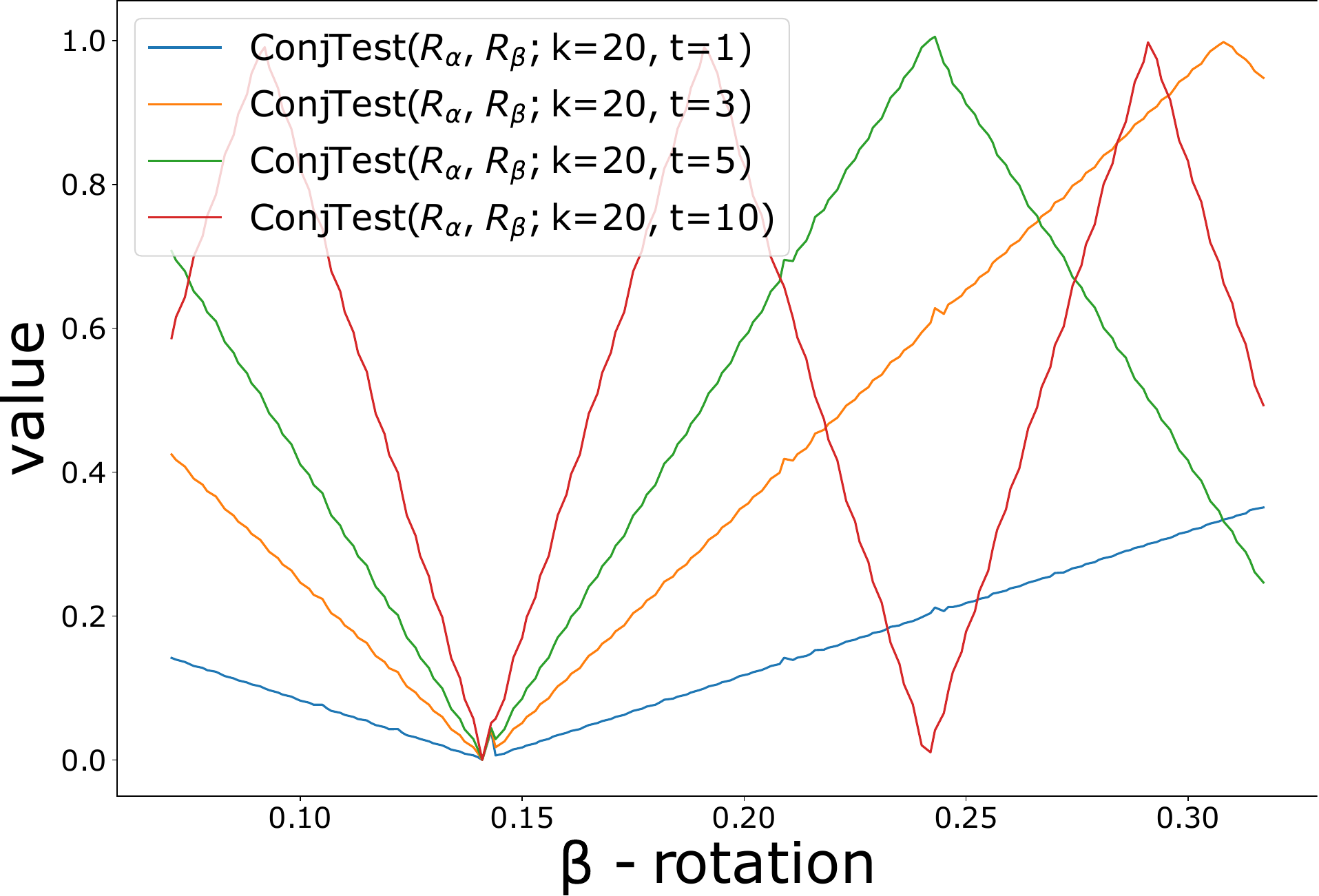}
    
    \includegraphics[width=0.8\textwidth]{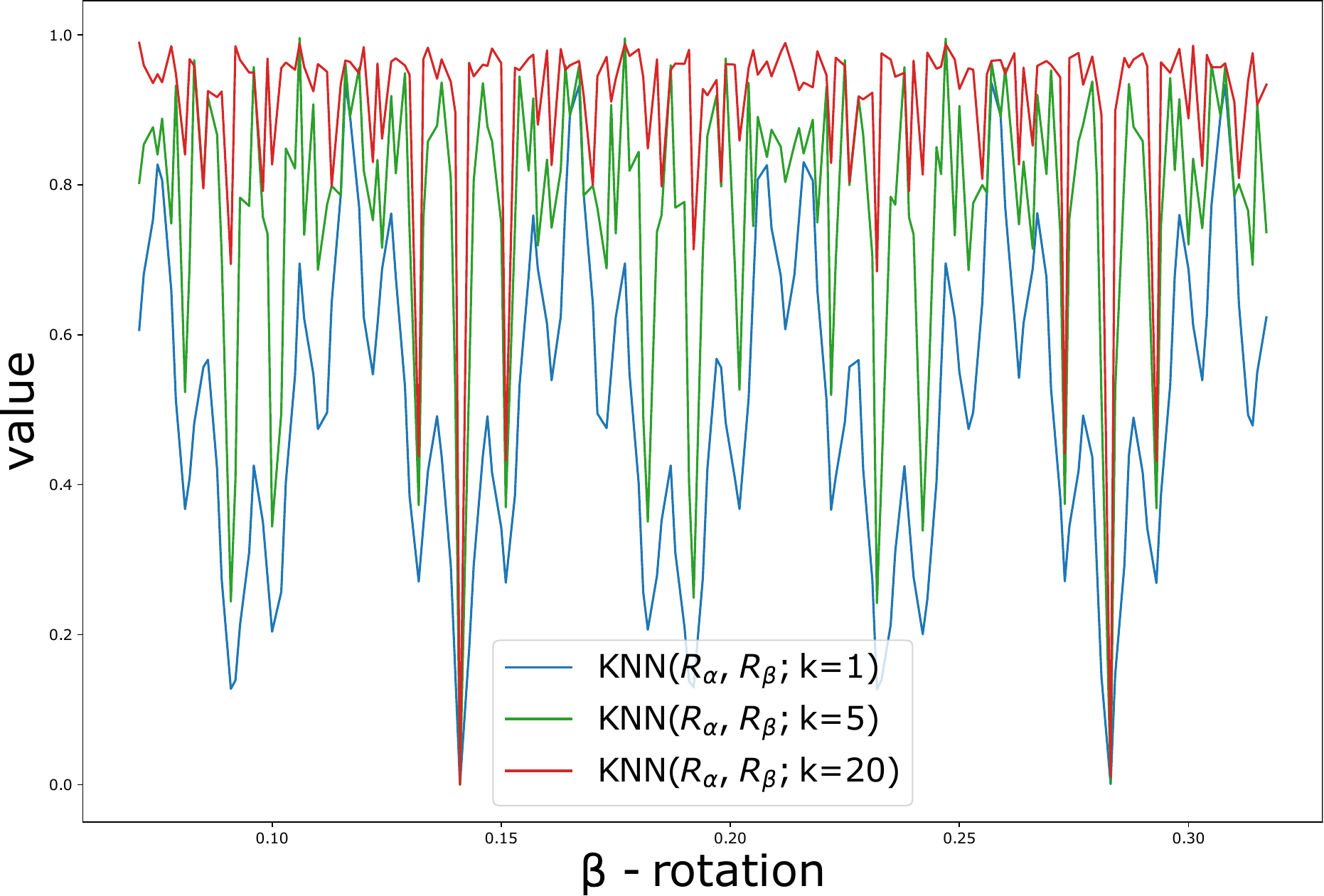}
    
    \caption{Dependence of the conjugacy measures on the perturbation of rotation angle. 
    Top left: $\fnn$ method.
    Top right: $\conjtest$ method.
    Bottom: $\KNN$ method.
    }
    \label{fig:experiment_rotation_angle}
\end{figure}
\paragraph{Results}
The outcome of the experiment is plotted in Figure \ref{fig:experiment_rotation_angle}.
We can see that all methods give values close to $0$ when comparing $\cR_\alpha$ with itself.
For different values of parameter $r$ of $\fnn$ plots (Figure \ref{fig:experiment_rotation_angle} top left looks almost identically. 
Even a small perturbation of the rotation parameter causes an immediate jump of $\fnn$ value from 0 to 1, making it extremely sensitive to any changes in the system. Obviously, unless $\beta=\alpha$, $\cR_{\alpha}$ and $\cR_{\beta}$ are not conjugate. However, sometimes it might be convenient to have a somehow smoother relation of the test value to the infinitesimal change of the rotation angle.
$\KNN$ method seems to behave inconsistently, but we can see that the higher parameter $k$ gets the closer we get to a shape resembling the curve obtained with $\fnn$.
On the other hand, $\conjtest$ shows a linear dependence on $\beta$ parameter. 
Moreover, different values of $\conjtest$'s parameter $t$ result in a different slope of this dependency.

Both, $\fnn$ and $\KNN$ exhibit an interesting drop of the value when $\beta\approx 0.283$ that is $\beta =2\alpha$. 
Formally, we know that $f_{[\alpha]}$ and $f_{[2\alpha]}$ are not conjugate systems.
However, we can explain this outcome by analyzing the methods.
Let $a_i\in\cR_\alpha$ and let $\tau\in\ZZ$ such that $a_j:=a_{i+\tau}\in\cR_\alpha$ be the nearest neighbor of $a_i$.
In particular, $\tau\in\ZZ$.
By \eqref{eq:sphere_metric} we get $\smetr(a_i, a_j)= \bmodone{\alpha\tau}$ or $\bmodone{-\alpha\tau}$.
There is an $N\in\ZZ$ and a $\delta\in[0,1)$ such that $\alpha\tau= N + \delta$.
Since $\smetr(a_i, a_j)\approx 0$, it follows that $\modone{\delta}\approx 0$.
To get $\fnn$ we also need to know $\smetr(b_i, b_j)$.
Let $\beta=z\alpha$.
Then, $b_i=\bmodone{z\alpha i}$, $b_j=\bmodone{z\alpha i + z\alpha\tau}$ and $\smetr(b_i, b_j)=\bmodone{z\alpha\tau}$ or $\bmodone{-z\alpha\tau}$.
Thus, $z\alpha\tau = zN + z\delta$.
We assume that $z\delta\in[0,1)$, because $\modone{\delta}\approx 0$ and $z$ is not very large .
Again, there exists an $M\in\ZZ$ and $\epsilon\in[0,1)$ such that $zN=M+\epsilon$.
Now, if $zN\in\ZZ$, then $\epsilon=0$, $\smetr(b_i,b_j)=\bmodone{z\delta} = z\,\smetr(a_i,a_j)$ (last equality given by $\modone{\delta}\approx 0$) and $\frac{\smetr(b_i,b_j)}{\smetr(a_i,a_j)}=z$.
If $zN\not\in\ZZ$ then $\epsilon\neq 0$ and $\frac{\smetr(b_i,b_j)}{\smetr(a_i,a_j)}=\frac{>0}{\sim 0}$.
Hence, the fraction gives a large number and the numerator of $\fnn$ will count most of the points, unless $zN\in\ZZ$ which is always satisfied when $z\in\ZZ$.
Moreover, for the irrational rotation $\tau$ might be large.
In our experiments we usually get $\abs{\tau}>1000$.
Thus, $N$ is large and $\epsilon$ is basically a random number.
In the case of $\KNN$ there is a chance that at least for some of the $k$-nearest neighbors $zN\in\ZZ$.
Hence, the more rugged shape of the curve.

In the case of $\conjtest$ we observe a clear impact of $\conjtest$'s parameter $t$ on the shape of the curve.
The method takes $k$-nearest neighbors of a point $x_i$ ($U_i^k$ in the formula \eqref{eq:conj_diag}) and moves them $t$ times about angle $\alpha$.
At the same time the corresponding image of those points in the system $\cR_\beta$ ($\tilde{h}(U_i^k)$ in the formula \eqref{eq:conj_diag}) is rotated $t$ times about $\beta$ angle.
Thus, the discrepancy of the position of those two sets of points is proportional to $t\beta$.
In particular, when $\bmodone{t\beta}=\alpha$, these two sets are in the same position.

\begin{figure}
    \centering
    \includegraphics[width=0.8\textwidth]{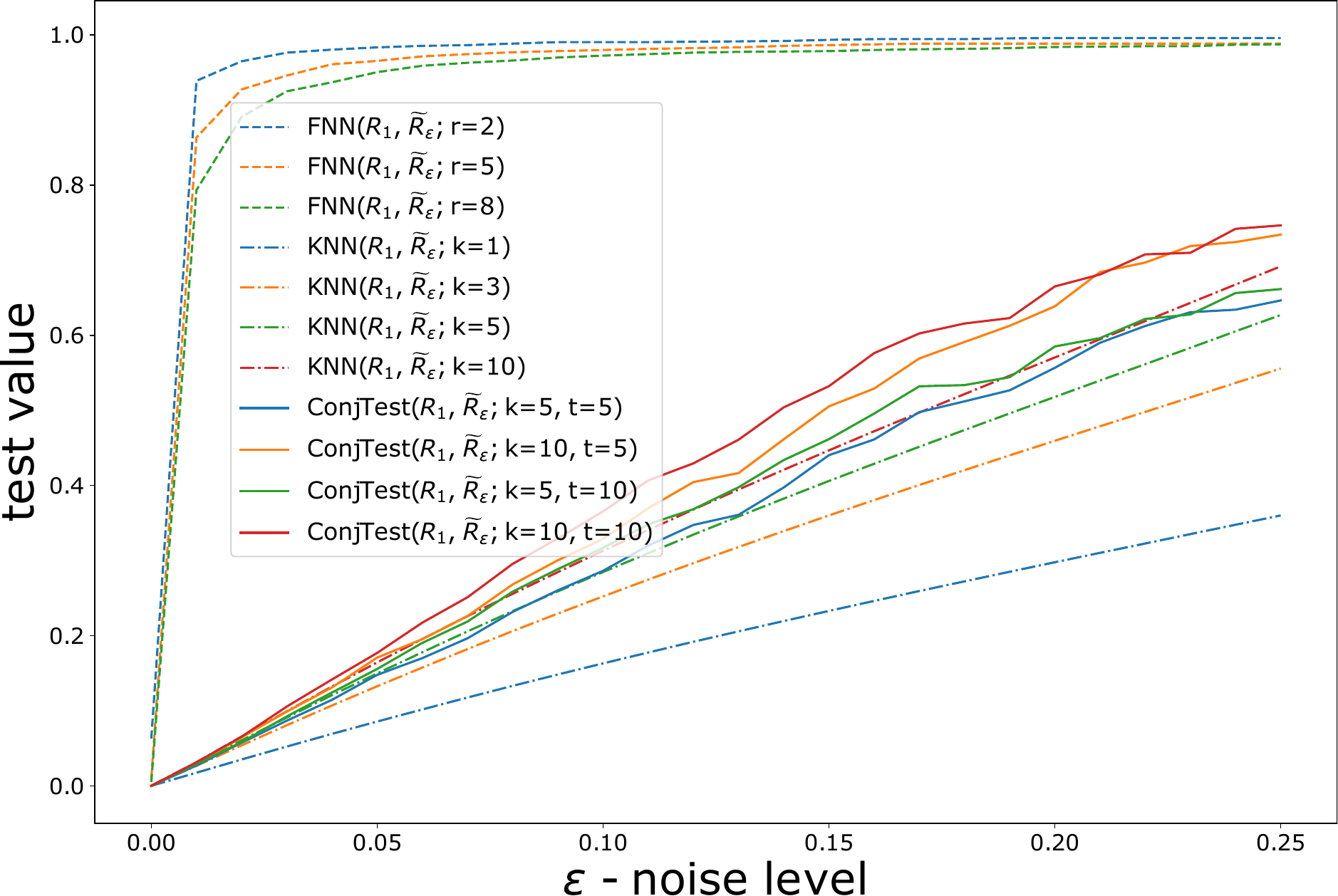}
    \caption{Dependence of conjugacy measures on the perturbation of time series.
    }
    \label{fig:experiment_noise_angle_dmax}
\end{figure}

\subsubsection{Experiment 1C}\label{sssec:exp_1c}
In this experiment, instead of perturbing the parameter of the system we perturb the time series itself by applying a noise to every point of the series.
\paragraph{Setup}
Set $\alpha=\frac{\sqrt{2}}{10}$.
We compare a time series $\cR_1:=\varrho(f_{[\alpha]},0.0,2000)$ with a family of time series:
\begin{equation}\label{eq:exp_noise_ts_family}
    \left\{\widetilde{\cR}_\epsilon:=\varrho(f_{[\alpha],2}, 0.0, 2000)+\err(\epsilon)\mid\epsilon\in[0.00, 0.25]\right\},
\end{equation}
where $\err(\epsilon)$ is a uniform noise sampled from the interval $[-\epsilon, \epsilon]$ applied to every point of the time series.
In the case of $\conjtest$ we again use $h(x)=\modone{x^2}$.

\paragraph{Results}
Results are presented in Figure \ref{fig:experiment_noise_angle_dmax}.
Again, $\fnn$ presents a very high sensitivity on any disruption of a time series and even a small amount of noise gives a conclusion that two systems are not conjugate.
On the other hand, $\KNN$ and $\conjtest$ present an almost linear dependence on noise level.
Note that higher values of parameters $k$ and $t$ make methods more sensitive to the noise.

%%%%%%%%%%%%%%%%%%%%%%%%%%%%%%%%%%%%%%%%%%%%%%
%%%%%%%%%%%%%%%%%%%%%%%%%%%%%%%%%%%%%%%%%%%%%%
%%%%%%%%%%%%%%%%%%%%%%%%%%%%%%%%%%%%%%%%%%%%%%
\subsection{Example: irrational rotation on a torus}\label{ssec:torus_rotation}
Let us consider a simple extension of the previous rotation example to a rotation on a torus.
With a torus defined as $\TT:=\sph\times\sph$, where $\sph=\RR/\ZZ$, we can 
introduce map $f_{[\phi_1, \phi_2]}:\TT\rightarrow\TT$ defined as
\begin{equation*}
    f_{[\phi_1,\phi_2]}(x^{(1)},x^{(2)}) = ( \bmodone{x^{(1)}+\phi_1}, \bmodone{x^{(2)}+\phi_2}),
\end{equation*}
where $\phi_1,\phi_2\in [0, 1)$.
We equip the space with the maximum metric $\textbf{d}_\TT$:
\[
    \mathbf{d}_\TT:\TT\times\TT\ni ((x_1, y_1), (x_2, y_2)) \mapsto 
        \max\left(
            \mathbf{d}_\sph(x_1,x_2),
            \mathbf{d}_\sph(y_1,y_2)
            \right)\in [0,1),
\]
where $\mathbf{d}_\sph$ is the sphere metric (see \eqref{eq:sphere_metric}).

Note that rotation on a torus described above and rotation on a circle $f_{[\phi_i]}:\sph\rightarrow\sph$ studied in Section \ref{ssec:circle_rotation} give a simple example of semiconjugate systems. Namely, let $h:\TT\rightarrow\sph$ be a projection $h_i(x^{(1)}, x^{(2)})=x^{(i)}$, $i=1,2$.
Then we get the equality $h_i\circ f_{[\phi_1,\phi_2]}=f_{[\phi_i]}\circ h_i$ for $i\in\{1,2\}$.

\subsubsection{Experiment 2A}\label{sssec:exp_2a}
\paragraph{Setup}
For this experiment we consider the following time series:
\begin{align*}
    \cT_1&=     \varrho(f_{[\alpha,\beta]}, (0.0, 0.0), 2000),\qquad
    &\cS_1&=     \cT_1^{(1)},\\
    \cT_2&=     \varrho(f_{[1.1\alpha,\beta]}, (0.1, 0.0), 2000),\qquad
    &\cS_2&=     \cT_2^{(1)},\\
    \cT_3&=     \varrho(f_{[\beta,\beta]}, (0.1, 0.0), 2000),\qquad
    &\cS_3&=     \cT_3^{(1)},
\end{align*}
where $\alpha=\sqrt{2}/10$, $\beta=\sqrt{3}/10$, and $\cS_i=\cT_i^{(1)}$, $i=1,2,3$, is a time series obtained from the projection of the elements of $\cT_i$ onto the first coordinate.
When comparing $\cT_i$ with $\cT_j$ for $i,j\in\{1,2,3\}$ we use $h\equiv\id$.
When we compare $\cT_i$ versus $\cS_j$ we use $h(x, y)=x$, and for $\cS_i$ versus $\cT_j$ we get $h(x)=(x, 0)$.

\begin{table}
\centering
\begin{footnotesize}
\begin{tabular}{|c|c|c|c|c|c|} \hline
\backslashbox{method}{test}
            & \thead{$\cT_1$ vs. $\cS_1$\\} 
            & \thead{$\cT_1$ vs. $\cT_2$\\} 
            & \thead{$\cT_1$ vs. $\cS_2$\\} 
            & \thead{$\cT_1$ vs. $\cT_3$\\}
            & \thead{$\cT_1$ vs. $\cS_3$\\} \\ \hline
$\fnn$ (r=2) &
    \slashbox{0.0}{1.0} & \slashbox{1.0}{.978} & \slashbox{1.0}{1.0} & \slashbox{0.0}{1.0} & \slashbox{0.0}{1.0}  \\\hline
$\KNN$ (k=5) & 
    \slashbox{.042}{.617} & 
    \slashbox{.275}{.855} & 
    \slashbox{.514}{.690} & 
    \slashbox{.041}{.938} & 
    \slashbox{.041}{.938}  \\\hline
\thead{$\conjtest$\\(k=5, t=5)}  & 
    \slashbox{.001}{.270} &
    \slashbox{.149}{.148} & 
    \slashbox{.142}{.270} & 
    \slashbox{.451}{.322} & 
    \slashbox{.318}{.322}  \\\hline
\thead{$\conjtest^+$\\(k=5, t=5)}  & 
    \slashbox{.018}{.272} & 
    \slashbox{.154}{.158} & 
    \slashbox{.143}{.271} & 
    \slashbox{.458}{.325} & 
    \slashbox{.319}{.324}  \\\hline
\end{tabular}
\end{footnotesize}
\caption{
    Comparison of conjugacy measures for time series generated by the rotation on a torus. 
    The number in the upper left part of the cell corresponds to a comparison of the first time series vs. the second one, while the lower right number corresponds to the inverse comparison.
}
\label{tab:experiment_torus}
\end{table}

\paragraph{Results}
The asymmetry of results in the first column ($\cT_1$ vs. $\cS_1$) in Table~\ref{tab:experiment_torus} shows that all methods detect a semiconjugacy between $\cT_1$ and $\cS_1$, i.e. that $f_{[\alpha]}$ is semiconjugate to $f_{[\alpha,\beta]}$ via $h_1$.
An embedding of a torus into a 1-sphere preserves a neighborhood of a point.
The inverse map clearly does not exist.

The rest of the results confirm conclusions from the previous experiment.
The second and the third column ($\cT_1$ vs. $\cT_2$ and $\cT_1$ vs. $\cS_2$)
    show that $\fnn$ and $\KNN$ are sensitive to a perturbation of the system parameters.
The fourth and the fifth column ($\cT_1$ vs. $\cT_3$ and $\cT_1$ vs. $\cS_3$)
    present another example where those two methods produce a false positive answer suggesting a semiconjugacy.  
    This time the problematic case is not due to a doubling of the rotation parameter, but because of coinciding rotation angles.
Again, the behavior of the $\conjtest$ method exhibits a response that is relative to the level of perturbation. 

%%%%%%%%%%%%%%%%%%%%%%%%%%%%%%%%%%%%%%%%%%%%%%
%%%%%%%%%%%%%%%%%%%%%%%%%%%%%%%%%%%%%%%%%%%%%%
%%%%%%%%%%%%%%%%%%%%%%%%%%%%%%%%%%%%%%%%%%%%%%
\subsection{Example: the logistic map and the tent map}\label{ssec:log_tent}
Our next experiment examines two broadly studied chaotic maps defined on a real line.
The logistic map and the tent map, $f_l,g_\mu:[0,1]\rightarrow[0,1]$, respectively defined as:
\begin{equation}
    f_l(x) := l x(1-x) 
    \qquad\text{and}\qquad
    g_\mu(x):=\mu\min\{x,\, 1-x\},
\end{equation}
where, typically, $l\in[0, 4]$ and $\mu\in[0,2]$.
For parameters $l=4$ and $\mu=2$ the systems are conjugate via homeomorphism:
\begin{equation}\label{eq:log_tent_homoemorphism}
    h(x):=\frac{2\arcsin(\sqrt{x})}{\pi},
\end{equation}
that is, $h\circ f_4=g_2\circ h$.
In this example we use the standard metric induced from $\RR$.

\subsubsection{Experiment 3A}\label{sssec:exp_3a}
\paragraph{Setup}
In the initial experiment for those systems we compare the following time series:
\begin{align*}
    \cA&=\varrho(f_4, 0.2, 2000), 
    &\cB_2&=\varrho(f_{4}, 0.21, 2000),\\
    \cB_1&=\varrho(g_2, h(0.2), 2000),
    &\cB_3&=\varrho(f_{3.99}, 0.2, 2000), \\ 
    & &\cB_4&=\varrho(f_{3.99}, 0.21, 2000).  
\end{align*}

Time series $\cA$ is conjugate to $\cB_1$ through the homeomorphism $h$.
Time series $\cA$ and $\cB_2$ come from the same system -- $f_4$, but are generated using different starting points.
Sequences $\cB_3$ and $\cB_4$ are both generated by the logistic map but with different parameter value ($l={3.99}$) than $\cA$; thus, they are not conjugate with $\cA$.
For $\conjtest$ methods we use \eqref{eq:log_tent_homoemorphism} to compare $\cA$ with $\cB_1$, and the identity map to compare $\cA$ with $\cB_2$, $\cB_3$ and $\cB_4$.

\paragraph{Results}
The first column of Table \ref{tab:experiment_log_tent} shows that all methods properly identify the tent map as a system conjugate to the logistic map (provided that the two time series are generated by dynamically corresponding points, i.e. $a_1$ and $b_1:=h(a_1)$, respectively).
The second column demonstrates that $\fnn$ and $\KNN$ get confused by a perturbation of the starting point generating time series.
This effect was not present in the circle and the torus example (Sections \ref{ssec:circle_rotation} and \ref{ssec:torus_rotation}) due to a full symmetry in those examples.
The $\conjtest$ methods are only weakly affected by the perturbation of the starting point. 
Nevertheless, we expect that higher values of parameter $t$ may significantly affect the outcome of $\conjtest$ due to the chaotic nature of the map.  
We test it further in the context of Lorenz attractor (Experiment $\hyperref[sssec:exp_4c]{4C}$).
The third and the fourth column reflect high sensitivity of $\fnn$ and $\KNN$ to the parameter of the system.
On the other hand, $\conjtest$ methods admit rather conservative response to a change of the parameter.

The experiment shows that $\fnn$ and $\KNN$ are able to detect a change caused by a perturbation of a system immediately.
However, in the context of empirical data we may not be able to determine whether the starting point was perturbed, or if the system has actually changed, or whether there was a noise in our measurements.
Thus, some robustness with respect to noise might be desirable and the seemingly blurred concept of the conjugacy represented by $\conjtest$ might be helpful.

\begin{table}
\begin{center}
\begin{tabular}{|c|c|c|c|c|} \hline
\backslashbox{method}{test}
        & $\cA$ vs. $\cB_1$ 
            & \thead{$\cA$ vs. $\cB_2$\\ (starting point \\perturbation)} 
            & \thead{$\cA$ vs. $\cB_3$\\ (parameter\\ perturbation)} 
            & \thead{$\cA$ vs. $\cB_4$\\ (st.point + param.\\ perturbation)}\\\hline
$\fnn$ (r=2) &
    \slashbox{.205}{0.0} &
    \slashbox{.998}{1.0} &
    \slashbox{1.0}{1.0} &
    \slashbox{1.0}{.999} \\\hline
$\KNN$ (k=5) & 
    \slashbox{0.0}{0.0} &
    \slashbox{.825}{.828} &
    \slashbox{.831}{.832} &
    \slashbox{.835}{.833} \\\hline
\thead{$\conjtest$\\ (k=5, t=5)} & 
    \slashbox{0.0}{0.0} &
    \slashbox{.017}{.017} &
    \slashbox{.099}{.059} &
    \slashbox{.099}{.059} \\\hline
\thead{$\conjtest^+$\\ (k=5, t=5) }& 
    \slashbox{0.0}{.001} &
    \slashbox{.027}{.023} &
    \slashbox{.104}{.065} &
    \slashbox{.104}{.064} \\\hline
\end{tabular}
\caption{
    Comparison of conjugacy measures for time series generated by logistic and tent maps. 
    The number in the upper left part of the cell corresponds to a comparison of the first time series vs. the second one, while the lower right number corresponds to the inverse comparison.
}
\label{tab:experiment_log_tent}
\end{center}
\end{table}

\subsubsection{Experiment 3B}\label{sssec:exp_3b}

The logistic map is one of the standard examples of chaotic maps.
Thus, we expect that the behavior of the system will change significantly if we modify the parameter $l$. 
Here, we examine how the perturbation of $l$ affects the outcome of tested methods.

\paragraph{Setup}
We generated a collection of time series:
\[
    \left\{\cB(l) := \varrho(f_{l}, 0.2, 2000) \mid l\in\{3.8, 3.805, 3.81,\ldots, 4.0\} \right\}. 
\]
Every time series $\cB(l)$ in the collection was compared with a reference time series $\cB(4.0)$.

\paragraph{Results}
The results are plotted in Figure \ref{fig:log_param_grid}. 
As Experiment $\hyperref[sssec:exp_3a]{\text{3A}}$ suggested, $\fnn$ and $\KNN$ quickly saturate, providing almost ``binary'' response i.e. the output value is either 0 or a fixed non-zero number depending on  parameter $k$. Similarly to Experiment $\hyperref[sssec:exp_1b]{\text{1B}}$ we observe that with a higher parameter $k$ the curve corresponding to $\KNN$ gets more similar to $\fnn$ and becomes nearly a step function.
$\conjtest$ admits approximately continuous dependence on the value of the parameter of the system. 
However, higher values of the parameter $t$ of $\conjtest$ make the curve more steep and forms a significant step down in the vicinity of $l=4$.
This makes sense, because the more time-steps forward we take into account the more nonlinearity of the system affects the tested neighborhood.

We presume that the observed drop of $\fnn$ values and increase of $\conjtest$ values for the parameter $l$ value approximately in the interval $\{3.83, 3.86\}$ is caused by the collapse of the attractor to the $3$-periodic orbit observed for these parameter values
(see bifurcation diagram in Figure \ref{fig:logistic_clustering}).

Obviously, the logistic map with different parameter values won't be conjugate.
However, since we work with only finite samples, it might be not enough to rigorously distinguish them if the difference of the parameter is small.
The results can only suggest an empirical similarity of the underlying dynamical systems.

\begin{figure}
    \centering
    \includegraphics[width=0.8\textwidth]{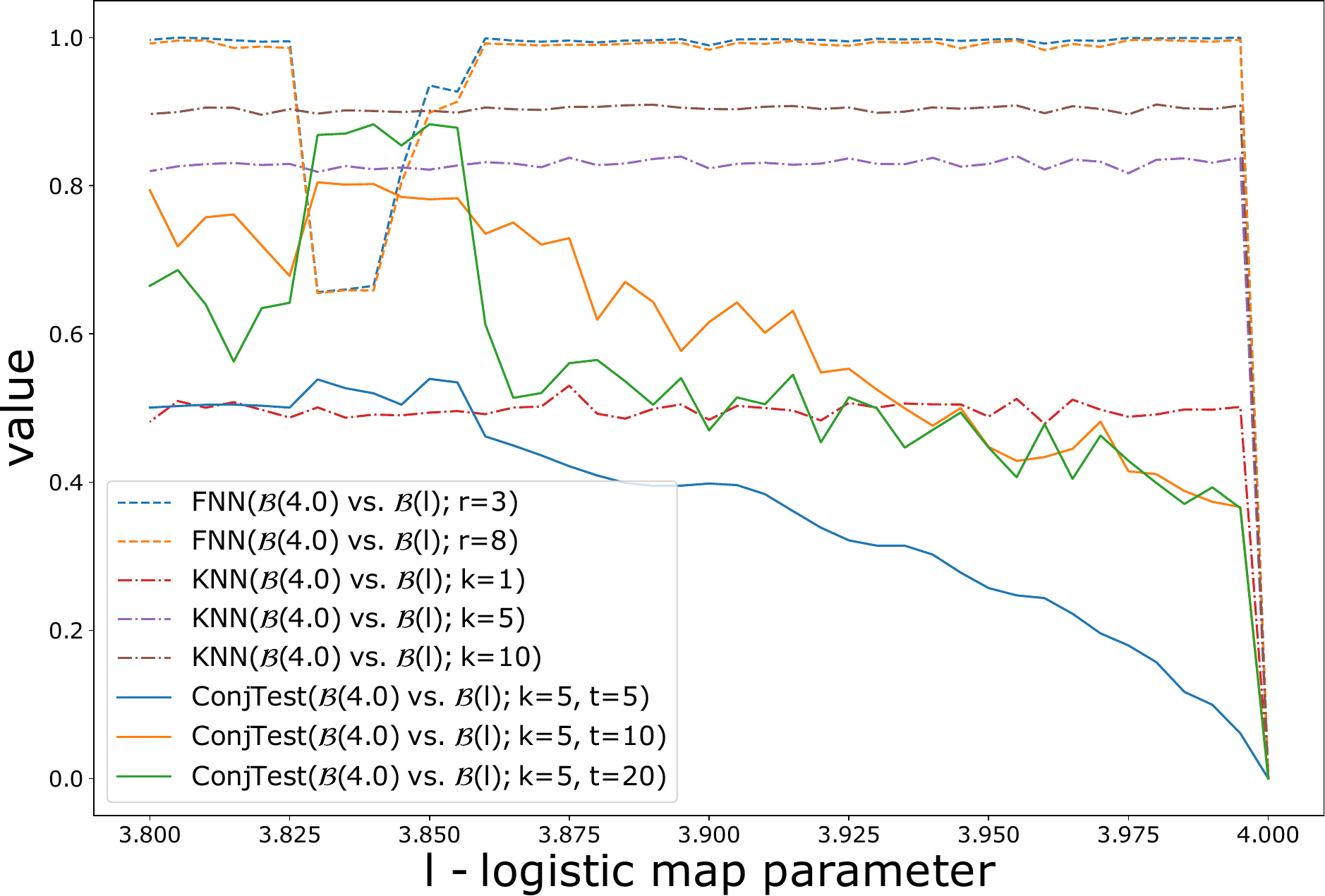}
    \caption{Dependence of the conjugacy measures on a change of the parameter of the logistic map. }
    \label{fig:log_param_grid}
\end{figure}

\subsubsection{Experiment 3C}\label{sssec:exp_3c}
As observed in the previous experiment, a change of the parameter $l$ in the logistic equation may significantly change the dynamical nature of the system.
In this experiment we use the $\conjtest$ to grasp the types of dynamics as a function of $l$ parameter.

\paragraph{Setup}
First, we generated the following collection of time series:
\begin{equation}\label{eq:bifurcation_experiment_time_series}
    \left\{\cB(l, p) := \varrho(f_{l}, f_{l}^{500}(p), 2000) \mid l\in\{3.4, 3.405, 3.41,\ldots, 4.0\},\ p\in\{0.11, 0.31\} \right\}. 
\end{equation}
Note that each time series starts at $500$-th iterate of point $p$.
It is a standard procedure allowing the trajectory to settle down on an attractor. 
We compared every two time series $\cB(l, p)$ and $\cB(l', p')$ with the formula
We assign a similarity for each pair of time series $\cB(l, p)$ and $\cB(l', p')$ via
\begin{equation}\label{eq:similarity_between_time_series}
        \max\left\{
            \conjtest(\cB(l, p), \cB(l', p'); k, t, \id), \conjtest(\cB(l', p'), \cB(l, p); k, t, \id)
        \right\}
\end{equation}
with fixed $k=5$ and $t=2$.
The obtained similarity matrix was then applied to the single-linkage agglomerative hierarchical clustering. 

\paragraph{Results}
The dendrogram in Figure \ref{fig:logistic_dendrogram} presents the output of the experiment.
Every leaf represents a single time series corresponding to a pair $(l, p)$ of the parameter value and a starting point.
With a threshold value $0.35$ we can distinguish $10$ clusters.
For every value of parameter $l$, both time series $\cB(l, 0.11)$ and $\cB(l, 0.31)$ fall into the same cluster.

We draw the result of the clustering on the bifurcation diagram in Figure \ref{fig:logistic_clustering}.
As one can expect, the time series grouped according to their dynamics type and their proximity in the parameter space.
The dendrogram structure indicates additional substructures within the clusters. 
For instance, the pink cluster contains two visible subclasses from which one corresponds to a set of 4-periodic orbit, while the second aggregates the attractors after the initial period doubling bifurcations.

\begin{figure}
    \centering
    \includegraphics[width=1.0\textwidth]{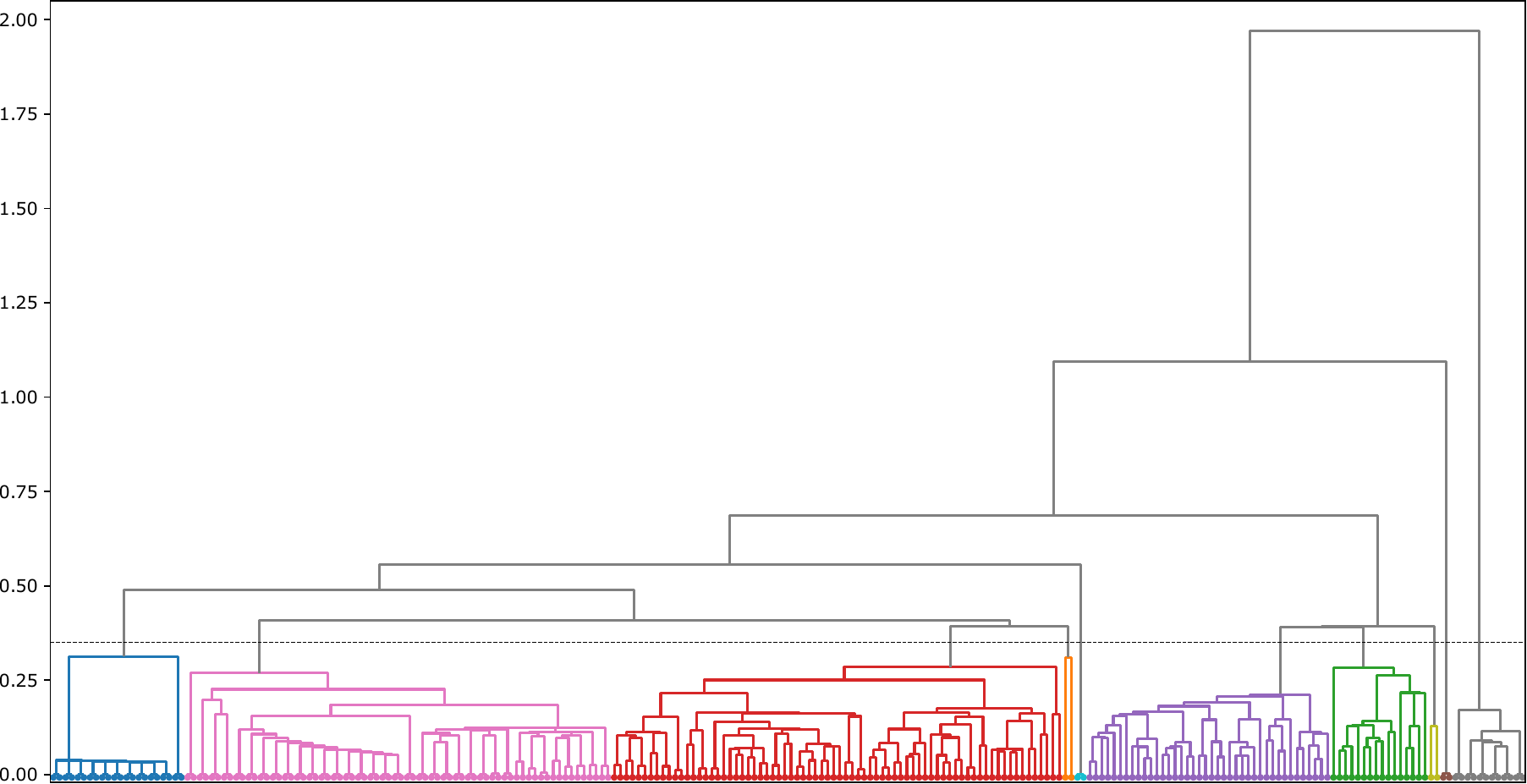}
    \caption{
        Dendrogram obtained from the single-linkage agglomerative hierarchical clustering of the collection of time series \eqref{eq:bifurcation_experiment_time_series} generated from the logistic map using similarity score defined \eqref{eq:similarity_between_time_series}.
        The horizontal dashed line represents the threshold chosen for the clustering.
        The distribution of the clustered time series on bifurcation diagram is presented in Figure~\ref{fig:logistic_clustering}.
    }
    \label{fig:logistic_dendrogram}
\end{figure}

\begin{figure}
    \centering
    \includegraphics[width=0.8\textwidth]{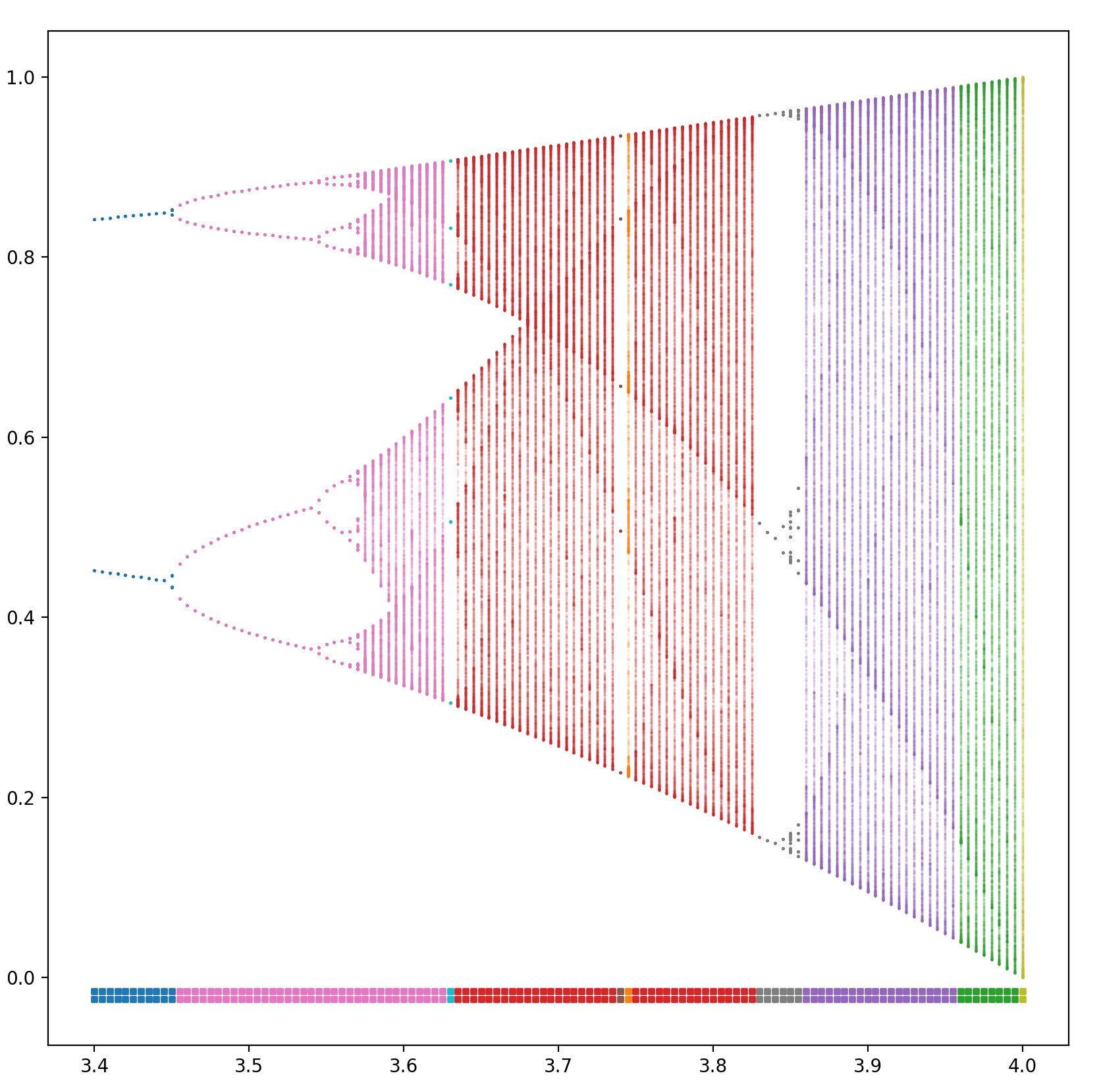}
    \caption{
        Partition of the bifurcation diagram based on the clustering of time series generated from the logistic map presented in Figure~\ref{fig:logistic_dendrogram}.
        Every square at the bottom of the image represents a time series corresponding to a value of parameter $l$ given by the horizontal axis. Top row corresponds to the starting point $0.11$, bottom row to $0.31$.
        Color of a square indicates the cluster into which the corresponding trajectory belongs.
    }
    \label{fig:logistic_clustering}
\end{figure}

%%%%%%%%%%%%%%%%%%%%%%%%%%%%%%%%%%%%%%%%%%%%%%
%%%%%%%%%%%%%%%%%%%%%%%%%%%%%%%%%%%%%%%%%%%%%%
%%%%%%%%%%%%%%%%%%%%%%%%%%%%%%%%%%%%%%%%%%%%%%
\subsection{Example: Lorenz attractor and its embeddings}\label{sec:lorenz_experiment}
The fourth example is based on the Lorenz system defined by equations:
\begin{equation}\label{eq:Lorenz_system}
   \begin{cases}
        \dot{x}&=\sigma(y-x),\\
        \dot{y}&=x(\rho-z)-y,\\
        \dot{z}&=xy-\beta z,
   \end{cases} 
\end{equation}
which induces a continuous dynamical system $\varphi:\RR^3\times\RR\rightarrow\RR^3$.
We consider the classical values of the parameters: $\sigma=10$, $\rho=28$, and $\beta=8/3$.
A time series can be generated by iterates of the map $f(x):=\varphi(x,\tilde{t})$, where $\tilde{t}>0$ is a fixed value of the time parameter.
For the following experiments we chose $\tilde{t}=0.02$ and we use the Runge-Kutta method of an order 5(4) to generate the time series. 

\subsubsection{Experiment 4A}\label{sssec:exp_4a}
\paragraph{Setup}
Let $p_1=(1,1,1)$ and $p_2=(2,1,1)$.
In this experiment we compare the following time series:
\begin{equation*}
    \cL_i=\varrho(f, f^{2000}(p_i), 10000),\ \
    \cP^i_{x,d}=\Pi(\proj{\cL_i}{1}, d, 5),\ \
    \cP^i_{z,d}=\Pi(\proj{\cL_i}{3}, d, 5),
\end{equation*}
where  $i\in\{1,2\}$.
Recall that $\Pi$ denotes the embedding of a time series into $\RR^d$ and $\cL_i^{j}$ is a projection of time series $\cL_i$ onto its $j$-th coordinate. In all the embeddings we choose the lag $l=5$.
Note that the first point of time series $\cL_i$ is equal to the $2000$-th iterate of point $p_i$ under map $f$.
It is a standard procedure to cut off some transient part of the time series. 

Time series $\cP^i_{x,d}$ and  $\cP^i_{z,d}$ are embeddings of the first and third coordinate of $\cL_i$, respectively.
As Figure \ref{fig:lorenz_time_series} (top right) suggests, the embedding of the first coordinate into $\RR^3$ results in a structure topologically similar to the Lorenz attractor.
The embedding of the third coordinate, due to the symmetry of the system, produces a collapsed structure with ``wings'' of the attractor glued together (Figure \ref{fig:lorenz_time_series}, right).
Thus, we expect time series $\cP^i_{z,d}$ to be recognized as non-conjugate to $\cL_i$. 

\begin{figure}
    \centering
    \includegraphics[width=0.35\textwidth]{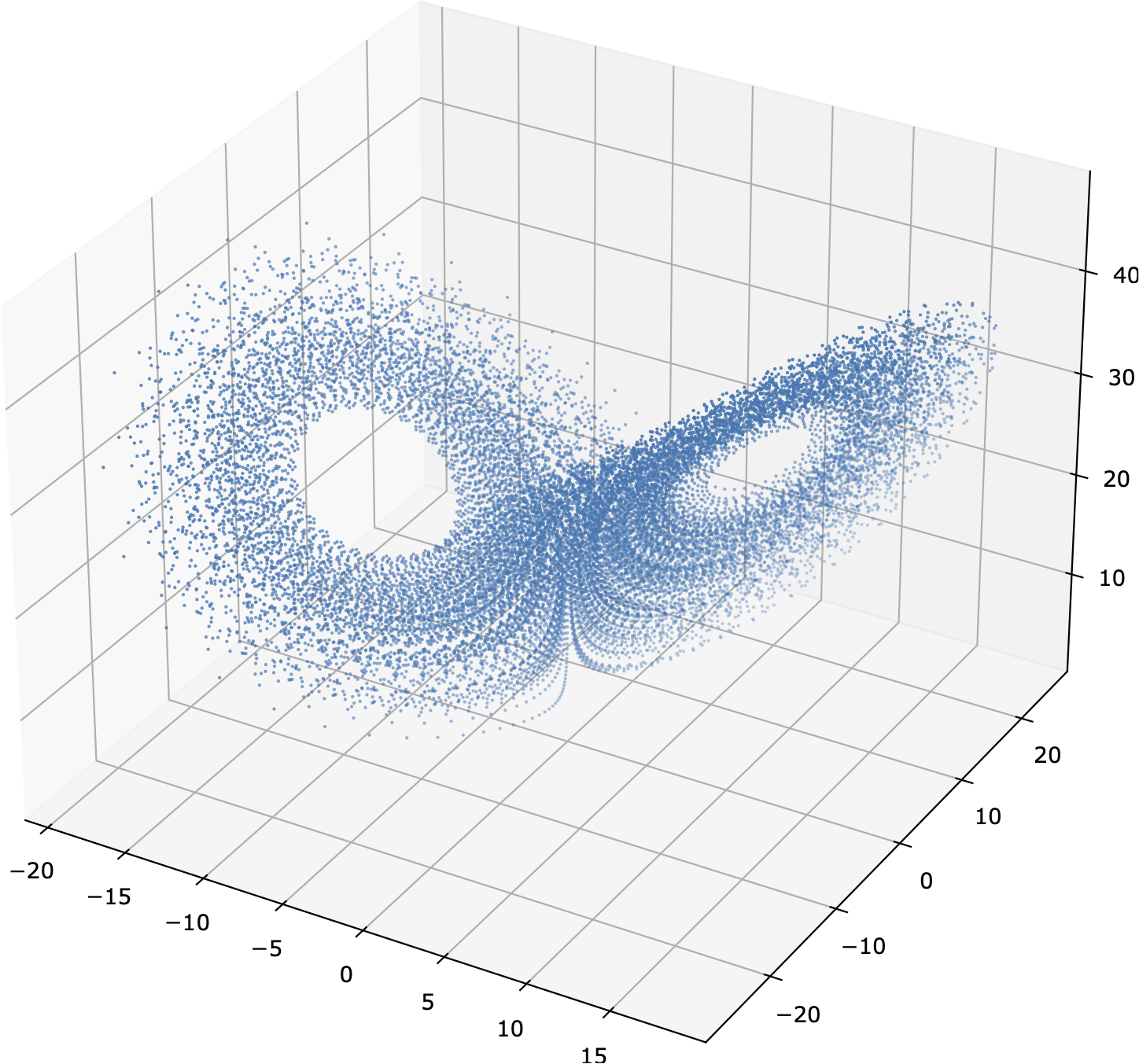}
    % \hfill
    \hspace{0.5cm}
    \includegraphics[width=0.35\textwidth]{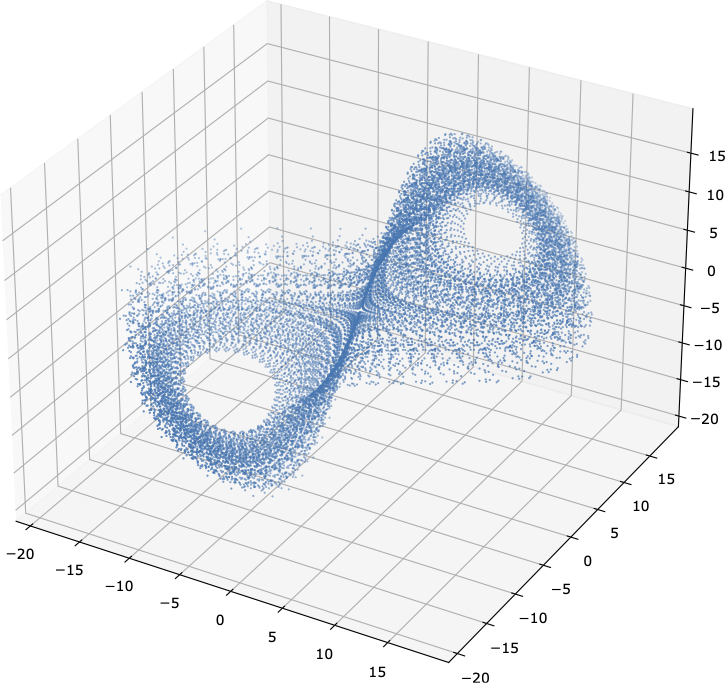}
    
    \includegraphics[width=0.35\textwidth]{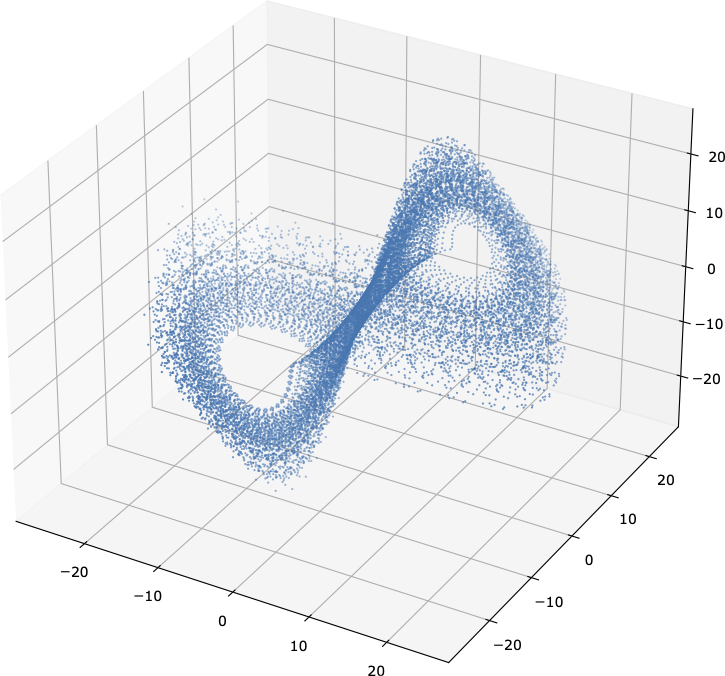}
    \hspace{0.5cm}
    % \hfill
    \includegraphics[width=0.35\textwidth]{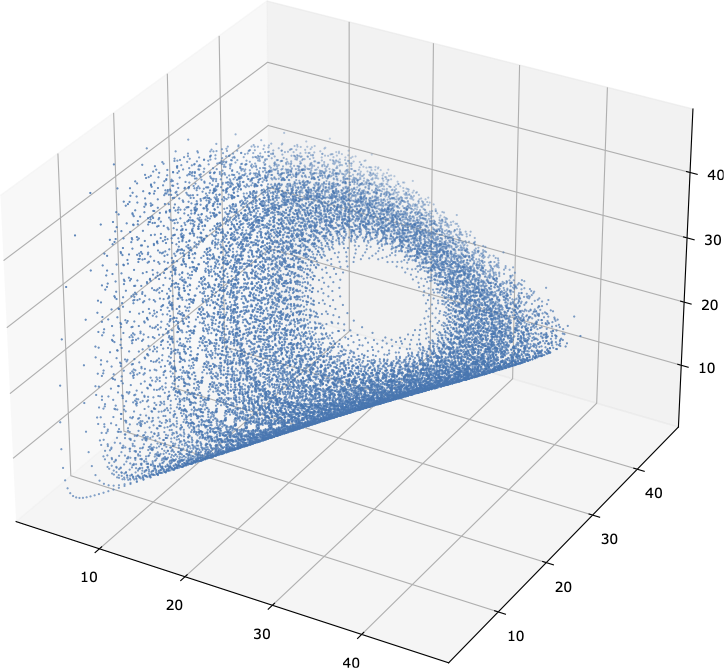}
    \caption{A time series generated from the Lorenz system (top left) and 3-d embeddings of its projections onto the $x$-coordinate (top right), $y$-coordinate (bottom left) and $z$-coordinate (bottom right) with a delay parameter $l=5$.}
    \label{fig:lorenz_time_series}
\end{figure}

In order to compare $\cL_i$ and embedded time series with $\conjtest$ we shall find the suitable map $h$. Ideally, such a map should be a homeomorphism between the Lorenz attractor $L\subset \RR^3$ (or precisely the $\omega$-limit set of the corresponding initial condition under the system \eqref{eq:Lorenz_system}) and its image $h(L)$. However, the construction of the time series allows us to easily define the best candidate for such a correspondence map pointwise for all elements of the time series. 

For instance, a local approximation of $h$ when comparing $\cL_i\subset \RR^3$ and $P^j_{x,d}\subset \RR^d$ will be given by:
\begin{equation}\label{eq:lorenz_h}
    h:\cL_i\ni \mathbf{x}_t \mapsto (\mathbf{x}^{(1)}_t, \mathbf{x}^{(1)}_{t+5}, \ldots, \mathbf{x}^{(1)}_{t+5d})\in P^i_{x,d}\subseteq \RR^d,
\end{equation}
where $\mathbf{x}_t:=(x_t, y_t, z_t)\in \RR^3$ denotes the state of the system \eqref{eq:Lorenz_system} at time $t$ and $\mathbf{x}^{(1)}_t=x_t$ denotes its projection onto the $x$-coordinate. When $j=i$ this formula matches the points of $\cL_i$  to the corresponding points of $P^j_{x,d}=P^i_{x,d}$. However, if $j\neq i$ then the points of $\cL_i$ are mapped onto the points of $P^i_{x,d}$, not  to $P^j_{x,d}$, thus in fact in our comparison tests we verify how well $P^j_{x,d}$ approximates the image of $\cL_i$ under $h$ and the original dynamics. 

For the symmetric comparison of $P^j_{x,d}$ with $\cL_i$ the local approximation of $h^{-1}$ $:\RR^d\rightarrow\RR^3$ will take a form:
\begin{equation}\label{eq:lorenz_h-1}
    h^{-1}:P^j_{x,d}\ni(\mathbf{x}^{(1)}_t, \mathbf{x}^{(1)}_{t+5}, \ldots, \mathbf{x}^{(1)}_{t+5d}) \mapsto \mathbf{x}_t\in\cL_j\subset\RR^3.
\end{equation}
These are naive and data driven approximations of the potential connecting map $h$.
In particular the homeomorphism from the Lorenz attractor to its 1D embedding cannot exist, but we still can construct map $h$ using the above receipt, which seems natural and the best candidate for such a comparison. More sophisticated ways of finding the optimal $h$ in general situations will be the subject of our future studies.

In the experiments below we use the maximum metric.

\paragraph{Results}
As one can expect, Table \ref{tab:experiment_lorenz_L1_to_P1_dmax} shows that embeddings of the first coordinate give in general noticeably lower values then embeddings of the $z$'th coordinate.  
Thus, suggesting that $\cL_1$ is conjugate to $\cP_{x,3}^1$, but not to $\cP_{z,3}^1$.
Again, Table \ref{tab:experiment_lorenz_L1_to_P2_dmax} shows that, in the case of chaotic systems, $\fnn$ and $\KNN$ are highly sensitive to variation in starting points of the series.

All methods suggest that 2-d embedding of the $x$-coordinate has structure reasonably compatible with $\cL_1$.
With the additional dimension values gets only slightly lower. 
One could expect that 3 dimensions would be necessary for an accurate reconstruction of the attractor. Note that Takens' Embedding Theorem suggests even dimension of $5$, as the Hausdorff dimension of the Lorenz attractor is about $2.06$ \cite{Viswanath2004}. However, it often turns out that the dynamics can be reconstructed with the embedding dimension less than given by Takens' Embedding Theorem (as implied e.g. by Probabilistic Takens' Embedding Theorem, see \cite{Baranski_2020,Baranski_2022}). 
We also attribute our outcome to the observation that the $x$-coordinate carries a large piece of the system information, which is visually presented in Figure \ref{fig:lorenz_time_series}.

Interestingly, when we use $\conjtest$ to compare $\cL_1$ with embedding time series generated from $\cL_1$ we always get values $0.0$.
The connecting maps used in this experiment, defined by \eqref{eq:lorenz_h} and \eqref{eq:lorenz_h-1}, establish a direct correspondence between points in two time series.
As a result we get $\tilde{h}=h$ in the definition of $\conjtest$, and consequently, every pair of sets in the numerator of equation \eqref{eq:conjtest} is the same.
If the embedded time series comes from another trajectory then $\tilde{h}\neq h$ and $\conjtest$ gives the expected results,  as visible in Table \ref{tab:experiment_lorenz_L1_to_P2_dmax}.
On the other hand, computationally more demanding $\conjtest^+$ exhibits virtually the same results in both cases, when $\cL_1$ is compared with embeddings of its own (Table \ref{tab:experiment_lorenz_L1_to_P1_dmax}) and when $\cL_1$ is compared with embeddings of $\cL_2$ (Table \ref{tab:experiment_lorenz_L1_to_P2_dmax}).  

\begin{table}
\centering
\begin{footnotesize}
\begin{tabular}{|c|c|c|c|c|c|c|} \hline
\backslashbox{method}{test}
            & \thead{$\cL_1$ vs. $\cP^1_{x,1}$} 
            & \thead{$\cL_1$ vs. $\cP^1_{x,2}$} 
            & \thead{$\cL_1$ vs. $\cP^1_{x,3}$} 
            & \thead{$\cL_1$ vs. $\cP^1_{z,1}$}
            & \thead{$\cL_1$ vs. $\cP^1_{z,3}$} \\ \hline
$\fnn$ (r=3)  & 
    \slashbox{0.0}{1.0} & 
    \slashbox{0.0}{.362} & 
    \slashbox{.05}{.196} & 
    \slashbox{0.0}{1.0} & 
    \slashbox{.111}{.541}  \\\hline
$\KNN$ (k=5)  & 
    \slashbox{.019}{.465} & 
    \slashbox{.003}{.036} & 
    \slashbox{.003}{.007} & 
    \slashbox{.024}{.743} & 
    \slashbox{.002}{.519}  \\\hline
\thead{$\conjtest$\\(k=5, t=10)}  & 
    \slashbox{0.0}{0.0} & 
    \slashbox{0.0}{0.0} & 
    \slashbox{0.0}{0.0} & 
    \slashbox{0.0}{0.0} & 
    \slashbox{0.0}{0.0}  \\\hline
\thead{$\conjtest^+$\\(k=5, t=10)}  & 
    \slashbox{.330}{.401} & 
    \slashbox{.030}{.087} & 
    \slashbox{.024}{.051} & 
    \slashbox{.406}{.396} & 
    \slashbox{.046}{.407}  \\\hline
\end{tabular}
\end{footnotesize}
\caption{
    Comparison of conjugacy measures for time series generated by the Lorenz system. 
    The number in the upper left part of the cell corresponds to a comparison of $\cL_1$ vs. the second time series, while the lower right number corresponds to the inverse comparison.
}
\label{tab:experiment_lorenz_L1_to_P1_dmax}
\end{table}

\begin{table}
\centering
\begin{footnotesize}
\begin{tabular}{|c|c|c|c|c|c|c|} \hline
\backslashbox{method}{test}
            & \thead{$\cL_1$ vs. $\cL_2$} 
            & \thead{$\cL_1$ vs. $\cP^2_{x,1}$} 
            & \thead{$\cL_1$ vs. $\cP^2_{x,2}$} 
            & \thead{$\cL_1$ vs. $\cP^2_{x,3}$} 
            & \thead{$\cL_1$ vs. $\cP^2_{z,1}$}
            & \thead{$\cL_1$ vs. $\cP^2_{z,3}$} \\ \hline
$\fnn$ (r=3)  & 
    \slashbox{.995}{.996} & 
    \slashbox{.955}{1.0} & 
    \slashbox{.987}{.996} & 
    \slashbox{.991}{.996} & 
    \slashbox{.963}{1.0} & 
    \slashbox{.996}{997.}  \\\hline
$\KNN$ (k=5)  & 
    \slashbox{.822}{.827} & 
    \slashbox{.826}{.832} & 
    \slashbox{.829}{.826} & 
    \slashbox{.829}{.825} & 
    \slashbox{.823}{.828} & 
    \slashbox{.820}{.833}  \\\hline
\thead{$\conjtest$\\(k=5, t=10)}  & 
    \slashbox{.010}{.009} & 
    \slashbox{.236}{.012} & 
    \slashbox{.016}{.010} & 
    \slashbox{.010}{.009} & 
    \slashbox{.391}{.012} & 
    \slashbox{.017}{.009}  \\\hline
\thead{$\conjtest^+$\\(k=5, t=10)}  & 
    \slashbox{.020}{.017} & 
    \slashbox{.331}{.400} & 
    \slashbox{.039}{.092} & 
    \slashbox{.033}{.056} & 
    \slashbox{.431}{.392} & 
    \slashbox{.060}{.404}  \\\hline
\end{tabular}
\end{footnotesize}
\caption{
    Comparison of conjugacy measures for time series generated by the Lorenz system. 
    The number in the upper left part of the cell corresponds to a comparison of $\cL_1$ vs. the second time series, the lower right corresponds to the symmetric comparison.
}
\label{tab:experiment_lorenz_L1_to_P2_dmax}
\end{table}

% Experiment embedding
\subsubsection{Experiment 4B}\label{sssec:exp_4b}
This experiment is proceeded according to the standard use of $\fnn$ for estimating optimal embedding dimension without an explicit knowledge about the original system.

\paragraph{Setup}
Let $p=(1,1,1)$, we generate the following collection of time series
\begin{equation*}
    \cL=\varrho(f, f^{2000}(p), 10000),\quad
    \cP_{d}=\Pi(\proj{\cL}{1}, d, 5),
\end{equation*}
where $d\in\{1,2,3,4,5,6\}$.
In the experiment we compare pairs of embedded time series corresponding to consecutive dimensions, e.g., $\cP_d$ with $\cP_{d+1}$, for the entire range of parameter values.
We are looking for the minimal value of $d$ such that $\cP_{d-1}$ is dynamically different from $\cP_d$, but $\cP_d$ is similar to $\cP_{d+1}$.
The interpretation says that $d$ is optimal, because by passing from $d-1$ to $d$ we split some false neighborhoods apart (hence, dissimilarity of dynamics), but by passing from $d$ to $d+1$ there is no difference, because there is no false neighborhood left to be separated.

\paragraph{Results}
Results are presented in Figure \ref{fig:lorenz_embedding}. 
In general, the outputs of all methods are consistent.
When the one-dimensional embedding, $\cP_1$, is compared with two-dimensional embedding, $\cP_2$, we get large comparison values for the entire range of every parameter.
When we compare $\cP_2$ with $\cP_3$ the estimation of dissimilarity drops significantly, i.e. we conclude that the time series $\cP_2$ and $\cP_3$ are more ``similar'' than $\cP_1$ and $\cP_2$.
The comparison of $\cP_3$ with $\cP_4$ still decreases the values, suggesting that the third dimension improves the quality of our embedding.
The curve corresponding to $\cP_4$ vs. $\cP_5$ essentially overlaps $\cP_3$ vs. $\cP_4$ curve.
Thus, the third dimension seems to be a reasonable choice.

We can see that $\fnn$ (Figure \ref{fig:lorenz_embedding} top left), originally designed for this test, gives a clear answer.
However, in the case of $\KNN$ (Figure \ref{fig:lorenz_embedding} top right) the difference between the yellow and the green curve is rather subtle.
Thus, the outcome could be alternatively interpreted with a claim that two dimensions are enough for this embedding.
In the case of $\conjtest^+$ we have two parameters.
For the fixed value of $t=10$ we manipulated the value of $k$ (Figure \ref{fig:lorenz_embedding} bottom left) and the outcome matched up with the $\fnn$ result.
However, the situation is slightly different when we fix $k=5$ and vary the $t$ (Figure \ref{fig:lorenz_embedding} bottom right).
For $t<30$ the results suggest dimension $3$ to be optimal for the embedding, but for $t>40$ the green and the red curve split.
Moreover, for $t>70$, we can observe the beginning of another split of the red ($\cP_4$ vs. $\cP_5$) and the violet ($\cP_5$ vs. $\cP_6$) curves.
Hence, the answer is not fully conclusive.
We attribute this effect to the chaotic nature of the attractor.
The higher the value of $t$ the higher the effect.
We investigate it further in the following experiment.

\begin{figure}
    \centering
    \includegraphics[width=0.49\textwidth]{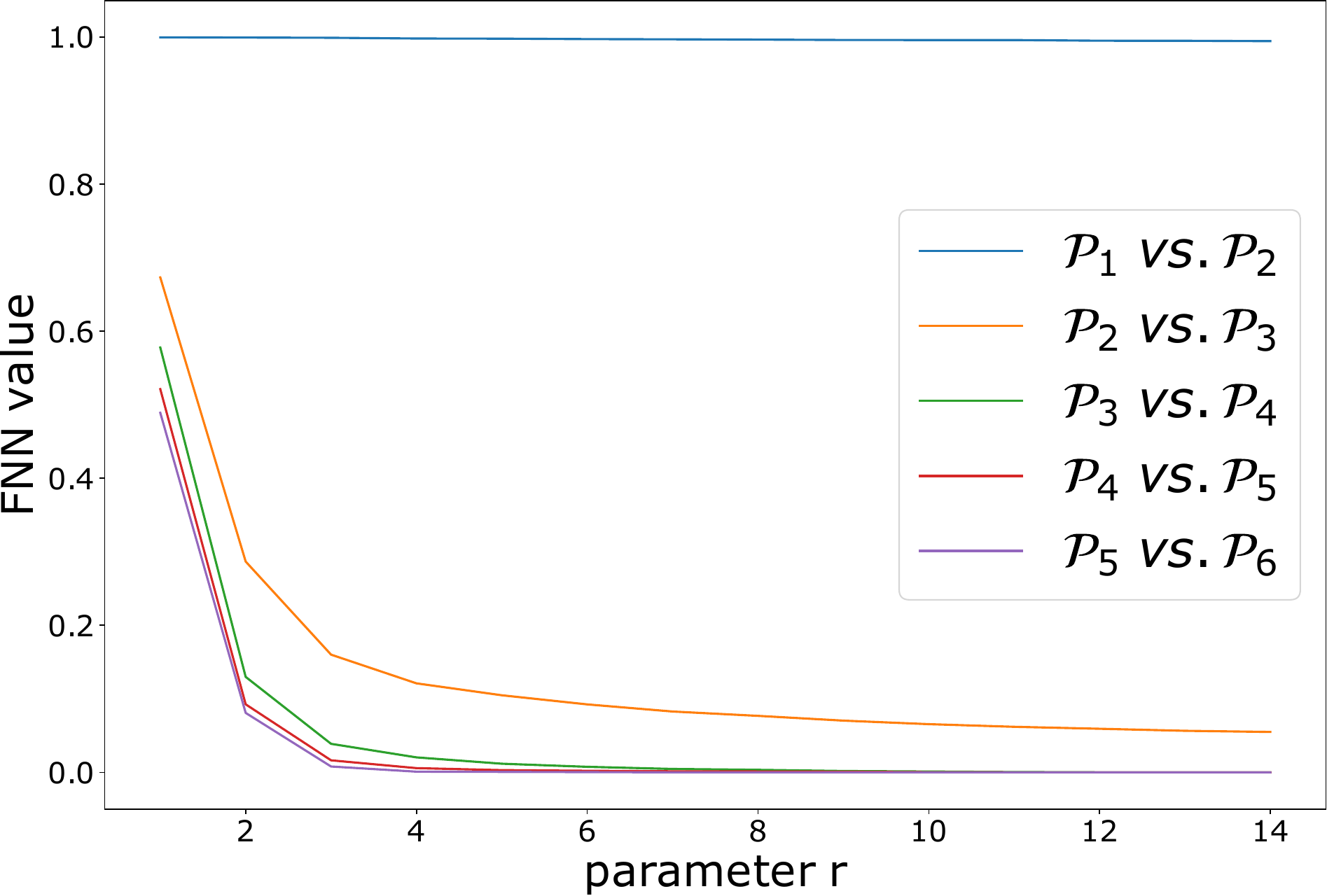}
    \hfill
    \includegraphics[width=0.49\textwidth]{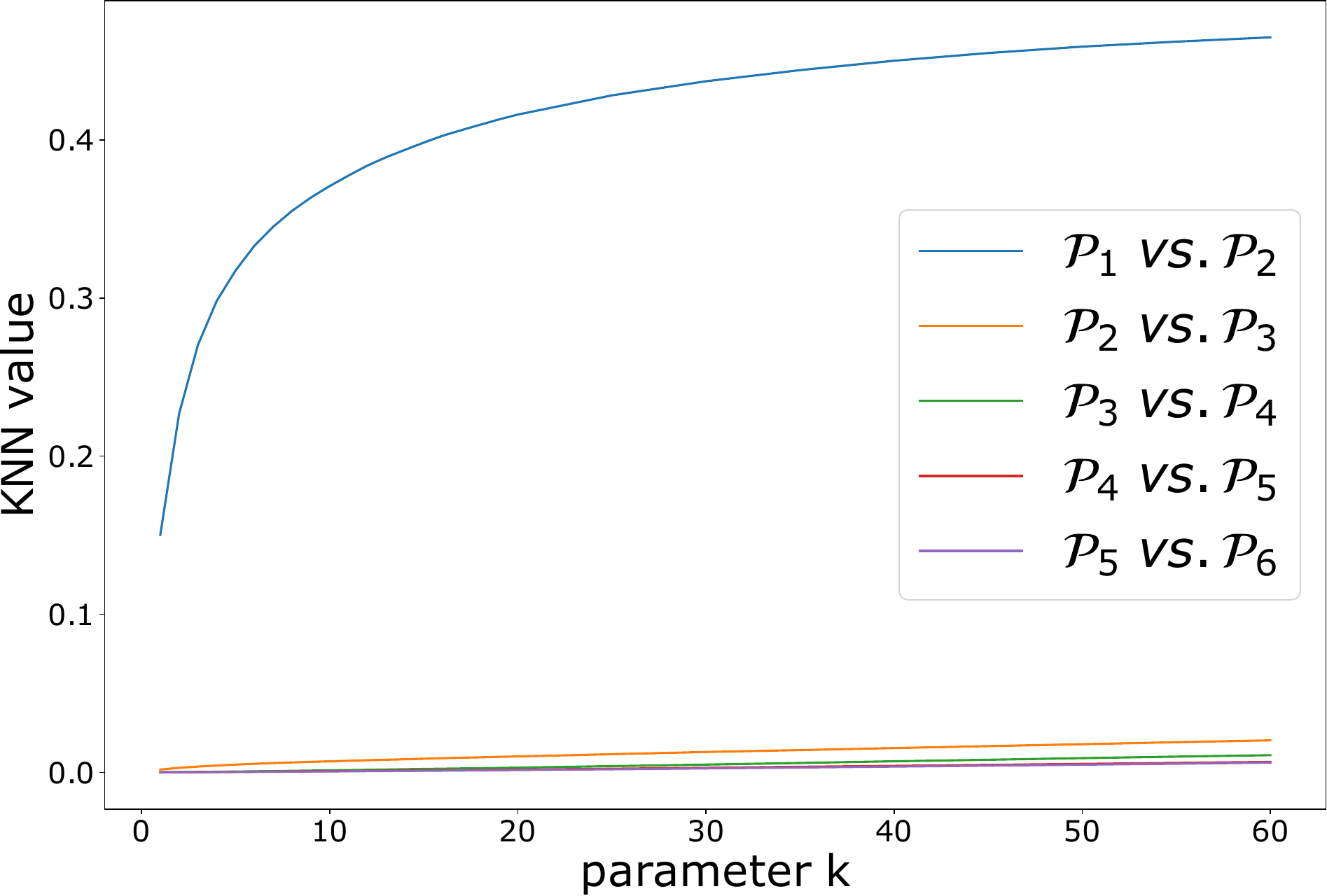}
    \vspace{0.5cm}
    
    \includegraphics[width=0.49\textwidth]{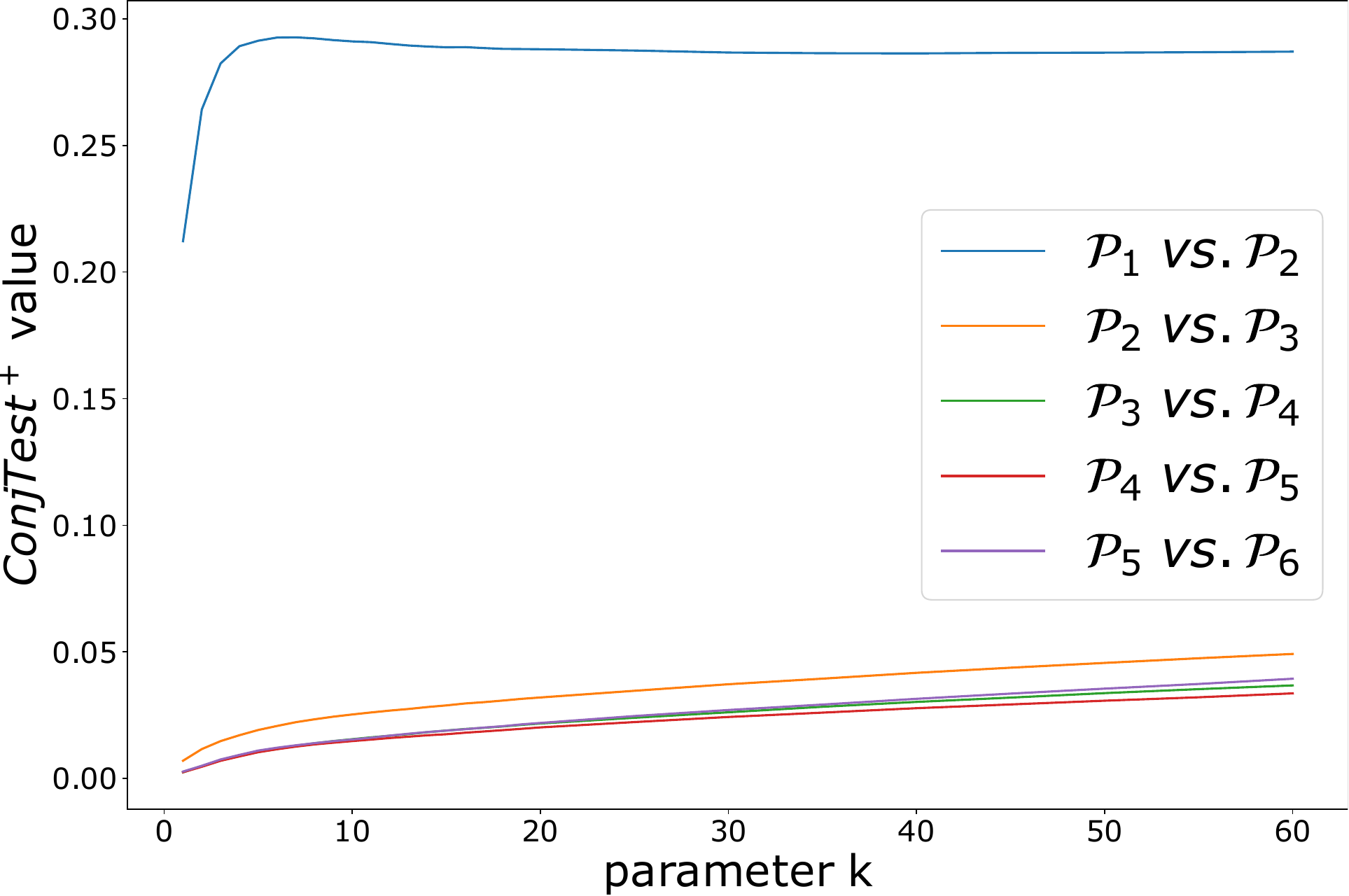}
    \hfill
    \includegraphics[width=0.49\textwidth]{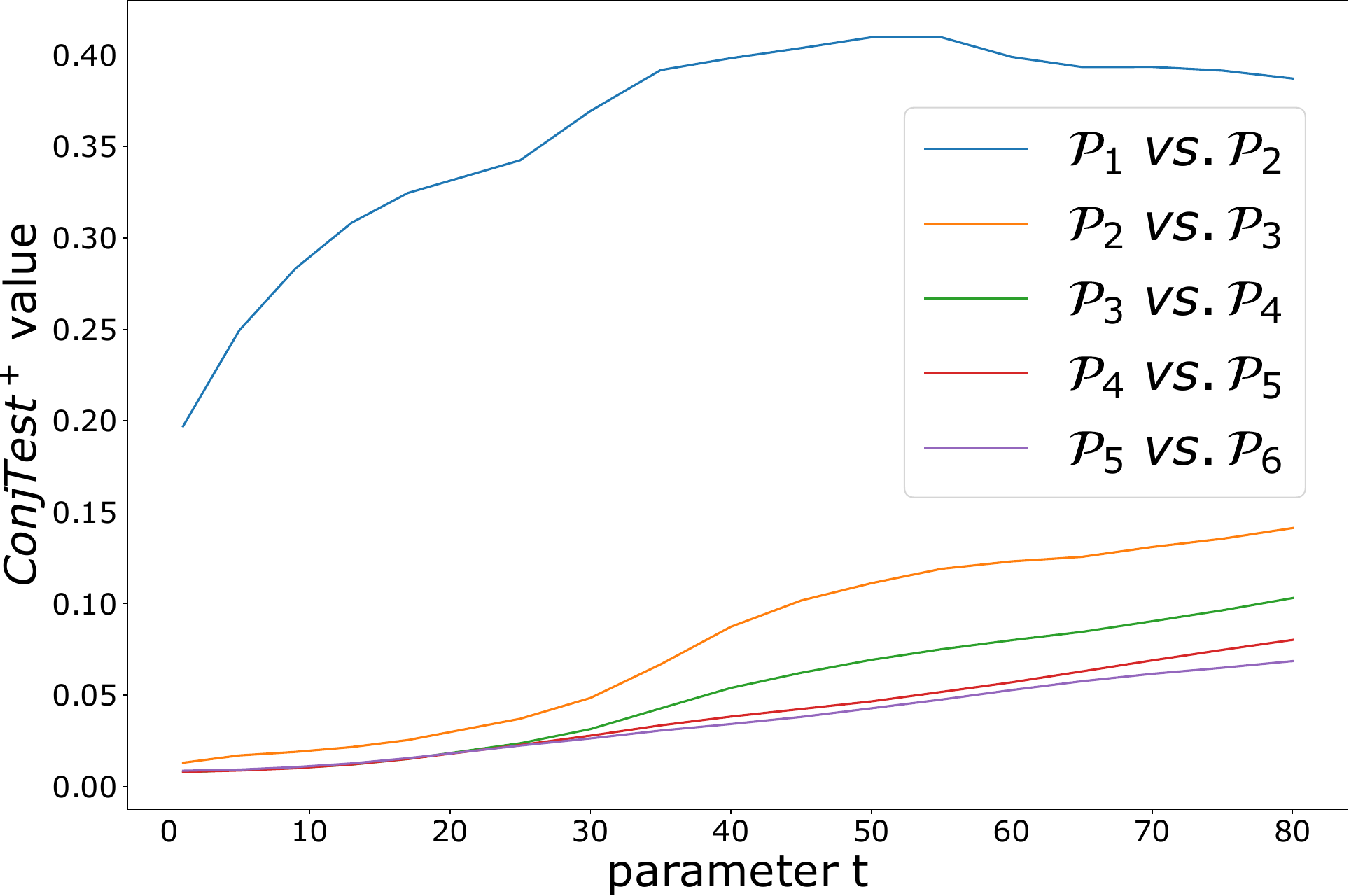}
    
    \caption{A comparison of embeddings for consecutive dimensions.
    Top left: $\fnn$ with respect to parameter $r$.
    Top right: $\KNN$ with respect to parameter $k$.
    Bottom left:  $\conjtest^+$ with respect to parameter $k$.
    Bottom right: $\conjtest^+$ with respect to parameter $t$.
    }
    \label{fig:lorenz_embedding}
\end{figure}

\subsubsection{Experiment 4C}\label{sssec:exp_4c}
In this experiment we investigate the dependence of $\conjtest^+$ on the choice of value of parameter $t$.
Parameter $t$ of $\conjtest^+$ controls how far we push the approximation of a neighborhood of a point $x_i$ ($U_k^i$ in \eqref{eq:conjtest_plus}) through the dynamics.
In the case of systems with a sensitive dependence on initial conditions (e.g., the Lorenz system) we could expect that higher values of $t$  spread the neighborhood over the attractor.
As a  consequence, we obtain higher values of $\conjtest^+$.

\paragraph{Setup}
Let $p_1=(1,1,1)$, $p_2=(2,1,1)$, $p_3=(1,2,1)$, and $p_4=(1,1,2)$.
In this experiment we study the following time series:
\begin{equation*}
    \cL_i=\varrho(f, f^{2000}(p_i), 10000),\quad
    \cP^i_{x,d}=\Pi(\proj{\cL_i}{1}, d, 5),\quad
    \cP^i_{y,d}=\Pi(\proj{\cL_i}{2}, d, 5),
\end{equation*}
where $i\in\{1,2,3,4\}$ and $d\in\{1,2,3,4\}$. 
We compare the reference time series $\cL_1$ with all the others using $\conjtest^+$ method with the range of parameter 
\[t\in\{1,5,9,13,17,21,25,30,35,40,45,50,55,60,65,70,75,80\}.\]

\paragraph{Results}
The top plot of Figure \ref{fig:lorenz_conj_time_grid} presents the results of comparing $\cL_1$ to the time series $\cL_i$ and $\cP^i_{x,d}$ with $i\in\{2,3,4\}$ and $d\in\{1,2,3,4\}$.
Red curves correspond to $\cL_1$ vs. $\cP^i_{x,1}$, green curves to $\cL_1$ vs. $\cP^i_{x,2}$, blue curves to $\cL_1$ vs. $\cP^i_{x,3}$, and  dark yellow curves to $\cL_1$ vs. $\cP^i_{x,4}$.
There are three curves of every color, each one corresponds to a different starting point $p_i$, $i\in\{2,3,4\}$. 
The bottom part shows results for comparison of $\cL_1$ to $\cP^i_{y,d}$ (we embed the $y$-coordinate time series instead of $x$-coordinate). 
The color of the curves is interpreted analogously.
Black curves on both plots are the same and correspond to the comparison of $\cL_1$ with $\cL_j$ for $j\in\{2,3,4\}$. 

As expected, we can observe a drift toward higher values of $\conjtest^+$ as the value of parameter $t$ increases.
Let us recall that $U^k_i$ in \eqref{eq:conjtest_plus} is a $k$-element approximation of a neighborhood of a point $x_i$.
The curve reflects how the image of $U^k_i$ under $f^t$ gets spread across the attractor with more iterations.
In consequence, a 2D embedding with $t=10$ might get lower value than $3$D embedding with $t=40$.
Nevertheless, Figure \ref{fig:lorenz_conj_time_grid} (top) shows consistency of the results across the tested range of values of parameter $t$.
Red curves corresponding to 1D embeddings give significantly higher values then the others.
We observe the strongest drop of values for 2D embeddings (green curves).
The third dimension (blue curves) does not improve the situation essentially, except for $t\in[1,25]$. 
The curves corresponding to 4D embeddings (yellow curves) overlap those of 3D embeddings.
Thus, the 4D embedded system does not resemble the Lorenz attractor essentially better than the 3D embedding.
It agrees with the analysis in the Experiment $\hyperref[sssec:exp_4b]{\text{4B}}$.

The $y$-coordinate embeddings presented in the bottom part of Figure \ref{fig:lorenz_conj_time_grid} give similar results.
However, we can see that gaps between curves corresponding to different dimensions are more visible.
Moreover, the absolute level of all curves is higher.
We interpret this outcome with a claim that the $y$-coordinate inherits a bit less information about the original system than the $x$-coordinate.
In Figure \ref{fig:lorenz_time_series} we can see that $y$-embedding is more twisted in the center of the attractor.
Hence, generally values are higher, and more temporal information is needed (reflected by higher embedding dimension) to compensate. 

Note that the comparison of $\cL_1$ to any embedding $\cP^i_{x,d}$ is always significantly worse than comparison of $\cL_1$ to any $\cL_j$. 
This may suggest that any embedding is not perfect. 

\begin{figure}
    \centering
    \includegraphics[width=0.8\textwidth]{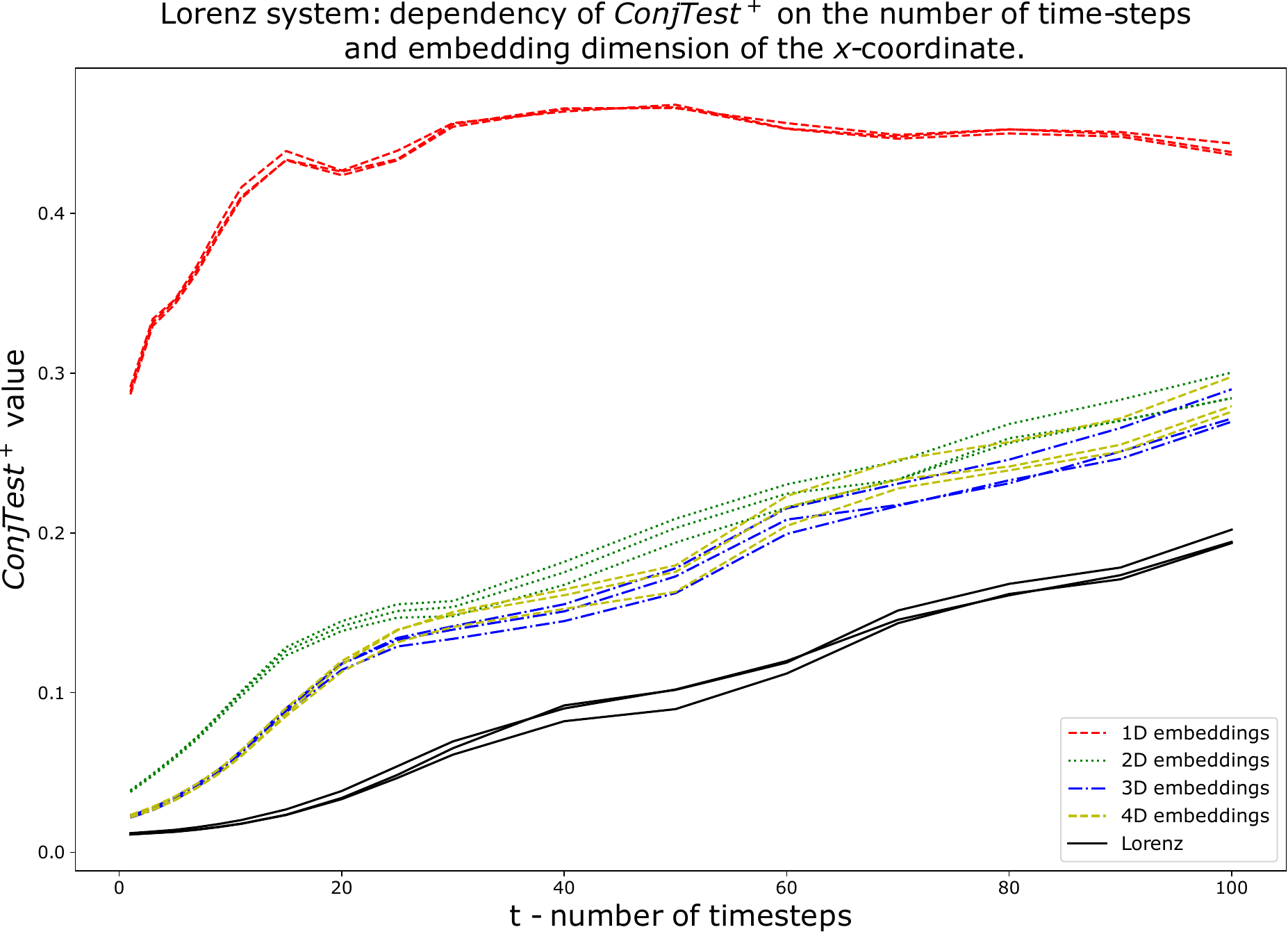}
    
    \vspace{0.2cm}
    \includegraphics[width=0.8\textwidth]{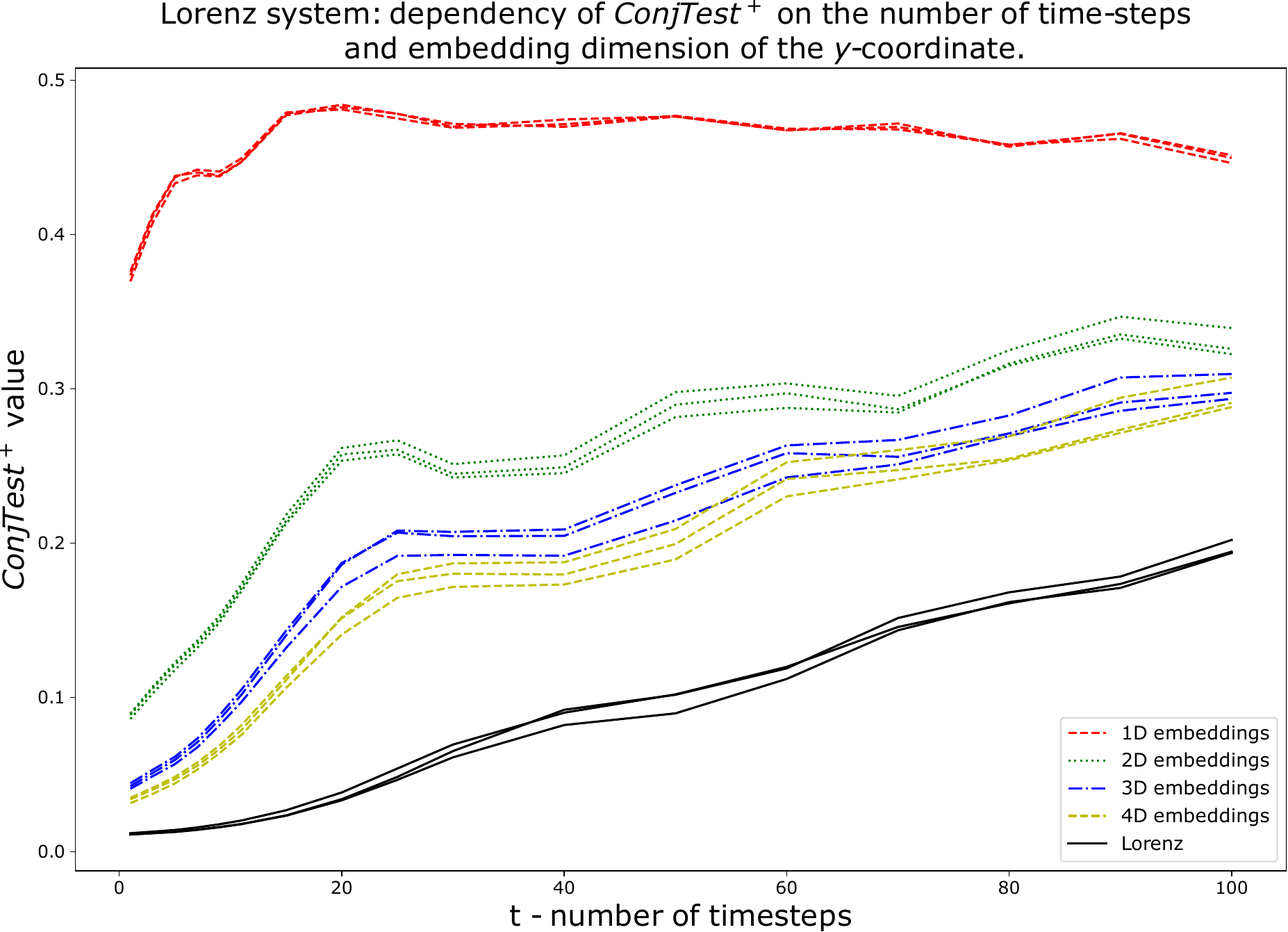}
    \caption{Dependence of $\conjtest^+$ on the parameter $t$ for Lorenz system.
        In this experiment multiple time series with  different starting points were generated.
        Each of them was used to produce an embedding.
        Top: comparison of $x$-coordinate embedding with $\cL_1$. 
        Bottom: comparison of $y$-coordinate embedding with $\cL_1$. For more explanation see text. 
    }
    \label{fig:lorenz_conj_time_grid}
\end{figure}

%%%%%%%%%%%%%%%%%%%%%%%%%%%%%%%%%%%%%%%%%%%%%%
%%%%%%%%%%%%%%%%%%%%%%%%%%%%%%%%%%%%%%%%%%%%%%
%%%%%%%%%%%%%%%%%%%%%%%%%%%%%%%%%%%%%%%%%%%%%%
\subsection{Example: rotation on the Klein bottle}\label{sec:klein_bottle}

In the next example we consider the Klein bottle, denoted $\KK$ and defined as an image $\KK:=\im \beta$ of the map $\beta$: 
\begin{equation}\label{eq:KleinParam}
    \beta: [0,2\pi)\times [0,2\pi)\ni  
    \begin{bmatrix}
     x\\  y
    \end{bmatrix}
    \mapsto
    \begin{bmatrix}
        \cos\frac{ x}{2}\cos y- \sin\frac{x}{2}\sin(2 y)\\
        \sin\frac{x}{2}\cos y + \cos\frac{x}{2}\sin(2 y)\\
        8 \cos x (1 + \frac{\sin y}{2})\\
        8 \sin x (1 + \frac{\sin y}{2})
    \end{bmatrix}
    \in \RR^4.
\end{equation}
In particular, the map $\beta$ is a bijection onto its image and the following ``rotation map'' $f_{[\phi_1,\phi_2]}:\KK\rightarrow\KK$ over the Klein bottle is well-defined:

\begin{equation*}
    f_{[\phi_1,\phi_2]}(x) := \beta\left( \beta^{-1}(x) + 
        \begin{bmatrix}
            \phi_1\\ \phi_2 
        \end{bmatrix}
        \mod 2\pi
        \right).
\end{equation*}

\subsubsection{Experiment 5A}\label{sssec:exp_5a}
We conduct an experiment analogous to Experiment $\hyperref[sssec:exp_4b]{4B}$ on estimating the optimal embedding dimension of a projection of the Klein bottle.  
\paragraph{Setup}
We generate the following time series
\begin{gather*}
    \cK=\varrho(f_{[\phi_1,\phi_2]}, (0,0,0,0), 8000),\\
    \cP_{d}=\Pi\left(
        (\proj{\cK}{1} + \proj{\cK}{2} + \proj{\cK}{3} + \proj{\cK}{4})/4
        , d, 8\right),
\end{gather*}
where $\phi_1=\frac{\sqrt{2}}{10}$, $\phi_2=\frac{\sqrt{3}}{10}$, $d\in\{2,3,4,5\}$ and $\proj{\cK}{i}$ denotes the projection onto the $i$-th coordinate. Note that in previous experiments we mostly used a simple observable $s$ which was a projection onto a given coordinate. However, in general, one can consider any (smooth) function as an observable. Therefore in the current experiment, in the definition of $\cP_d$, $s$ is a sum of all the coordinates, not the projection onto a chosen one. Note also that because of the symmetries (see formula \eqref{eq:KleinParam}) a single coordinate might be not enough to reconstruct the Klein bottle. 

\paragraph{Results}
We can proceed with the interpretation similar to Experiment $\hyperref[sssec:exp_4b]{4B}$. 
The $\fnn$ results (Figure \ref{fig:klein_embedding} top left) suggests that $4$ is a sufficient embedding dimension. 
The similar conclusion follows from $\KNN$ (Figure \ref{fig:klein_embedding} top right) and $\conjtest^+$ with a fixed parameter $k=10$ (Figure \ref{fig:klein_embedding} bottom right). 
The bottom left figure of \ref{fig:klein_embedding} is inconclusive as for the higher values of $k$  the curves do not stabilize even with high dimension.

Note that the increase of parameter $t$ in $\conjtest^+$ (Figure \ref{fig:klein_embedding} bottom right) does not result in drift of values as in Figure \ref{fig:lorenz_embedding} (bottom right).
In contrast to the Lorenz system studied in Experiment $\hyperref[sssec:exp_4b]{4B}$ the rotation on the Klein bottle is not  sensitive to the initial conditions. 

\begin{figure}
    \centering
    \includegraphics[width=0.49\textwidth]{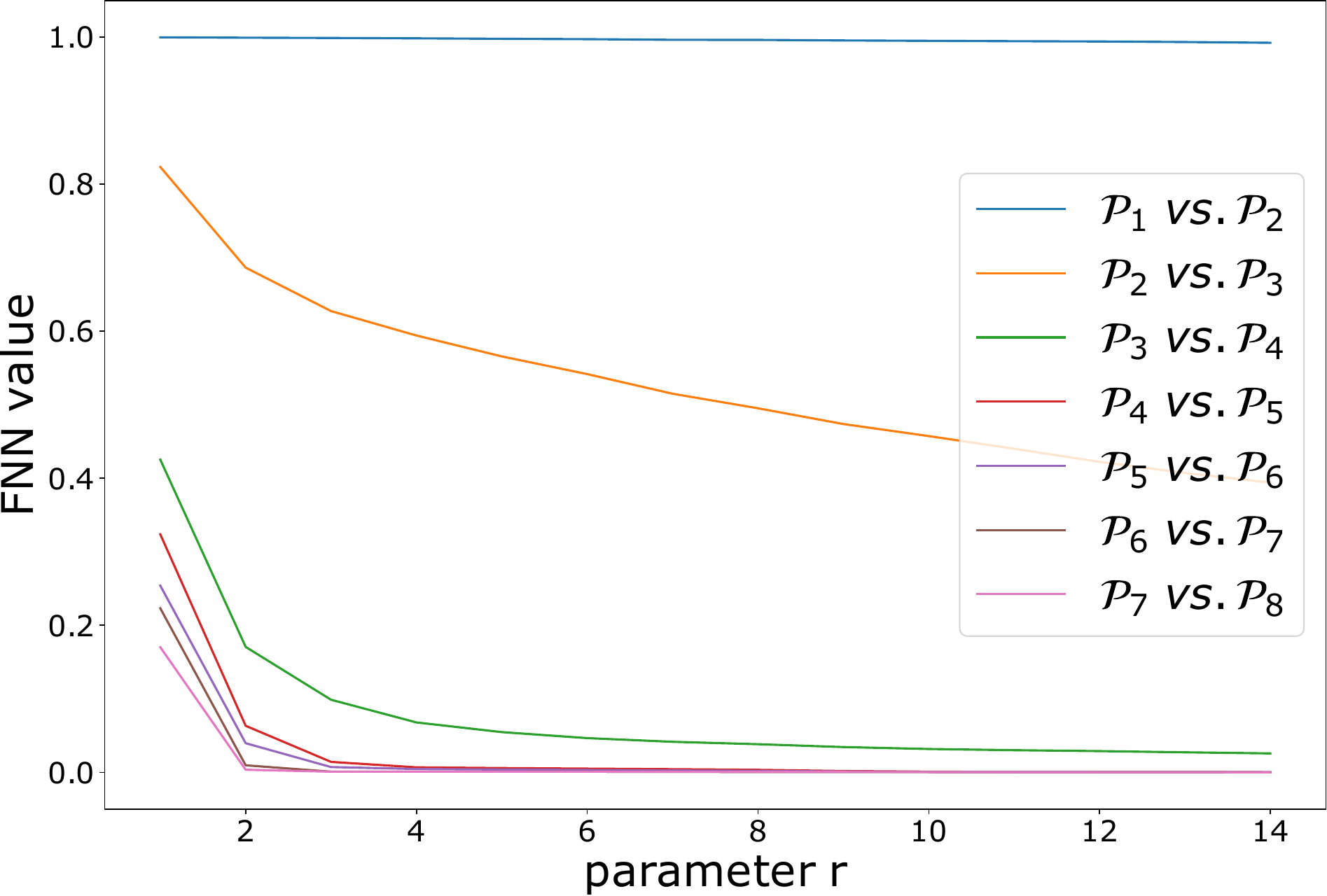}
    \hfill
    \includegraphics[width=0.49\textwidth]{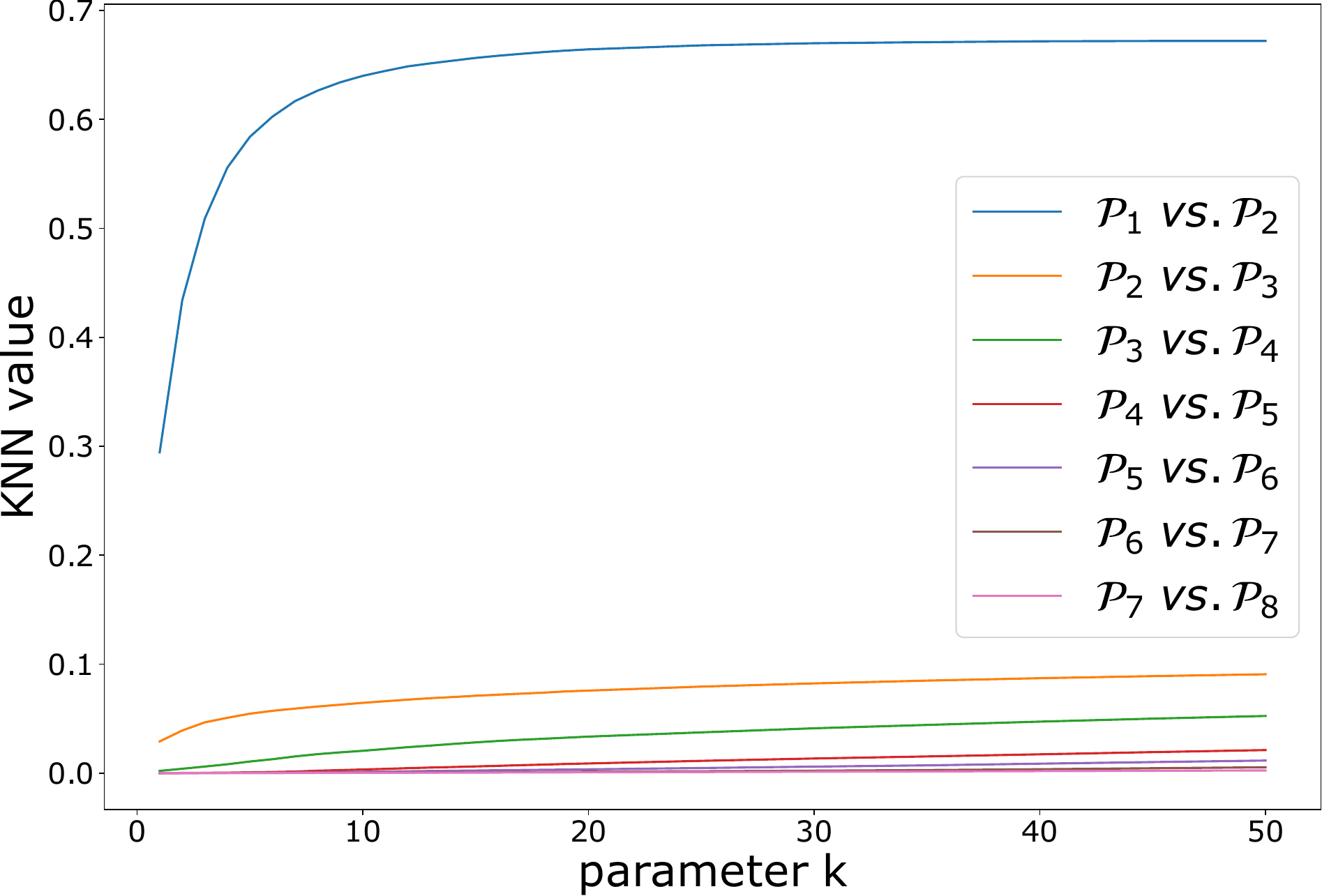}
    \vspace{0.5cm}
    
    \includegraphics[width=0.49\textwidth]{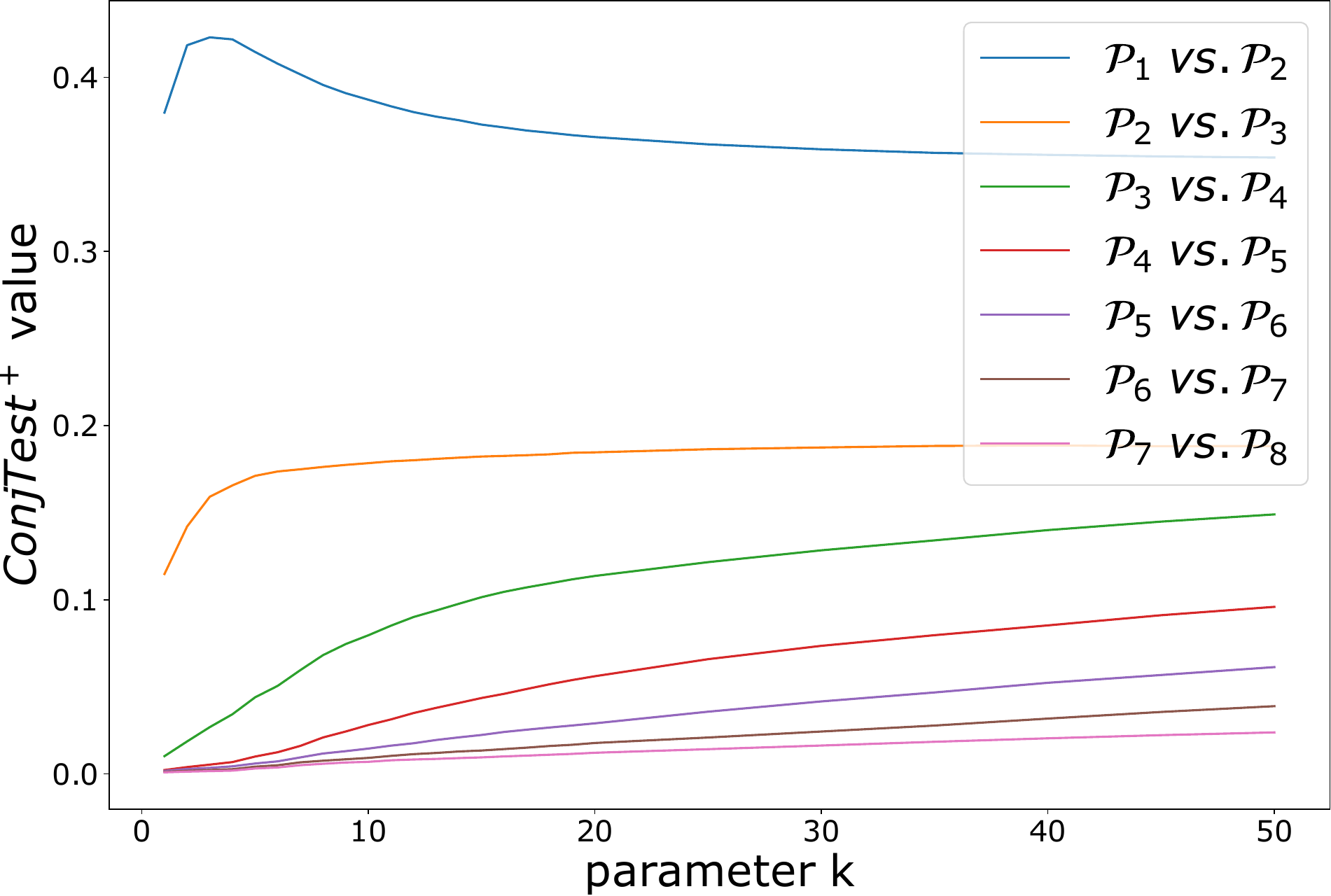}
    \hfill
    \includegraphics[width=0.49\textwidth]{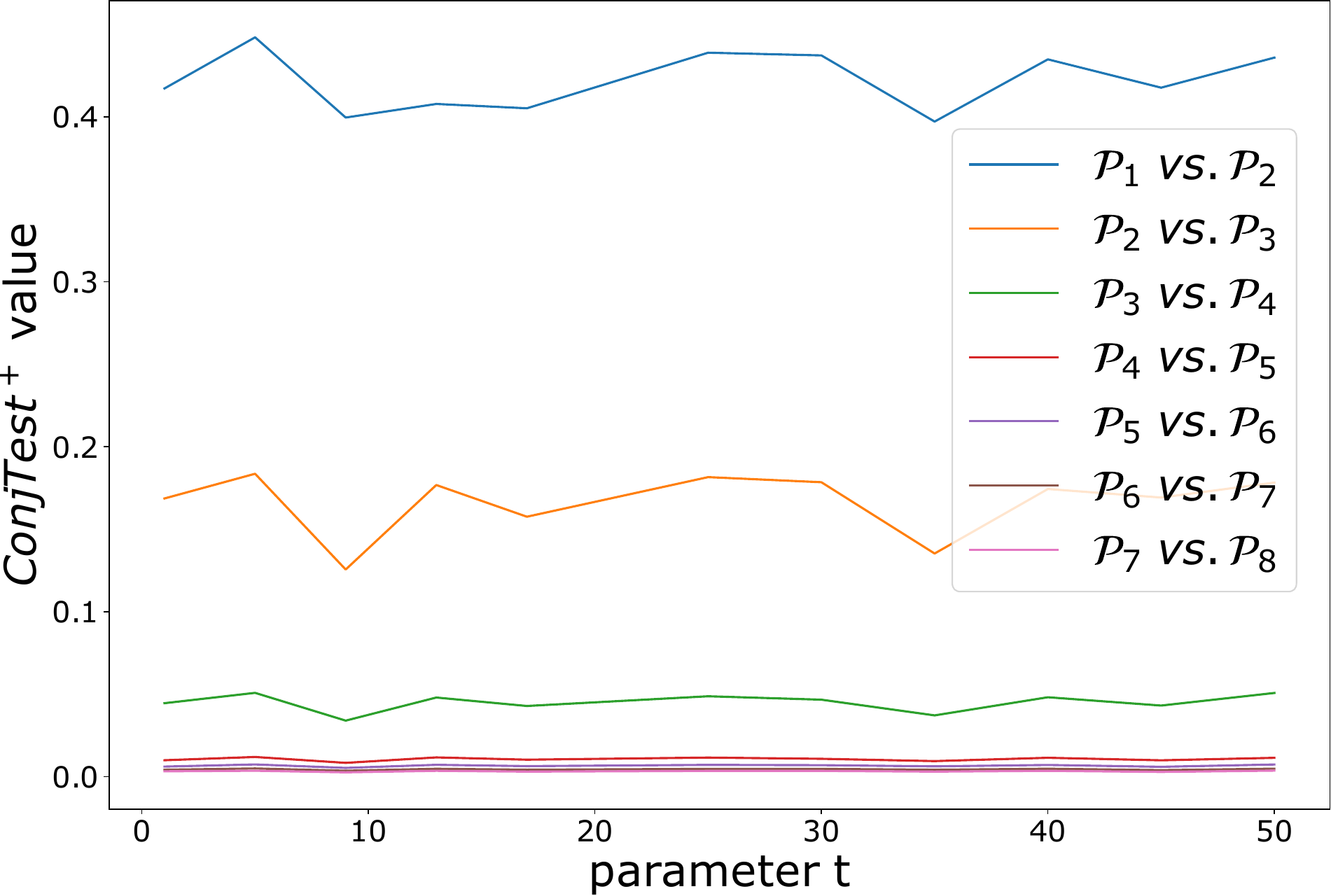}
    
    \caption{A comparison of the conjugacy measures for embeddings of the Klein bottle for consecutive dimensions.
    Top left: $\fnn$ with respect to parameter $r$.
    Top right: $\KNN$ with respect to parameter $k$.
    Bottom left:  $\conjtest^+$ with respect to parameter $k$ ($t=10$ fixed).
    Bottom right: $\conjtest^+$ with respect to parameter $t$ ($k=5$ fixed). 
    }
    \label{fig:klein_embedding}
\end{figure}

\section{Approximation of the connecting homeomorphism}
\label{sec:in_a_search_for_h}
In whole generality, finding the connecting homeomorphism between conjugate dynamical systems, is a very difficult task.
Some prior work has been done in this direction but the existing methods still have many limitations in applying for a broader class of systems. In particular, the works \cite{skufca2007,skufca2008,Zheng09} developed a method to produce conjugacy functions based on a functional fixed-point iteration scheme that can also be generalized to compare non-conjugate dynamical systems in which case the limit point of a fixed-point iteration scheme yields a  function called a ``commuter''. Quantifying how much the commuter function fails to be a homeomorphism (in various measures) led to the notion of a ``homeomorphic defect'' that, as the authors point out, allows one to quantify the dissimilarity of the two dynamical systems. However, the method  has been illustrated on a very few specific examples and could be rigorously mathematically justified only under very restrictive assumptions e.g. uniform contraction of at least one of the systems when comparing systems in one dimension.  Significant problems occur in rigorous extension to systems of higher dimension.  Consequently, later work \cite{bollt2010} extended this theory to allow for multivariate transformations and presented ideas on constructing commuter functions different than fixed-point iteration scheme. This method is based on symbolic dynamics approach which, however, requires existence and finding a general partition for systems being compared. Although the approach of finding a (semi-)conjugacy to a symbolic shift space through generating or Markov partition seems very natural from the point of view of dynamical systems theory, in practice finding a \emph{reasonable} partition even for systems given by explicit equations (not to mention time series of real data) is often not feasible. Moreover, one can face an explosion of computational complexity as the number of symbols increases.  Later work \cite{zheng2013} employs a method of graph matching between the graphs representing the underlying symbolic dynamics or, alternatively,  between the graphs approximating the action of the systems on some eligible partition. Interestingly, the authors show that the permutation matrices that relate the adjacency matrices of the merging graphs
coincide with the solution of Monge’s mass transport problem.

The above earlier works contain valuable ideas on finding to-be-conjugacies or commuter functions and  the defect measures of the arising commuters might serve as a quantification of the dynamical similarity between two given systems. In turn, our proposed tools, $\conjtest$ and $\conjtest^+$ can be applied, among others, to explicitly assess the quality of the matching between the two systems through the commuting functions obtained by the above mentioned methods and these matching functions can be candidates for testing (semi-)conjugacy of given systems by $\conjtest$s. Note also that, contrary to the previous works, $\conjtest$ methods can be applied directly to the time-series since we work on the point clouds and do not need a priori the formulas for systems which generated them - these are only used as benchmark tests.

However, as our contribution and small step forward towards effective algorithms of finding conjugating maps, in this section we present a proof-of-concept gradient-descent algorithm, utilizing the $\conjtest$ as a cost function, to approximate such a connecting homeomorphism. 
More precisely, as an example, we use it to discover an approximation of the map \eqref{eq:log_tent_homoemorphism} that constitutes a topological conjugacy between the tent and the logistic map (see~Section~\ref{ssec:log_tent}). 
Instead of finding an analytical formula approximating the connecting homeomorphism our strategy aims to construct a cubical set representing the map.
Further development and generalization of the presented procedure will be a subject of forthcoming studies.

Consider the following sequence $0=a_1< a_2<\ldots< a_{n+1}=1$.
Denote $\ttA_{i,j}:=[a_i,a_{i+1}]\times[a_j, a_{j+1}]$ and 
    $\mathbb{A}:=\{\ttA_{i,j}\mid i,j\in\{1,2,\ldots n\}\}$.
Let $h:I\rightarrow I$ be an increasing homeomorphism from the unit interval $I$ to itself 
and by $\pi(h):=\{(x,y)\in I\times I\mid y=h(x)\}$ denote the graph of $h$.
We say that a collection $\tth=\{\ttA_{i,j}\in\bbA\mid \inter \ttA_{i,j}\cap\pi(h)\neq\emptyset\}$ is a \emph{the cubical approximation} of $h$ and we denote it by $[h]$.
Equivalently, $\tth$ is the minimal subset of $\bbA$ such that $\pi(h)\subset\bigcup\tth$.
We refer to 
\[
    \ttH:=\{[h]\subset\bbA\mid h:I\rightarrow I \text{ - an increasing homeomorphism}\}
\]
as a family of all \emph{cubical homeomorphisms} of $\bbA$.

In~\ref{apx:approx_h} we show how to construct a class of piecewise linear homeomorphisms for any $\tth\in\ttH$.
We denote a \emph{selector} of $\tth$, that is a homeomorphism representing $\tth$, by $f_\tth$.
Take, as an example, cubical sets marked with yellow cubes in Figure~\ref{fig:x5_logtent_example}. 
At every panel, the blue curve corresponds to the graph of the selector.

The size of family $\ttH$ grows exponentially with the resolution of the grid (the explicit formula for the size of $\ttH$ is given in \ref{apx:approx_h}).
For instance, the number of cubical homeomorphisms for $m=21$ is about $2.6\cdot 10^{14}$.
Thus, it is hopeless to find the optimal approximation of the connecting homeomorphism by a brute examination of all elements of $\ttH$.
Instead, we propose an algorithm based on the gradient descent strategy using $\conjtest$ as a cost function.

\begin{algorithm}
\caption{{\tt ApproximateH}}\label{alg:approximateH}
\begin{algorithmic}[1]
\REQUIRE $\cA, \cB$ -- time series on an interval,
    $\tth_0\in\tth$ -- initial approximation of the homeomorphism,
    $\tt nsteps$ -- number of steps,
    $\tt p$ -- memory size.
\ENSURE $\tt best\_h$ -- approximated connecting homeomorphism
    \STATE $\tth \leftarrow \tth_0$
    \STATE $\tt q \leftarrow$ initialize queue of size $\tt p$ with $\tt null$'s
    \STATE $\tt best\_h \leftarrow h$
    \FOR{$\tt t=0\ to\ nsteps$}
        \STATE $\tt c \leftarrow \{h'\in h \mid h'\in nbhd(h) \text{ and } \hdiff(h,h')\not\subset q\}$
        \IF{$\tt \# c = 0$}
            \STATE \textbf{break}
        \ELSE
            \STATE ${\tt h} \leftarrow {\tt h'\in c}$ with a minimal value of $\hscore(\tth', \cA, \cB)$
            \IF{$\tt score(h, \cA, \cB) < score(best\_h, \cA, \cB)$}
                \STATE $\tt best\_h \leftarrow h$
            \ENDIF
            \STATE{$\tt q.append(\hdiff(h,h'))$} \COMMENT{append the unique element differentiating $\tth$ and $\tth'$}
            \STATE{$\tt q.pop()$} 
        \ENDIF
    \ENDFOR
    \RETURN $\tt best\_h$
\end{algorithmic}
\end{algorithm}

Let $\cA$ and $\cB$ be time series on a unit interval.
Algorithm \ref{alg:approximateH} attempts to find an element of $\ttH$ with as small value of the $\conjtest$ as possible. For that purpose, each element $\tth \in \ttH$ can be assigned with the following score:
\[
    \hscore(\tth, \cA, \cB):=
        \max\left\{
            \conjtest(\cA, \cB; k, t, f_{\tth}), \conjtest(\cB, \cA; k, t, f_{\tth}^{-1})
        \right\}.
\]
Since elements $\tt h,h'\in H$ are collections of sets, the symmetric difference gives a set of cubes differing $\tth$ and $\tth'$.
We denote it by
\[
    \tt\hdiff(\tth, \tth') := (h\setminus h') \cup (h'\setminus h).
\]

Let $\tt h\in\ttH$ be an initial guess for the connecting homeomorphism.
In each step of the algorithm an attempt is made to update it in a way that the $\hscore$ gets improved.
Each iteration of the main loop considers all neighbors $\tth'$ of $\tth$ in $\ttH$ such that $\tt nbhd(h):=\{h'\in\ttH\mid\#\hdiff(h,h')=1\}$ (a unit sphere in a Hamming distance centered in $\tth$).
The element $\tth'$ of $\tt nbhd(h)$ with minimal $\hscore(\tth')$ is chosen for the next iteration of the algorithm.
Note that it might happen that $\hscore(\tth') < \hscore(\tth)$.
This prevents the algorithm from being stuck at a local minimum.
In addition, to avoid orbiting around them we exclude elements of $\tt nbhd(h)$ for which element $\tt\hdiff(h, h')$ is an element added or removed in previous $\tt p$ iterations of the algorithm.
This strategy is more restrictive from just avoiding assigning to $\tth$ the same cubical homeomorphism twice within $\ttp$ consecutive steps which 
still could result in oscillatory changes of some cubes.

We conducted an experiment for the problem studied in Section \ref{ssec:log_tent}, that is, a comparison of time series generated by the logistic and the tent map.
Take the following time series
\begin{equation*}
    \cA = \varrho(f_4, 0.02, 500) 
    \quad\text{ and }\quad
    \cB = \varrho(g_2, 0.87, 500),
\end{equation*}
where $f_4$ and $g_2$ are respectively a logistic and a tent map as in Section \ref{ssec:log_tent}.
We use $\hscore$ with parameters $k=5$ and $t=1$.
As an initial guess of the connecting homeomorphism $\tth_0$ we naively took a cubical approximation of a $h(x)=x^5$ with resolution $m=21$, as presented in the top-left panel of Figure~\ref{fig:x5_logtent_example}.
We run Algorithm \ref{alg:approximateH} for time series $\cA$ and $\cB$ for $1000$ steps with the memory parameter ${\tt p}=m=21$.
Figure~\ref{fig:cts_decay} shows the values of the $\hscore$ for the consecutive approximations.
We can see that the algorithm falls temporarily into local minima, but eventually, thanks to the memory parameter, it escapes them and settles down towards the low score values. 
Figure \ref{fig:x5_logtent_example} shows relations corresponding to the $1$st, $200$th, $400$th, and $612$nd iteration of the algorithm run.
The bottom-right panel, the $612$nd iteration is the relation inducing the lowest score among all iterations. 
The orange curve, at the same panel, is the graph of homeomorphism \eqref{eq:log_tent_homoemorphism} -- the analytically correct map conjugating $f_4$ and $g_2$. Clearly, the iterations are converging towards the right value of the connecting homeomorphism.

Clearly, the presented approach can be applied to any one-dimensional time series.
A generalization of the algorithm will be a subject of further studies.

\begin{figure}
    \centering
    \includegraphics[width=0.475\textwidth]{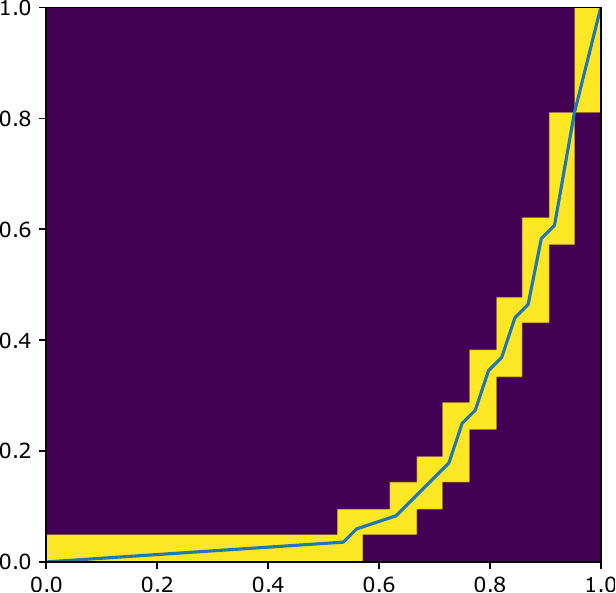}
    \hspace{0.2cm}
    \includegraphics[width=0.475\textwidth]{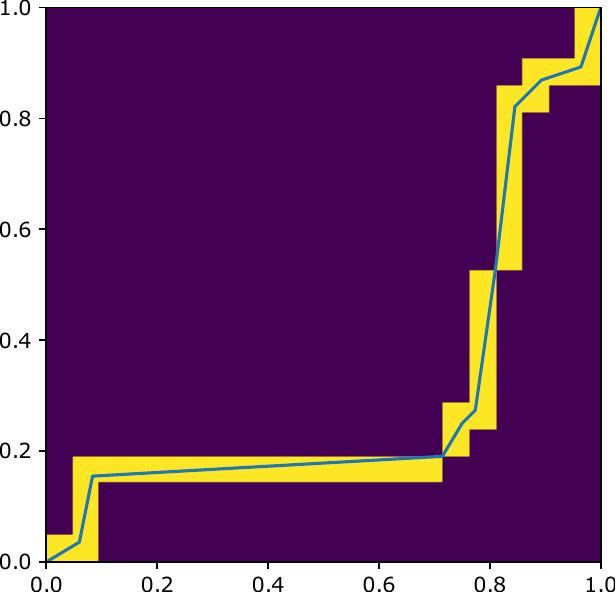}

    \vspace{1.2cm}
    \includegraphics[width=0.475\textwidth]{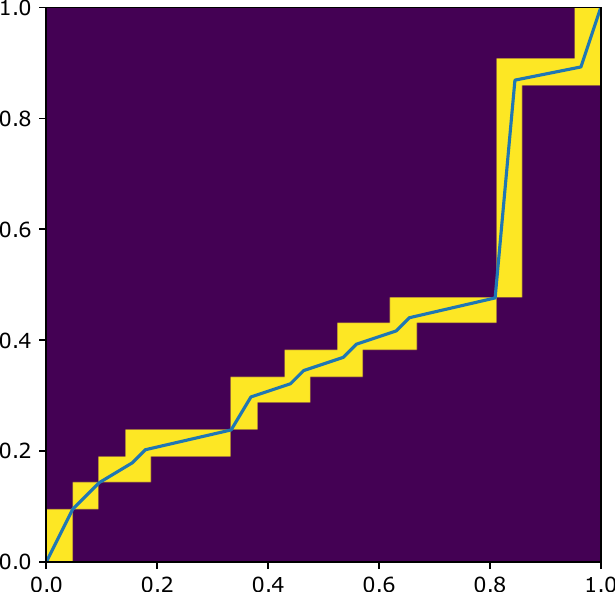}
    \hspace{0.2cm}
    \includegraphics[width=0.475\textwidth]{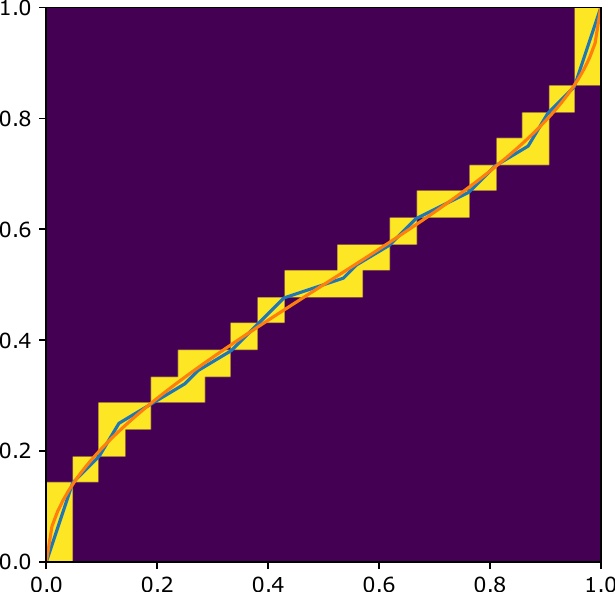}
    \caption{
        Steps $0$, $200$, $400$, and $612$ (top left, top right, bottom left, bottom right, respectively) of the run of Algorithm \ref{alg:approximateH} in a search for the connecting homeomorphism between the logistic and the tent map.
        The blue lines corresponds to a selector of a cubical homeomorphism.
        The orange curve in bottom right panel is a graph of the actual connecting homeomorphism \eqref{eq:log_tent_homoemorphism} between the logistic and the tent map.
    }
    \label{fig:x5_logtent_example}
\end{figure}

\begin{figure}
    \centering
    \includegraphics[width=0.75\textwidth]{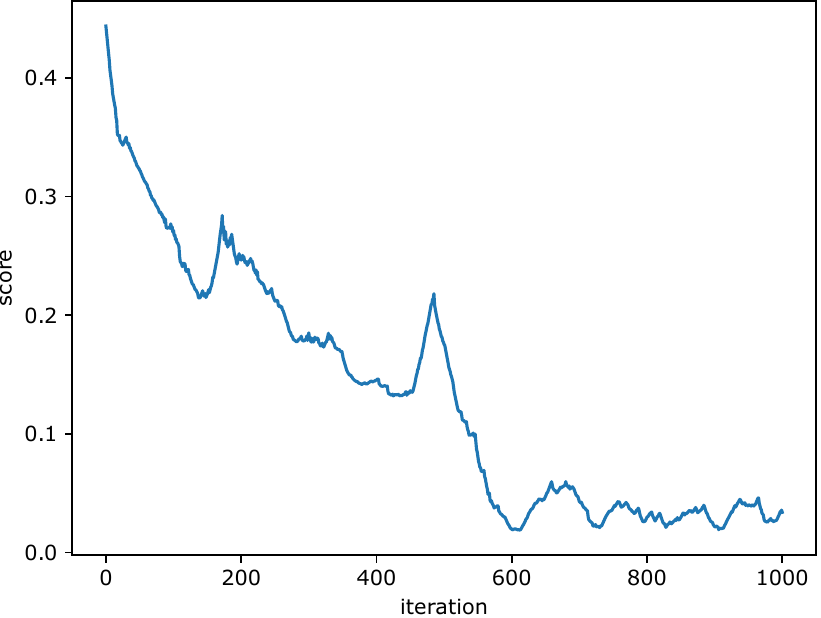}
    \caption{The value of $\hscore(\tth, \cA, \cB)$ for every step of the experiment approximating the connecting homeomorphism between the logistic and the tent map.}
    \label{fig:cts_decay}
\end{figure}

%%%%%%%%%%%%%%%%%%%%%%%%%%%%%%%%%%%%%%%%%%%%%%
%%%%%%%%%%%%%%%%%%%%%%%%%%%%%%%%%%%%%%%%%%%%%%
%%%%%%%%%%%%%%%%%%%%%%%%%%%%%%%%%%%%%%%%%%%%%%
\section{Discussion and Conclusions}\label{sec:disscussion}

There is a considerable gap between theory and practice when working with dynamical systems; In theoretical consideration, the exact formulas describing the considered system is usually known. 
Yet in biology, economy, medicine, and many other disciplines, those formulas are unknown; only a finite sample of dynamics is given. 
This sample contains either sequence of points in the phase space, or one-dimensional time series obtained by applying an observable function to the trajectory of the unknown dynamics. 
This paper provides tools, $\fnn$, $\KNN$, $\conjtest$, and $\conjtest^+$, which can be used to test how similar two dynamical systems are, knowing them only through a finite sample.
Proof of consistency of some of the presented methods is given.

The first method, $\fnn$ distance, is a modification of the classical False Nearest Neighbor technique designed to estimate the embedding dimension of a time series. 
The second one, $\KNN$ distance, has been proposed as an alternative to $\fnn$ that takes into account larger neighborhood of a point, not only the nearest neighbor.
The conducted experiments show a strong similarity of $\fnn$ and $\KNN$ methods. 
Additionally, both methods admit similar requirements with respect to the time series being compared: they should have the same length and their points should be in the exact correspondence, i.e., we imply that an $i$-th point of the first time series is a dynamical counterpart of the $i$-th point of the second time series. 
An approximately binary response characterizes both methods in the sense that they return either a value close to $0$ when the compared time series come from conjugate systems, or a significantly higher, non-zero value in the other case.
This rigidness might be advantageous in some cases.
However, for most empirical settings, due to the presence of various kind of noise, $\fnn$ and $\KNN$ may fail to recognize similarities between time series. Consequently, these two methods are very sensitive to any perturbation of the initial condition of time series as well as the parameters of the considered systems.
However, $\KNN$, in contrast to $\fnn$, admits robustness on a measurements noise as presented in Experiment $\hyperref[sssec:exp_1c]{\text{1C}}$.
On the other hand, $\fnn$ performs better than $\KNN$ in estimating the sufficient embedding dimension (Experiments $\hyperref[sssec:exp_4b]{\text{4B}}$, $\hyperref[sssec:exp_5a]{\text{5A}}$). Moreover, the apparently clear response given by $\fnn$ and $\KNN$ tests might not be correct (see Experiment $\hyperref[sssec:exp_1a]{1A}$, $\cR_1$ vs. $\cR_4$).

\begin{table}
\centering
\begin{footnotesize}
\begin{tabular}{|m{3.cm}|m{2.5cm}|m{2.5cm}|m{2.7cm}|m{2.7cm}|} \hline
\backslashbox{Property}{Method}
            & \thead{$\fnn$} 
            & \thead{$\KNN$} 
            & \thead{$\conjtest$} 
            & \thead{$\conjtest^+$} \\ \hline
Requirements  & 
    \multicolumn{2}{|m{5.cm}|}{
        identical matching between indexes of the elements (in particular the series must be of the same length)
        }  & 
    \multicolumn{2}{|m{5.4cm}|}{
        \begin{itemize}[leftmargin=-2pt]
            \item an exact correspondence between the two time series is not needed 
            \item  allow examining arbitrary, even very complicated potential relations between the series 
            \item can be used for comparison of time series of different length 
            \item require defining the possible (semi)conjugacy $h$ at least locally i.e. giving the corresponding relation between indexes of the elements of the two series 
        \end{itemize}} \\ \hline
Parameters  & 
    only one parameter: $r$ (but one should examine large interval of $r$ values)  & only one parameter: $k$ (but is recommend to check a couple of different $k$ values)  & 
    \multicolumn{2}{|m{5.cm}|}{involve two parameters: $k$ and $t$} \\ \hline Robustness  & 
    \multicolumn{2}{|m{5.cm}|}{
        \begin{itemize}[leftmargin=-2pt]
            \item less robust to noise and perturbation than $\conjtest$ methods
            \item give nearly a binary output 
            \item $\KNN$ seems to admit robustness with respect to the measurement noise 
        \end{itemize}}  & 
            \multicolumn{2}{|m{5.cm}|}{
                \begin{itemize}[leftmargin=-2pt] 
                    \item more robust to noise and perturbation
                    \item the returned answer depends continuously on the level of perturbation and noise compared to the binary response given by $\fnn$ or $\KNN$ 
                \end{itemize}} \\ \hline
    Recurrent properties  & takes into account only the one closest return of a series (trajectory) to each neighborhood 
      & \multicolumn{3}{|m{8.4cm}|}  {takes into account $k$-closest returns} \\ \hline

    Further properties  &    & & more likely to give false positive answer than $\conjtest^+$
    
    & more computationally demanding than $\conjtest$ but usually more reliable \\ \hline
\end{tabular}
\end{footnotesize}
\caption{
    Comparison of the properties of discussed conjugacy measures.  
}
\label{tab:comparisontab}
\end{table}

Both $\conjtest$ and $\conjtest^+$ (collectively called $\conjtest$ methods) are directly inspired by the definition and properties of topological conjugacy. They are more flexible in all considered experiments and can be applied to time series of different lengths and generated by different initial conditions (the first point of the series). 
In contrast to $\fnn$ and $\KNN$, they admit more robust behavior with respect to any kind of perturbation, be it measurement noise (Experiment $\hyperref[sssec:exp_1c]{1C}$), perturbation of the initial condition (Experiments $\hyperref[sssec:exp_1a]{\text{1A}}$, $\hyperref[sssec:exp_2a]{\text{2A}}$, $\hyperref[sssec:exp_3a]{\text{3A}}$, and $\hyperref[sssec:exp_4a]{\text{4A}}$), $t$ parameter (Experiment $\hyperref[sssec:exp_4c]{4C}$), or a parameter of a system (Experiment $\hyperref[sssec:exp_1b]{1B}$). 
In most experiments, we can observe a continuous-like dependence of the test value on the level of perturbations. We see this effect as softening the concept of topological conjugacy by $\conjtest$ methods. 
A downside of this weakening is a lack of definite response whether two time series come from conjugate dynamical systems. 
Hence the $\conjtest$ methods should be considered as a means for a quantification of a dynamical similarity of two processes. Experiments $\hyperref[sssec:exp_1a]{\text{1A}}$, $\hyperref[sssec:exp_2a]{\text{2A}}$, and $\hyperref[sssec:exp_3a]{\text{3A}}$ show that both methods, $\conjtest$ and $\conjtest^+$, capture essentially the same information from data. 
In general, $\conjtest$ is simpler and, thus, computationally more efficient. 
However, Experiment $\hyperref[sssec:exp_4a]{\text{4A}}$ shows that $\conjtest$ (in contrast to $\conjtest^+$) does not work well in the context of embedded time series, especially when the compared embeddings are constructed from the same time series. 
Experiments $\hyperref[sssec:exp_4b]{\text{4B}}$ and $\hyperref[sssec:exp_5a]{\text{5A}}$ show that the variation of $\conjtest$ methods with respect to the $t$ parameter can also be used for estimating a good embedding dimension. 
Further comparison between $\conjtest$ and $\conjtest^+$ reveals that $\conjtest^+$  is more computationally demanding than $\conjtest$, but also more reliable. 
Indeed,  in our examples with rotations on the circle and torus and with the logistic map, both these tests gave nearly identical results, but the examples with the Lorenz system show that $\conjtest$ is more likely to give a false positive answer.  
This is due to the fact that $\conjtest$  works well if the map $h$ connecting time series $\cA$ and $\cB$  is a reasonably good approximation of the true conjugating homeomorphism, but in case of embeddings and naive, point-wise connection map, as in some of our examples with Lorenz system, the Hausdorff distance in formula \eqref{eq:conjtest} might vanish resulting in false positive.

The advantages of $\conjtest$ and $\conjtest^+$ methods come with the price of finding a connecting map relating two time series. 
When it is unknown, in the simplest case, one can try the map $h$ which is defined only locally i.e. on points of the time series and provide an order- preserving matching of indexes of corresponding points in the time series. The simplest example of such a map is an identity map between indices.  The question of finding an optimal matching is, however, much more challenging and will be a subject of a further study. Nonetheless, in Section \ref{sec:in_a_search_for_h} and Appendix \ref{apx:approx_h} we present preliminary results approaching this challenge.

A convenient summary of the presented methods is gathered in Table~\ref{tab:comparisontab}.

\appendix
\section{Cubical homeomorphisms}\label{apx:approx_h}

This appendix offers additional characterization of the family of cubical homeomorphisms introduced in Section~\ref{sec:in_a_search_for_h}.

At first, observe that elements of $\ttH$ have the following straightforward observations.
\begin{proposition}\label{prop:01_cubes}
Let $[h]\in\ttH$. Then, $\ttA_{1,1},\ttA_{n,n}\in[h]$.
\end{proposition}
\begin{proof}
    Since $h$ is an increasing homeomorphism it follows that $(0,0), (1,1)\in\pi(h)$.
    In consequence, $\ttA_{1,1},\ttA_{n,n}\in[h]$, because these are the only elements of~$\bbA$ containing $(0,0)$ and $(1,1)$.
\end{proof}

We picture the idea of the following simple proposition with Figure \ref{fig:cub_hoemo_three_cases}.
\begin{proposition}\label{prop:cases_cub_homeo}
    Let $\pi(h)\cap \inter\ttA_{i,j}\neq\emptyset$ then exactly one of the following holds
    \begin{enumerate}[label=(\arabic*)]
        \item\label{it:right_crossing_point} 
            $\inter\ttA_{i+1,j}\cap\pi(h)\neq\emptyset$ and 
            $\inter\ttA_{i, j+1}\cap\pi(h)=\emptyset$, when $h(a_{i+1})\in(a_j, a_{j+1})$,
        \item\label{it:up_crossing_point} 
            $\inter\ttA_{i,j+1}\cap\pi(h)\neq\emptyset$ and 
            $\inter\ttA_{i+1, j}\cap\pi(h)=\emptyset$, when $h^{-1}(a_{i+1})\in(a_i, a_{i+1})$,
        \item\label{it:diag_crossing_point} 
            $\inter\ttA_{i+1,j+1}\cap\pi(h)\neq\emptyset$, 
            $\inter\ttA_{i+1, j}\cap\pi(h)=\emptyset$ and 
            $\inter\ttA_{i, j+1}\cap\pi(h)=\emptyset$, when $h(a_{i+1})=a_{j+1}$.
    \end{enumerate}
\end{proposition}
\begin{figure}[h]
    \centering
    \includegraphics[scale=0.35]{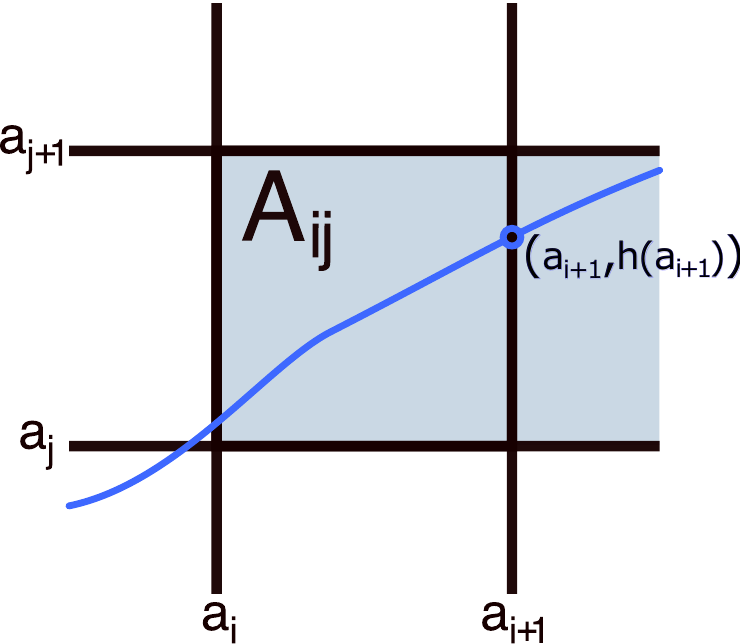}
    \includegraphics[scale=0.35]{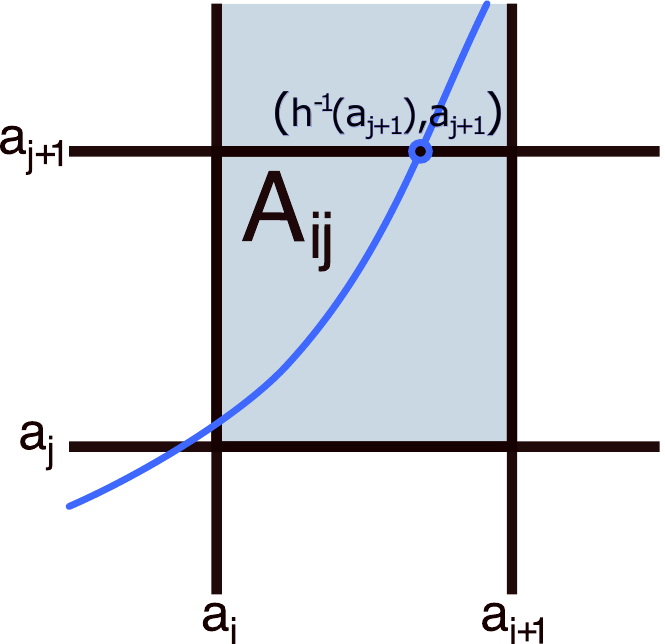}
    \includegraphics[scale=0.35]{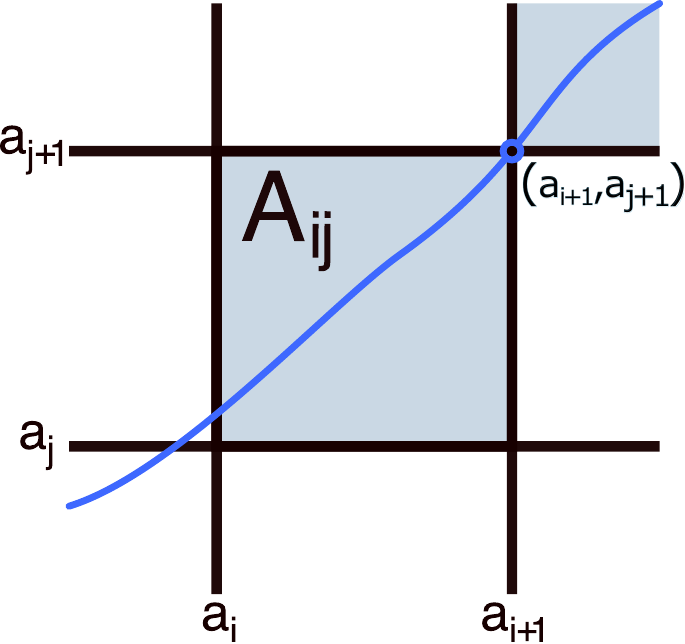}
    \caption{From left to right, cases \ref{it:right_crossing_point}, \ref{it:up_crossing_point} and \ref{it:diag_crossing_point} of Proposition \ref{prop:cases_cub_homeo}.}
    \label{fig:cub_hoemo_three_cases}
\end{figure}

Again, the proof for the next proposition is a consequence of basic properties of homeomorphism $h$ such that $[h]=\tth$.
\begin{proposition}\label{prop:cub_h_properties}
Let $\tth\in\ttH$. Then    
\begin{enumerate}[label=(\roman*)]
    \item\label{it:cub_h_property_1i}
    for every $i\in\{1,\ldots, n\}$ set $\bigcup\{\ttA_{i,j}\in\tth\mid j\in\{0,1,\ldots, n\}\}$ is nonempty and connected,
    \item\label{it:cub_h_property_1j}
    for every $j\in\{1,\ldots, n\}$ set $\bigcup\{\ttA_{i,j}\in\tth\mid i\in\{0,1,\ldots, n\}\}$ is nonempty and connected.
    \item\label{it:cub_h_property_2i}
    if $\ttA_{i,j}\in\tth$ then for every $i'>i$ and $j'<j$ we have 
        $\ttA_{i',j'}\not\in\tth$,
    \item\label{it:cub_h_property_2j}
    if $\ttA_{i,j}\in\tth$ then for every $j'>j$ and $i'<i$ we have 
        $\ttA_{i',j'}\not\in\tth$,
\end{enumerate}
\end{proposition}

The above propositions implies that every element $\tth\in\ttH$ can be seen as a path starting from element $\ttA_{1,1}$ and ending at $\ttA_{n,n}$.
In particular, $\tth$ can be represented as a vector of symbols 
    $R$ (right, incrementation of index $i$, case \ref{it:right_crossing_point}), 
    $U$ (up, incrementation of index $j$, case \ref{it:up_crossing_point}), and 
    $D$ (diagonal, incrementation of both indices, case \ref{it:diag_crossing_point}).
The vectors do not have to be of the same length.
In particular, a single symbol D can replace a pair of symbols R and U.
Denote by $n_R$, $n_U$ and $n_D$ the number of corresponding symbols in the vector.
As indices $i$ and $j$ have to be incremented from $1$ to $n$ we have the following properties:
\begin{align}\label{eq:cub_homeo_vec_conditions}
    0\leq n_R, n_U, n_D\leq n-1, \ n_R+n_U+2n_D = 2(n-1) \ \text{ and }\ n_R=n_U.
\end{align}

Actually, any vector $\ttV$ of symbols $\{R,U,D\}$ satisfying the above conditions \eqref{eq:cub_homeo_vec_conditions} corresponds to a cubical homeomorphism.
We show it by constructing a piecewise-linear homeomorphism $h$ such that $[h]=\tth$ for any $\tth$ represented by $\ttV$. 
We refer to the constructed $h$ as a \emph{selector} of $\tth$.
In particular, Algorithm \ref{alg:selector} produces a sequence of points corresponding to points of non-differentiability of the homeomorphism.
We have five type of points, the starting point $(0,0)$ (type B), the ending point $(1,1)$ (type E), and points corresponding to subsequences $UR$ (type $UR$), $ RU$ (type $RU$) and $D$ (type $D$). 
Proposition \ref{prop:alg_gives_selector} shows that the map generated by the algorithm is an actual homeomorphism.

\begin{algorithm}
\caption{{\tt FindSelector}}\label{alg:selector}
\begin{algorithmic}[1]
\REQUIRE $\ttV$ -- a vector of symbols $\{R,U,D\}$ satisfying \eqref{eq:cub_homeo_vec_conditions},
    $p$ -- a parameter for breaking points selection 
\ENSURE $L$ -- a sequence encoding the selector
    \STATE $L \leftarrow \{(0,0)\}$
    \STATE $\tt prev \leftarrow Null$
    \STATE $i, j \leftarrow 0$
    \FOR{$\tt s \in V$}
        \IF{${\tt prev}=U$ and $\tts=R$} 
            \STATE $L = L \cup\{
                (p\,a_i + (1-p)\,a_{i+1}, (1-p)\,a_j + p\,a_{j+1})
                \}$ 
        \ELSIF{ ${\tt prev}=R$ and $\tts=U$} 
            \STATE $L = L \cup\{
                ((1-p)\,a_i + p\,a_{i+1}, p\,a_j + (1-p)\,a_{j+1})
                \}$ 
        \ELSIF{$\tts=D$} 
            \STATE $L = L \cup\{
                (a_{i+1}, a_{j+1})
                \}$ 
        \ENDIF
        \STATE{$\tt prev \leftarrow s$}
        \IF{$\tts = R$ or $\tts = D$}
            \STATE $i \leftarrow i+1$
        \ENDIF
        \IF{$\tts = U$ or $\tts = D$}
            \STATE $j \leftarrow j+1$
        \ENDIF
    \ENDFOR 
    \STATE $L \leftarrow \{(1,1)\}$ 
    \RETURN $L$
\end{algorithmic}
\end{algorithm}

\begin{proposition}\label{prop:alg_gives_selector}
    Let $\ttV$ be a vector of symbols $\{R,U,D\}$ satisfying \eqref{eq:cub_homeo_vec_conditions} and $\tth$ the corresponding cubical set.
    Let $L=\{(x_1,y_1),(x_2,y_2),\ldots,(x_K,y_K)\}$ be a sequence of points generated by Algorithm \ref{alg:selector} for $\ttV$.
    Then, for every $k\in\{1,2,\ldots K-1\}$ following properties are satisfied:
    \begin{enumerate}[label=(\roman*)]
        \item\label{it:selector_monotonicity} $x_k< x_{k+1}$ and $y_k< y_{k+1}$,
        \item\label{it:selector_contained} if $(x_k,y_k)\in L$ then 
            $(x_k,y_k)+(1-t)(x_k,y_{k+1})\in\bigcup\tth$ for $t\in[0,1]$.
    \end{enumerate}
\end{proposition}
\begin{proof}
    First, note that if $(x_k,y_k)$ is of type $UR$ or $RU$ then 
        we have $(x_k,y_k)\in\inter A_{i,j}$.
    If $(x_k, y_k)$ is of type $D$
        we have $(x_k, y_k)=\ttA_{i,j}\cap\ttA_{i+1,j+1}$.
    In particular, in that case $x_k, y_k\in\{a_2, a_3,\ldots,a_{n}\}$.

    Sequence $L$ always begins $(x_0,y_0)=(0,0)$ and ends with $(x_K,y_K)=(1,1)$.
    By Proposition \ref{prop:01_cubes} we have $\ttA_{0,0}, \ttA_{n,n}\in\tth$.
    The above observations shows that for every $(x_k,y_k)$ with $k\in\{1,2\ldots,K-1\}$ we have $0<x_k,y_k<1$.
    Thus, two first and two last points of $L$ satisfies \ref{it:selector_monotonicity}.
    If $(x_1, y_1)$ is of type $UR$ then it follows that $\ttV$ begins with a sequence of $U$'s.
    We have $(x_1, y_1)\in\inter\ttA_{0, j}$ for some $j>0$.
    By Proposition \ref{prop:cub_h_properties}\ref{it:cub_h_property_1j} all $\ttA_{0,j'}\in\tth$ for $0\leq j'\leq j$.
    Hence, the interval spanned by $(x_0,y_0)$ and $(x_1,y_1)$ is contained in $\bigcup\tth$ proving \ref{it:selector_contained}.
    The cases when $(x_1, y_1)$ is of type $RU$ or $D$ as well as analysis of points $(x_{n-1}, y_{n-1})$ and $(x_{n}, y_{n})$ follows by similar argument.

    Let $(x,y)$ and $(x',y')$ be two consecutive points of $L$.
    Suppose that $(x,y)$ is of type $UR$ and $(x',y')$ of type $RU$.
    This situation arises when two symbols $U$ are separated by a positive number of symbols $R$ (see Figure~\ref{fig:selector_cases} top).
    It follows that
    \begin{align*}
        (x,y) &= (p\,a_i + (1-p)\,a_{i+1}, (1-p)\,a_j + p\,a_{j+1})\inter \ttA_{i,j},\\
        (x',y') &= ((1-p)\,a_{i'} + p\,a_{i'+1}, p\,a_{j} + (1-p)\,a_{j+1})\inter \ttA_{i',j},
    \end{align*}
    where $i < i'$.
    Thus,
    \begin{align*}
        y' - y &= pa_{j} + (1-p)a_{j+1} - ((1-p)a_j + pa_{j+1}) \\
            &= 2p a_j + a_{j+1} - a_j > a_{j+1} - a_j > 0.
    \end{align*}
    Consequently, we get $x<x'$ and $y<y'$ proving \ref{it:selector_monotonicity} for this case.
    By Proposition \ref{prop:cub_h_properties}\ref{it:cub_h_property_1j} we get that $\ttA_{i'',j}\in\tth$ for all $i\leq i''\leq i'$.
    Hence, the interval spanned by $(x,y)$ and $(x',y')$ is contained in $\bigcup\tth$ proving \ref{it:selector_contained}.

    The case when $(x,y)$ is of type $RU$ and $(x',y')$ of type $UR$ is analogous.

    Suppose that $(x,y)$ is of type $UR$ and $(x',y')$ of type $D$.
    This situation arises when symbols $U$ and $D$ are separated by a positive number of symbols $R$ (see Figure~\ref{fig:selector_cases} middle).
    It follows that
    \begin{align*}
        (x,y) &= (p\,a_i + (1-p)\,a_{i+1}, (1-p)\,a_j + p\,a_{j+1}) \in\inter\ttA_{i,j},\\
        (x',y') &= (a_{i'+1}, a_{j+1}) = \ttA_{i',j}\cap\ttA_{i'+1,j+1},
    \end{align*}
    where $i < i'$.
    It follows that $x<x'$ and $y<y'$ proving \ref{it:selector_monotonicity} for this case.
    Again, by Proposition \ref{prop:cub_h_properties}\ref{it:cub_h_property_1j} we can prove \ref{it:selector_contained}.
    
    The case when $(x,y)$ is of type $RU$ and $(x',y')$ of type $D$ is analogous.
    
    Now, suppose that $(x,y)$ is of type $D$ and $(x',y')$ of type $UR$.
    This situation arises when symbols $D$ and $R$ are separated by a positive number of symbols $U$ (see Figure~\ref{fig:selector_cases} bottom).
    It follows that
    \begin{align*}
        (x,y) &= (a_{i+1}, a_{j+1}) = \ttA_{i,j}\cap\ttA_{i+1,j+1},\\
        (x',y') &= (p\,a_{i+1} + (1-p)\,a_{i+2}, (1-p)\,a_{j'} + p\,a_{j'+1}) \in\inter\ttA_{i+1,{j'+1}},
    \end{align*}
    where $j < j'$.
    It follows that $x<x'$ and $y<y'$ proving \ref{it:selector_monotonicity} for this case.
    By Proposition \ref{prop:cub_h_properties}\ref{it:cub_h_property_1i} follows property \ref{it:selector_contained}.
    
    The case when $(x,y)$ is of type $D$ and $(x',y')$ of type $RU$ is analogous.
    
    Finally, if both $(x,y)$ and $(x',y')$ are of type $D$ it follows that
    \begin{align*}
        (x,y) &= (a_{i+1}, a_{j+1}) = \ttA_{i,j}\cap\ttA_{i+1,j+1},\\
        (x',y') &= (a_{i+2}, a_{j+2}) = \ttA_{i+1,j+1}\cap\ttA_{i+2,j+2}.
    \end{align*}
    Thus, $(x,y)$ and $(x',y')$ are the opposite corners of cube $\ttA_{i+1,j+1}$ which immediately gives both properties \ref{it:selector_monotonicity} and \ref{it:selector_contained}.
\end{proof}

\begin{figure}
    \centering
    \includegraphics[scale=0.26]{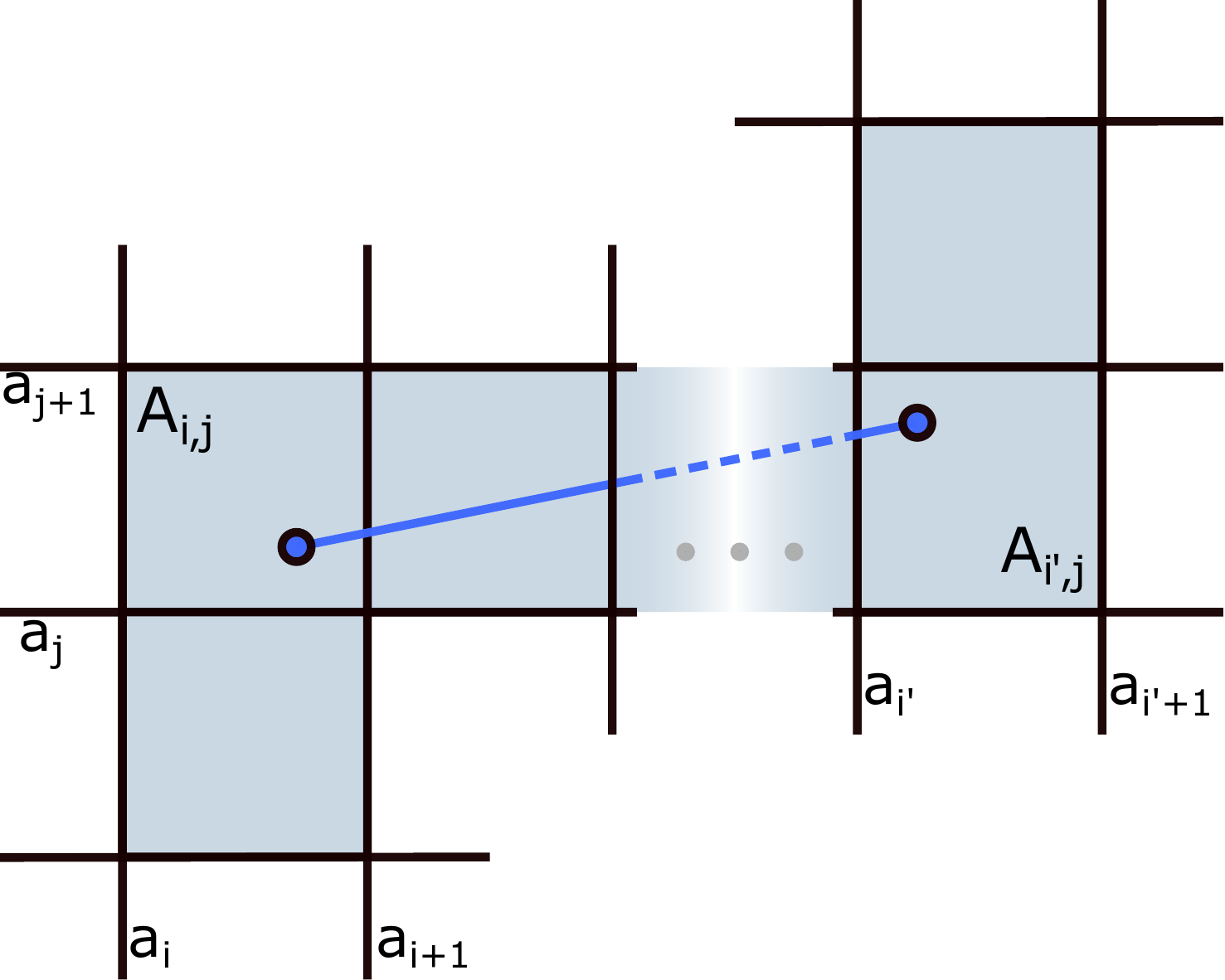}

    \hrule
    \includegraphics[scale=0.26]{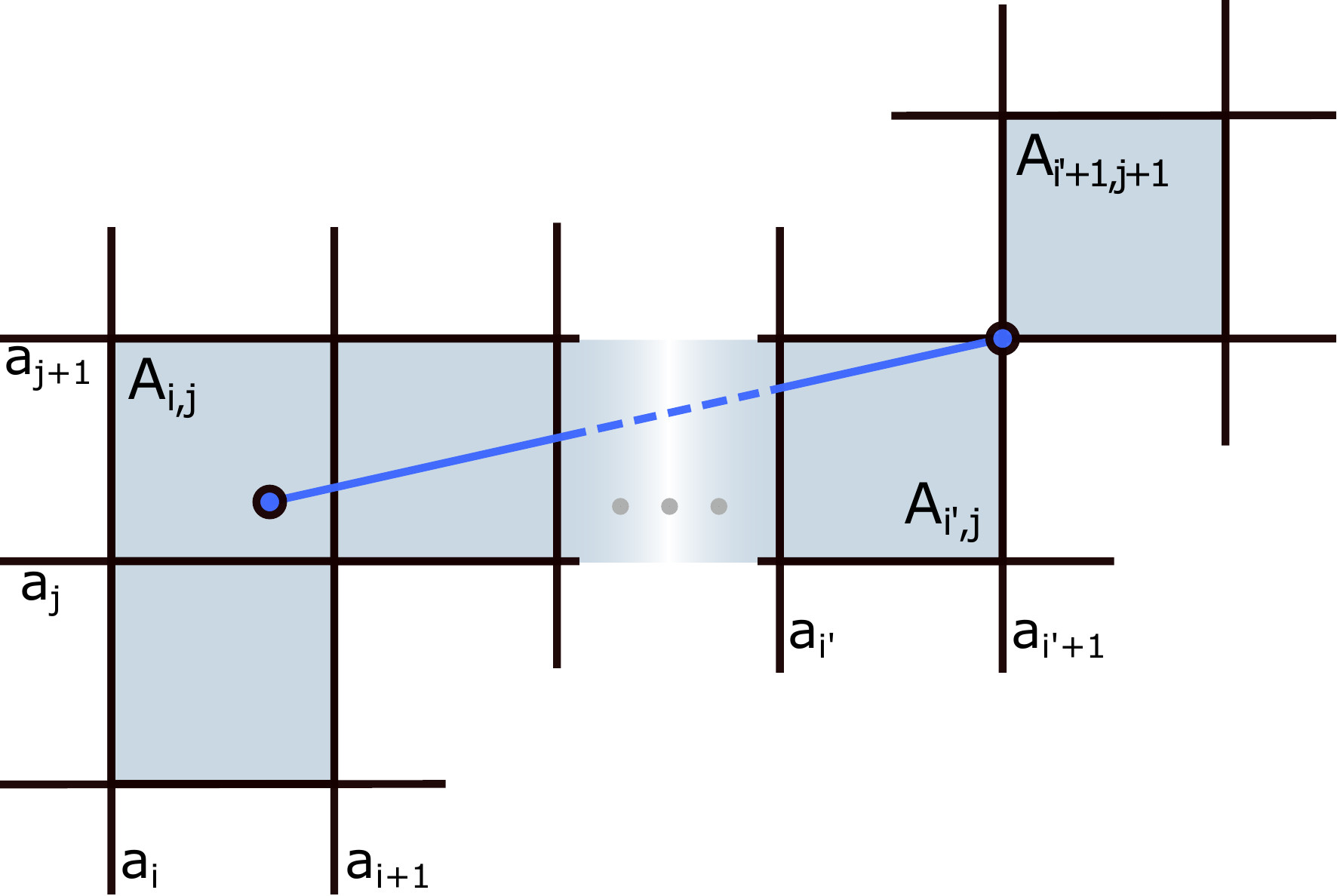}
    
    \hrule
    \includegraphics[scale=0.26]{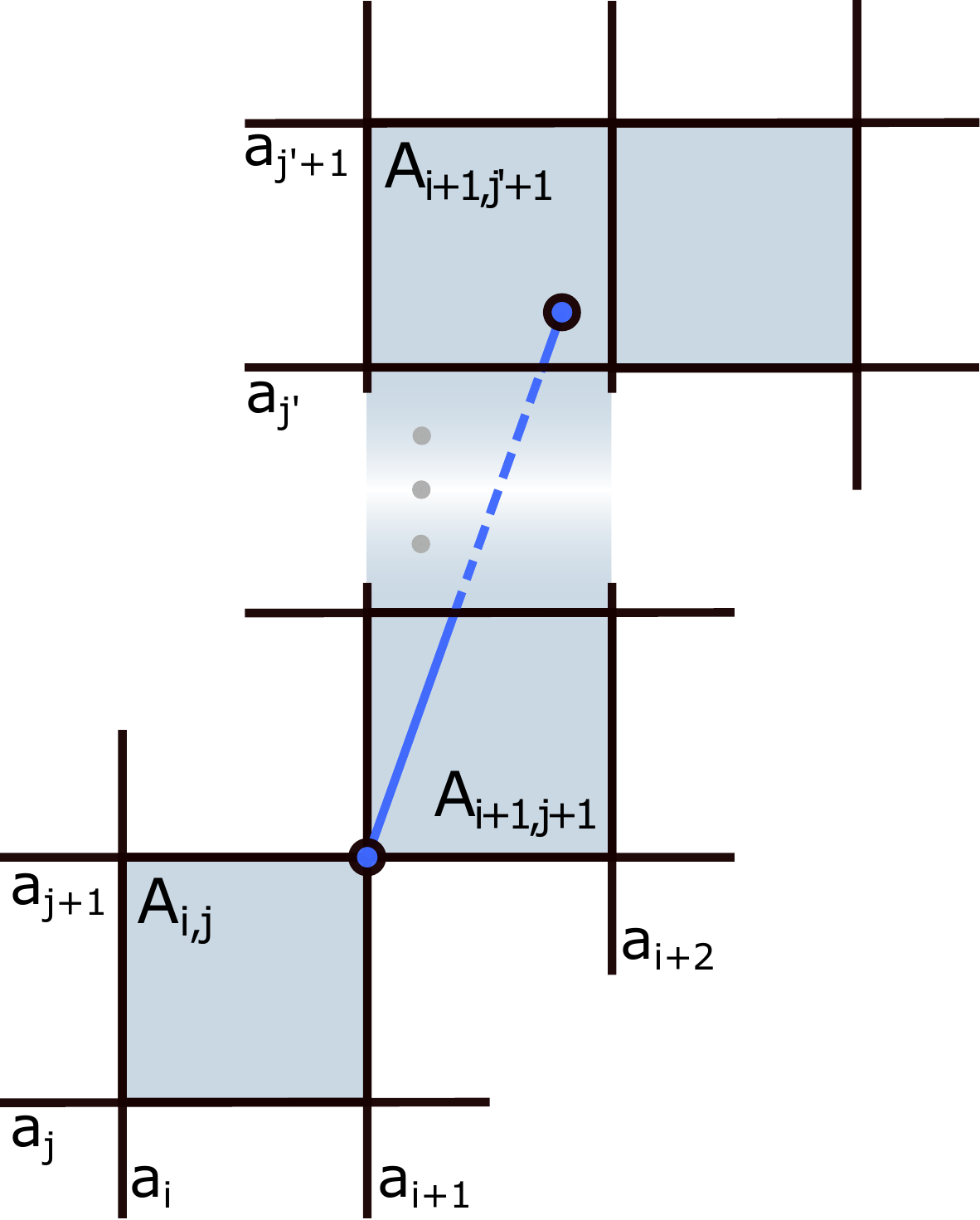}
    \caption{Segments of the piecewise linear homeomorphism being a selector constructed by Algorithm \ref{alg:selector} for a certain cubical homeomorphism $\tth$.
    Three panels corresponds to cases when:
        point of type $UR$ is followed by point of type $UR$ (top),
        point of type $UR$ is followed by point of type $D$ (middle),
        point of type $D$ is followed by point of type $UR$ (bottom).
    }
    \label{fig:selector_cases}
\end{figure}

By counting all possible vectors of symbols $\{R,U,D\}$ satisfying \eqref{eq:cub_homeo_vec_conditions} we obtain an exact size of family $\ttH$. 
In case of $n_D=0$, the vector has size $2(n-1)$ and, therefore, we get ${{2(n-1)}\choose (n-1)}$ ways of ordering symbols $R$ and $U$.
If $n_D=1$ then the vector size is $2(n-1)-1=2n-3$ and $n_R=n-2$.
Hence, we have ${{2n-3}\choose{n-2}}$ choices of slots for symbols R and we have choose a place for $D$ symbol among the remaining $2n-3-(n-2)=n-1$ slots.
Thus, the total number of ordering for $n_D=1$ is ${{2n-3}\choose{n-2}} (n-1)$.
In the general case, we get ${{2n-2-n_D}\choose{n-1-n_D}}{{n-1}\choose {n_D}}$.
Finally, the total number of vectors of $\ttH$ is given by the following formula:
\begin{equation*}\label{eq:size_of_H}
    \sum_{n_D=0}^{n-1} {{2n-2-n_D}\choose{n-1-n_D}}{{n-1}\choose {n_D}}.
\end{equation*}

\bibliographystyle{siamplain}
\bibliography{refs}

\end{document}

%% file: conjtest_shared.tex
\usepackage{lipsum}
\usepackage{amsfonts}
\usepackage{graphicx}
\usepackage{epstopdf}
\usepackage{algorithmic}
\ifpdf
  \DeclareGraphicsExtensions{.eps,.pdf,.png,.jpg}
\else
  \DeclareGraphicsExtensions{.eps}
\fi

\usepackage{enumitem}
\setlist[enumerate]{leftmargin=.5in}
\setlist[itemize]{leftmargin=.5in}

\newsiamremark{thm}{Theorem}
\newsiamremark{prop}{Proposition}
\newsiamremark{remark}{Remark}
\newsiamremark{hypothesis}{Hypothesis}
\crefname{hypothesis}{Hypothesis}{Hypotheses}
\newsiamthm{claim}{Claim}

\headers{Testing topological conjugacy of time series}{P. D\L{}otko, M. Lipi\'nski, and J. Signerska-Rynkowska}

\title{Testing topological conjugacy of time series
	\thanks{
    \funding{
        The authors gratefully acknowledge the support of Dioscuri program initiated by the Max Planck Society, jointly managed with the National Science Centre (Poland), and mutually funded by the Polish Ministry of Science and Higher Education and the German Federal Ministry of Education and Research. J S-R was also supported by National Science Centre (Poland) grant 2019/35/D/ST1/02253. 
    }}
}

\author{
    Pawe\l{} D\l{}otko
    \thanks{Dioscuri Centre in TDA, Institute of Mathematics, Polish Academy of Sciences, ul. \'Sniadeckich~8, 00-656 Warszawa, Poland,
    (\email{pdlotko@impan.pl}, \email{michal.lipinski@impan.pl}, \email{j.signerska@impan.pl})
    (\url{https://dioscuri-tda.org/members/pawel.html}, \url{https://sites.google.com/view/michal-lipinski}, \url{https://dioscuri-tda.org/members/justyna.html}).}
    \and 
    Micha\l{} Lipi\'nski
    \footnotemark[2] 
    \thanks{
        Institute of Computer Science and Computational Mathematics, Jagiellonian University, ul. \L{}ojasiewicza~6, 30-348 Krak\'ow, Poland
    }
    \and 
    Justyna Signerska-Rynkowska
    \footnotemark[2]
    \thanks{
        Institute of Applied Mathematics, Gda\'nsk University of Technology, ul. Narutowicza 11/12, 80-233 Gda\'nsk, Poland
    }
}

\usepackage{amsopn}

%% file: conjtest.bbl
\begin{thebibliography}{10}

\bibitem{correlation}
{\sc A.~Albano, A.~Passamante, and M.~E. Farrell}, {\em Using higher-order
  correlations to define an embedding window}, Physica D: Nonlinear Phenomena,
  54 (1991), pp.~85--97, \url{https://doi.org/10.1016/0167-2789(91)90110-U}.

\bibitem{gp}
{\sc A.~M. Albano, J.~Muench, C.~Schwartz, A.~I. Mees, and P.~E. Rapp}, {\em
  Singular-value decomposition and the {G}rassberger-{P}rocaccia algorithm},
  Phys. Rev. A, 38 (1988), pp.~3017--3026,
  \url{https://link.aps.org/doi/10.1103/PhysRevA.38.3017}.

\bibitem{Baranski_2020}
{\sc K.~Bara\'nski, Y.~Gutman, and A.~\'Spiewak}, {\em A probabilistic {T}akens
  theorem}, Nonlinearity, 33 (2020), p.~4940,
  \url{https://dx.doi.org/10.1088/1361-6544/ab8fb8}.

\bibitem{Baranski_2022}
{\sc K.~Bara\'nski, Y.~Gutman, and A.~\'Spiewak}, {\em On the
  {S}hroer–{S}auer–{O}tt–{Y}orke predictability conjecture for time-delay
  embeddings}, Commun. Math. Phys., 391 (2022), pp.~609--641,
  \url{https://doi.org/10.1007/s00220-022-04323-y}.

\bibitem{bollt2010}
{\sc E.~M. Bollt and J.~D. Skufca}, {\em On comparing dynamical systems by
  defective conjugacy: A symbolic dynamics interpretation of commuter
  functions}, Physica D: Nonlinear Phenomena, 239 (2010), pp.~579--590,
  \url{https://doi.org/https://doi.org/10.1016/j.physd.2009.12.007}.

\bibitem{Bramburger}
{\sc J.~J. Bramburger, S.~L. Brunton, and J.~N. Kutz}, {\em Deep learning of
  conjugate mappings.}, Physica D: Nonlinear Phenomena, 427 (2021), p.~133008.

\bibitem{Broer2011}
{\sc H.~Broer and F.~Takens}, {\em Reconstruction and time series analysis},
  Springer New York, New York, NY, 2011, pp.~205--242,
  \url{https://doi.org/10.1007/978-1-4419-6870-8_6}.

\bibitem{buzug}
{\sc T.~Buzug and G.~Pfister}, {\em Comparison of algorithms calculating
  optimal embedding parameters for delay time coordinates}, Physica D:
  Nonlinear Phenomena, 58 (1992), pp.~127--137,
  \url{https://www.sciencedirect.com/science/article/pii/016727899290104U}.

\bibitem{doi:10.1073/pnas.1906995116}
{\sc K.~Champion, B.~Lusch, J.~N. Kutz, and S.~L. Brunton}, {\em Data-driven
  discovery of coordinates and governing equations}, Proceedings of the
  National Academy of Sciences, 116 (2019), pp.~22445--22451.

\bibitem{DeAngelis:2015te}
{\sc D.~L. DeAngelis and S.~Yurek}, {\em Equation-free modeling unravels the
  behavior of complex ecological systems.}, Proc Natl Acad Sci U S A, 112
  (2015), pp.~3856--3857.

\bibitem{Bradley}
{\sc V.~Deshmukh, E.~Bradley, J.~Garland, and J.~D. Meiss}, {\em Using
  curvature to select the time lag for delay reconstruction}, Chaos: An
  Interdisciplinary Journal of Nonlinear Science, 30 (2020), p.~063143,
  \url{https://doi.org/10.1063/5.0005890}.

\bibitem{mutual}
{\sc A.~M. Fraser and H.~L. Swinney}, {\em Independent coordinates for strange
  attractors from mutual information}, Phys. Rev. A, 33 (1986), pp.~1134--1140,
  \url{https://link.aps.org/doi/10.1103/PhysRevA.33.1134}.

\bibitem{HeggerKantzFNN1999}
{\sc R.~Hegger and H.~Kantz}, {\em Improved false nearest neighbor method to
  detect determinism in time series data}, Phys. Rev. E, 60 (1999),
  pp.~4970--4973, \url{https://link.aps.org/doi/10.1103/PhysRevE.60.4970}.

\bibitem{kantz_schreiber_2003}
{\sc H.~Kantz and T.~Schreiber}, {\em Nonlinear Time Series Analysis},
  Cambridge University Press, 2~ed., 2003,
  \url{https://doi.org/10.1017/CBO9780511755798}.

\bibitem{Kennel1992}
{\sc M.~B. Kennel, R.~Brown, and H.~D.~I. Abarbanel}, {\em Determining
  embedding dimension for phase-space reconstruction using a geometrical
  construction}, Phys. Rev. A, 45 (1992), pp.~3403--3411,
  \url{https://link.aps.org/doi/10.1103/PhysRevA.45.3403}.

\bibitem{KIANIMAJD201711005}
{\sc A.~Kianimajd, M.~Ruano, P.~Carvalho, J.~Henriques, T.~Rocha, S.~Paredes,
  and A.~Ruano}, {\em Comparison of different methods of measuring similarity
  in physiologic time series}, IFAC-PapersOnLine, 50 (2017), pp.~11005--11010,
  \url{https://www.sciencedirect.com/science/article/pii/S2405896317333967}.
\newblock 20th IFAC World Congress.

\bibitem{kim}
{\sc H.~Kim, R.~Eykholt, and J.~Salas}, {\em Nonlinear dynamics, delay times,
  and embedding windows}, Physica D: Nonlinear Phenomena, 127 (1999),
  pp.~48--60,
  \url{https://www.sciencedirect.com/science/article/pii/S0167278998002401}.

\bibitem{Lejarza:2022vl}
{\sc F.~Lejarza and M.~Baldea}, {\em Data-driven discovery of the governing
  equations of dynamical systems via moving horizon optimization}, Scientific
  Reports, 12 (2022), p.~11836.

\bibitem{garcia}
{\sc M.~Matilla-García, I.~Morales, J.~M. Rodríguez, and M.~Ruiz~Marín},
  {\em Selection of embedding dimension and delay time in phase space
  reconstruction via symbolic dynamics}, Entropy, 23 (2021),
  \url{https://doi.org/10.3390/e23020221},
  \url{https://www.mdpi.com/1099-4300/23/2/221}.

\bibitem{PhysRevLett.82.1144}
{\sc K.~Mischaikow, M.~Mrozek, J.~Reiss, and A.~Szymczak}, {\em Construction of
  symbolic dynamics from experimental time series}, Phys. Rev. Lett., 82
  (1999), pp.~1144--1147.

\bibitem{moniz2004}
{\sc L.~Moniz, L.~Pecora, J.~Nichols, M.~Todd, and J.~R. Wait}, {\em Dynamical
  assessment of structural damage using the continuity statistic}, Structural
  Health Monitoring, 3 (2004), pp.~199--212,
  \url{https://doi.org/10.1177/1475921704042681}.

\bibitem{LMoniz_nichols2005}
{\sc J.~Nichols, L.~Moniz, J.~Nichols, L.~Pecora, and E.~Cooch}, {\em Assessing
  spatial coupling in complex population dynamics using mutual prediction and
  continuity statistics}, Theoretical Population Biology, 67 (2005), pp.~9--21,
  \url{https://doi.org/https://doi.org/10.1016/j.tpb.2004.08.004}.

\bibitem{pecora1995}
{\sc L.~M. Pecora, T.~L. Carroll, and J.~F. Heagy}, {\em Statistics for
  mathematical properties of maps between time series embeddings}, Phys. Rev.
  E, 52 (1995), pp.~3420--3439, \url{https://doi.org/10.1103/PhysRevE.52.3420}.

\bibitem{pecora}
{\sc L.~M. Pecora, L.~Moniz, J.~Nichols, and T.~L. Carroll}, {\em A unified
  approach to attractor reconstruction}, Chaos: An Interdisciplinary Journal of
  Nonlinear Science, 17 (2007), p.~013110,
  \url{https://doi.org/10.1063/1.2430294},
  \url{https://arxiv.org/abs/https://doi.org/10.1063/1.2430294}.

\bibitem{Embedology}
{\sc T.~Sauer, J.~A. Yorke, and M.~Casdagli}, {\em Embedology}, Journal of
  Statistical Physics, 65 (1991), pp.~579--616.

\bibitem{skufca2007}
{\sc J.~D. Skufca and E.~M. Bollt}, {\em Relaxing conjugacy to fit modeling in
  dynamical systems}, Phys. Rev. E, 76 (2007), p.~026220,
  \url{https://doi.org/10.1103/PhysRevE.76.026220}.

\bibitem{skufca2008}
{\sc J.~D. Skufca and E.~M. Bollt}, {\em {A concept of homeomorphic defect for
  defining mostly conjugate dynamical systems}}, Chaos: An Interdisciplinary
  Journal of Nonlinear Science, 18 (2008), p.~013118,
  \url{https://doi.org/10.1063/1.2837397}.

\bibitem{szlenk}
{\sc W.~Szlenk}, {\em An Introduction to the Theory of Smooth Dynamical
  Systems}, Pa\'nstwowe Wydawnictwo Naukowe (PWN), 1984.

\bibitem{TakensOriginal}
{\sc F.~Takens}, {\em Detecting strange attractors in turbulence}, in Dynamical
  Systems and Turbulence, Warwick 1980, D.~Rand and L.-S. Young, eds., Berlin,
  Heidelberg, 1981, Springer Berlin Heidelberg, pp.~366--381.

\bibitem{Viswanath2004}
{\sc D.~Viswanath}, {\em The fractal property of the {L}orenz attractor},
  Physica D: Nonlinear Phenomena, 190 (2004), pp.~115--128,
  \url{https://www.sciencedirect.com/science/article/pii/S0167278903004093}.

\bibitem{Yuan:2019vx}
{\sc Y.~Yuan, X.~Tang, W.~Zhou, W.~Pan, X.~Li, H.-T. Zhang, H.~Ding, and
  J.~Goncalves}, {\em Data driven discovery of cyber physical systems}, Nature
  Communications, 10 (2019), p.~4894.

\bibitem{Zheng09}
{\sc J.~Zheng, E.~Bollt, and J.~Skufca}, {\em Regularity of commuter functions
  for homeomorphic defect measure in dynamical systems model comparison},
  Dynamics of Continuous, Discrete \& Impulsive Systems. Series A: Mathematical
  Analysis, 18 (2011).

\bibitem{zheng2013}
{\sc J.~Zheng, J.~D. Skufca, and E.~M. Bollt}, {\em Comparing dynamical systems
  by a graph matching method}, Physica D: Nonlinear Phenomena, 255 (2013),
  pp.~12--21,
  \url{https://doi.org/https://doi.org/10.1016/j.physd.2013.03.012}.

\end{thebibliography}
